\documentclass[12pt]{article}
\usepackage{amsmath, amssymb,amsfonts,bbm,verbatim,multirow,color,xcolor}
\usepackage{epic,eepic,psfrag,epsfig}
\usepackage[sf,bf,SF,footnotesize]{subfigure}
\usepackage{graphicx}
\usepackage{amsthm,breakcites}
\usepackage{amscd}
\usepackage{epsfig}
\usepackage{fullpage}

\usepackage{natbib}[round]
\setcitestyle{aysep={,}}

\usepackage{setspace}
\usepackage{enumerate}
\usepackage{array}
\usepackage[small]{caption}
\RequirePackage[colorlinks,citecolor=blue,urlcolor=blue,breaklinks]{hyperref}

\newtheorem{thm}{Theorem}

\newtheorem{lemma}[thm]{Lemma}

\newtheorem{prop}[thm]{Proposition}
\newtheorem{cor}[thm]{Corollary}
\newtheorem{eg}[thm]{Example}

\newcommand{\red}[1]{\textcolor{red}{#1}}
\newcommand{\blue}[1]{\textcolor{blue}{#1}}

\newcommand{\tb}[1]{\textcolor{blue}{(TB: #1)}}

\DeclareMathOperator*{\Var}{Var}
\DeclareMathOperator*{\Cov}{Cov}
\def\hat{\widehat}

\allowdisplaybreaks

\newcommand{\bbP}{\mathbb{P}}
\newcommand{\bbS}{\mathbb{S}}
\newcommand{\bbT}{\mathbb{T}}
\newcommand{\bbN}{\mathbb{N}}
\newcommand{\bbR}{\mathbb{R}}
\newcommand{\bbE}{\mathbb{E}}
\newcommand{\bS}{\mathbf{S}}
\newcommand{\bi}{\boldsymbol{i}}

\title{Efficient estimation with incomplete data via generalised ANOVA decompositions}
\date{}
\author{Thomas B. Berrett \\ Department of Statistics, University of Warwick}

\begin{document}

\maketitle

\begin{abstract}
We study the semiparametric efficient estimation of a class of linear functionals in settings where a complete multivariate dataset is supplemented by additional datasets recording subsets of the variables of interest. These datasets are allowed to have a general, in particular non-monotonic, structure. Our main contribution is to characterise the asymptotic minimal mean squared error for these problems and to introduce an estimator whose risk approximately matches this lower bound. We show that the efficient rescaled variance can be expressed as the minimal value of a quadratic optimisation problem over a function space, thus establishing a fundamental link between these estimation problems and the theory of generalised ANOVA decompositions. Our estimation procedure uses iterated nonparametric regression to mimic an approximate influence function derived through gradient descent. We prove that this estimator is approximately normally distributed, provide an estimator of its variance and thus develop confidence intervals of asymptotically minimal width. Finally we present extensions of our theory demonstrating that the framework can be adapted to include various types of sampling bias and non-linear functionals.

\end{abstract}

\section{Introduction}

Large modern datasets are increasingly often compiled from smaller datasets recording different variables or collected under different experimental conditions. This feature of modern data science is reflected in the vast range of statistical and machine learning literature tackling issues of incompleteness, data heterogeneity and sampling bias, including topics such as semi-supervised learning~\citep{chakrabortty2018efficient,kim2024semi}, data fusion~\citep{li2023efficient,qiu2023efficient}, missing data~\citep{robins2017minimax} and transfer learning~\citep{cai2021transfer}. Areas of application are diverse and cover genomics~\citep{angelopoulos2023prediction}, climate science~\citep{lucas2015bivariate} and causal inference~\citep{bareinboim2016causal}. Particularly when clean data is scarce, the information provided by additional data is invaluable. 

Relevant literature commonly seeks to derive optimal estimators given various structures of datasets, with statistical efficiency~\citep[in the sense of,~e.g.][]{van2000asymptotic} being one of the most widespread notions of optimality. The simplest structures are those in which a complete dataset is supplemented by an additional dataset recording a subset of the variables of interest. These structures arise in semi-supervised settings, where we aim to solve a supervised learning problem with the aid of additional unlabelled data \citep{chakrabortty2018efficient,cannings2020local,kim2024semi}. They also arise in missing data settings, where we aim to carry out a regression analysis when covariates are fully observed and response variables may be missing \citep{robins2017minimax}. In causal inference we may combine a large dataset with unmeasured confounding variables with a smaller dataset recording information on these confounders \citep{yang2019combining}. %In such literature the covariate distribution is often allowed to differ according to whether or not the response is observed, with the Missing At Random assumption~\citep{rubin1976inference} being common. 
It is shown in many of these settings that estimators properly incorporating incomplete data have smaller risk than those based solely on the complete data and can sometimes be proved to attain efficiency lower bounds.

Often one would like to combine more than two datasets, though the development of optimal estimators in these settings is typically more complex. In recent examples of work on this problem, \cite{li2023efficient} and \cite{qiu2023efficient} study efficiency in settings in which a complete dataset is supplemented by multiple incomplete datasets having a nested structure, while allowing for distributional shifts between datasets. A key example giving rise to these structures is that of a longitudinal study, where multiple response variables are recorded over time on the same data subjects who may drop out of the study at various points. In such cases data is available on subsets $S_1,\ldots,S_m$ of all variables, with $S_1 \subseteq S_2 \subseteq \ldots \subseteq S_m$. This monotonic structure enables the derivation of explicit expressions for efficient influence functions and thus the construction of efficient estimators.

However, non-monotonic structures are common in practice and require new techniques. Healthcare datasets such as that collected by the Alzheimer's Disease Neuroimaging Initiative (ADNI) \citep{mueller2005alzheimer} and MIMIC-III \citep{johnson2016mimic} often exhibit non-monotonic missingness structures in settings where patients may undergo several diagnostic procedures. Certain readings may be refused due to their expense or invasive nature and others may be discarded as inaccurate, leading to complex datasets where different groups of patients provide valuable and complementary information on targets of inference. Data of this type are studied in literature on \emph{block-wise missingness} \citep{xue2021integrating}, where work on efficient estimation has been developed in parametric models. Multi-phase sampling designs in causal inference can also lead to the combination of multiple datasets with non-monotonic structures \citep{clayton1998analysis,yang2019combining}. General missing data settings are non-monotonic, though the analysis of such data and the modelling of missingness mechanisms is significantly more challenging when this is the case~\citep{robins1997non,tsiatis2006semiparametric,sun2018inverse}. In such settings it is therefore useful to consider different types of assumptions controlling sampling biases.

\subsection{Formal setting and contributions}
\label{Sec:Contributions}

We will suppose that we have access to a complete\footnote{Throughout, when we refer to the `complete' data we mean this fully-observed dataset of size $n$, not the entire dataset of size $n+\sum_{S \in \bbS} n_S$.} $d$-dimensional dataset $X_1,\ldots,X_n \overset{\text{i.i.d.}}{\sim} f$
for some density function $f$ on $\bbR^d$ with $d \geq 2$. In addition, writing $[d]=\{1,\ldots,d\}$ and $2^{[d]}$ for the power set of $[d]$, for a collection $\bbS \subseteq 2^{[d]} \setminus \{[d]\}$ we will suppose that we have access to independent incomplete data $(X_{S,i} : S \in \bbS, i\in[n_S])$. Here $X_{S,i}$ takes values in $(\mathbb{R} \times \{\texttt{NA}\})^d$ such that the $j$th component of $X_{S,i}$ is $\texttt{NA}$ if and only if $j \in S^c= [d] \setminus S$. We may assume, without loss of generality, that the samples sizes $n_\bbS=(n_S:S\in\bbS)$ are positive. We will sometimes use the notation $\bbS^+ = \bbS \cup \{[d]\}$, $n_{[d]} = n$ and $X_{[d],i}=X_i$ so that our data can be summarised as $(X_{S,i} : S \in \bbS^+, i \in [n_S])$. 

The distribution of the incomplete data is assumed to be linked to the target density $f$ through a sequence of positive functions $(r_S : S \in \bbS)$ representing distributional shifts. Letting $f_S$ be the density of $(X_1)_S = (X_{1,j} : j \in S)$, that is the marginal density of $f$ associated to those variables in $S$, the density of $(X_{S,i})_S = (X_{S,i,j} : j \in S)$ is proportional to $r_Sf_S$ for each $S \in \bbS$. More formally, we assume that
\begin{equation}
\label{Eq:IncompleteData}
    (X_{S,1})_S,\ldots,(X_{S,n_S})_S \overset{\text{i.i.d.}}{\sim} \bar{r}_S f_S \quad \text{and} \quad X_{S,i,j} = \texttt{NA} \text{ almost surely for } j \in S^c,
\end{equation}
where we write $\bar{r}_S=r_S/\int fr_S$ for the normalised version of $r_S$. Taking all $r_S \equiv 1$ we recover the Missing Completely At Random (MCAR) case, but the additional flexibility allows us to specify the $r_S$ to account for distributional shifts; see the discussion of related work below.

Our aim is to use this data to estimate linear functionals of the form %~\citep[e.g.][]{li2023nonparametric} of the form
\begin{equation}
\label{Eq:Functional}
    \theta(f)= \int_{\mathbb{R}^d} a(x)f(x) \,dx = \bbE\{a(X)\},
\end{equation}
where $a : \mathbb{R}^d \rightarrow \mathbb{R}$ is a known measurable function and we let $X \sim f$. This canonical statistical problem includes the estimation of moments, probabilities of events and prediction risks. When $\int a^2 f < \infty$ and we only have access to the complete cases in the data, it is known that the natural estimator $\hat{\theta}^\mathrm{CC} = n^{-1} \sum_{i=1}^n a(X_i)$ is optimal in the local asymptotic minimax sense. The goal of this work is to use the incomplete data to outperform $\hat{\theta}^\mathrm{CC}$, in particular to characterise the minimal asymptotic squared risk for the estimation of $\theta(f)$ and develop estimators that attain this optimal risk and thus are statistically efficient. Our main contributions are summarised below.
\begin{itemize}
    \item In Section~\ref{Sec:LowerBound} we give a local asymptotic minimax lower bound (Theorem~\ref{Thm:ShiftedLowerBound}) for the estimation of $\theta(f)$ given the data $(X_{S,i} : S \in \bbS^+, i \in [n_S])$ introduced above, in the case that the shift functions $(r_S:S \in \bbS)$ are known. This bound can be interpreted as the minimal expected squared error in approximating the function $a(x)$ by sums $\sum_{S \in \mathbb{S}} \alpha_S(x_S)$ of functions of subsets of our variables and thus provides a fundamental link between this estimation problem and the theory of generalised ANOVA decompositions. We introduce and discuss such decompositions in Section~\ref{Sec:ANOVADecompositions}, state our lower bound in Section~\ref{Sec:ANOVALowerBound} and give examples in Section~\ref{Sec:ANOVAExamples}. 
    \item In Section~\ref{Sec:UpperBound} we develop an estimator whose rescaled risk, under natural regularity conditions, is close in a nonasymptotic sense to the lower bound in Theorem~\ref{Thm:ShiftedLowerBound}.  In Section~\ref{Sec:UpperBoundApprox} we take a gradient-descent approach to approximate the generalised ANOVA decomposition of $a(\cdot)$, resulting in an approximate influence function for our estimation problem. In Section~\ref{Sec:UpperBoundData} we introduce a data-driven estimator of $\theta(f)$ based on iterated nonparametric regression and bound its risk. In Section~\ref{Sec:UpperBoundConfidence} we prove the approximate normality of our estimator and show how its variance can be estimated to give confidence intervals for $\theta(f)$. Numerical results are given in Section~\ref{Sec:Simulations}.
    \item In Section~\ref{Sec:Extensions} we present extensions of our basic framework to incorporate more complex sampling biases. In Section~\ref{Sec:MAR} we work under a Missing At Random assumption, assuming that our data is first generated from the target distribution and is then subject to random missingness that may depend on certain covariates. In Section~\ref{Sec:UnknownShifts} we broaden our data model in a different direction, allowing the distributional shift functions $r_S$ to be unknown elements of a parametric model.
\end{itemize}
Additionally, in Appendix~\ref{Sec:DirectEstimator} we study an alternative estimator for the MCAR case that does not require cross-fitting, unlike the estimator in Section~\ref{Sec:UpperBound}. Appendix~\ref{Sec:MICE} presents additional numerical results, Appendix~\ref{Sec:NonLinear} discusses extensions to non-linear and two-sample functionals, and proofs and an auxiliary lemma are given in Appendix~\ref{Sec:Proofs}. Taken together, our results introduce efficient estimators for general structures of incomplete data and allow us to compare the effects of different types of sampling biases on estimation accuracy.

\subsection{Related work}

%Our basic task is the estimation of linear functionals of the form $f \mapsto \int af$, where $a$ is known and $f$ is unknown. This canonical statistical problem includes the estimation of moments, probabilities of events and prediction risks. In standard models for complete data the estimator $\hat{\theta}^\mathrm{CC}$ is typically optimal in the local asymptotic minimax sense; see, for example, \cite{ibragimov1991asymptotically} and the references therein for an introduction to the classical theory. However, estimation in more complex models can be much more involved and is often taken as a benchmark task with which we are able to provide insights into the model itself; see work in density models under nonparametric shape constraints~\citep{jankowski2014convergence}, nonparametric regression with one-sided errors~\citep{reiss2017efficient}, Bayesian density models~\citep{rivoirard2012bernstein}, PET imaging \citep{bickel1995estimating}, differential privacy~\citep{rohde2020geometrizing} and density models with fixed marginal distributions~\citep{bickel1991efficient}. The estimation of these simple functionals also provides a stepping stone towards the estimation of quadratic and smooth functionals \citep{bickel1988estimating,laurent1996efficient} and other quantities arising from nonparametric statistics and information theory \citep{berrett2019efficient} and causal inference \citep{chernozhukov2022locally}.

Our model is closely related to models for semi-supervised learning, particularly when $\bbS$ is a singleton. There is a vast literature on this topic so here we highlight references that focus on the estimation of mean functionals and statistical efficiency. 
%In this literature the aim is to use the information in additional incomplete data to construct estimators with smaller variance than standard estimators based only on the complete data, often of the form of $\hat{\theta}^\mathrm{CC}$. 
The most common setting is that one has access to a dataset of i.i.d.~copies of a pair $(X,Y)$, where $Y$ is thought of as a label and $X$ as a feature vector, as well as an unlabelled dataset recording $X$ values only. In this setting, \cite{zhang2019semi} construct estimators of $\mathbb{E}(Y)$ that have smaller variance than the natural estimator based only on the labelled data, and also construct estimators that are shown to be efficient under regularity conditions. \cite{yang2019combining} and \cite{cannings2022correlation} consider the estimation of more general functionals, constructing estimators that improve upon the basic estimator and which can be shown to be efficient. \cite{kim2024semi} use $U$-statistics to estimate functionals of the form $\mathbb{E}\{a(Y_1,\ldots,Y_r)\}$, where $Y_1,\ldots,Y_r$ are independent copies of the label $Y$, and provide local asymptotic minimax lower bounds to prove that their estimators are efficient. Beyond mean functionals, we note that there is also work on efficient parameter estimation in semiparametric models in semi-supervised settings \citep{chakrabortty2018efficient,azriel2022semi}. While the work above focuses on the common setting where $\mathbb{S}$ is a singleton, \cite{yang2019combining} and \cite{cannings2022correlation} also consider cases where $|\mathbb{S}|>1$, though their estimators are typically not efficient.

Models similar to semi-supervised settings arise in causal inference and missing data literature, where we may have a regression model with a partially-observed response $Y$ and fully-observed covariates $X$ and aim to estimate a mean response $\mathbb{E}(Y)$ or average treatment effect. %Here a key concern is to correct for the sampling bias induced by the missingness of $Y$ being allowed to depend on the covariates $X$. 
Assuming that $\pi(x)=\mathbb{P}(Y \text{ observed } | X=x)^{-1}$ is a H\"older smooth function of $x$ bounded away from zero and infinity, \cite{robins2017minimax} develop a theory of double robustness and efficiency for the estimation of $\mathbb{E}(Y)$. Work such as \cite{kallus2020role} and \cite{zhang2023double} weakens the assumption that $\pi$ is bounded to allow for the regime that the unlabelled dataset is much larger than the labelled dataset.

Statistical efficiency is a key concern in data fusion and missing data models. \cite{li2023efficient} study settings where the joint target distribution $P$ is decomposed into conditional distributions $P_{X_1}P_{X_2|X_1}\ldots P_{X_d|(X_1,\ldots,X_{d-1})}$ and it is assumed that additional datasets contain accurate information on some of these conditional distributions. In settings such as longitudinal studies with a monotonic missingness structure such that $\bbS \subseteq \{[1],[2],\ldots,[d-1]\}$, efficient influence functions are explicitly derived. Using similar decompositions of $P$ and distributional shift conditions, \cite{qiu2023efficient} develop multiple-robustness theory for the efficient estimation of functionals of the form $\mathbb{E}\{a(X)\}$. %, also providing finite-sample upper bounds and an investigation of the effects of misspecification of the distributional shift conditions. 
In the causal inference literature it is often of interest to combine the results of a randomised controlled trial with observational data in order to improve the efficiency of estimators; see \cite{bareinboim2016causal} for an introduction to the topic and \cite{colnet2024causal,lin2024data} for recent survey articles.
%\cite{luedtke2023one} studies one-step estimators more generally, where unknown parameters take values in Hilbert spaces. A contribution of the current work is to provide efficient estimators for general classes $\bbS \subseteq 2^{[d]}$.

Tackling a non-monotic setting, \cite{bickel1991efficient} give efficient estimators of mean functionals of bivariate densities whose marginal densities are known, corresponding in our notation to the $d=2$ setting with $\bbS=\{\{1\},\{2\}\}$ in the limit that $n_{\{1\}}=n_{\{2\}}=\infty$. Our interest is in finite sample sizes and general structures, though the consideration of $L_2$ projections is similarly crucial. Their estimators work by binning continuous data, while our methods are based on iterated nonparametric regression with kernels. Shortly after the first version of this work appeared on \texttt{arXiv}, the relevant work~\cite{graham2024towards} appeared. This was similarly motivated by the desire to extend existing work, such as~\cite{li2023efficient}, to settings with non-monotonic structure. The authors study general data fusion problems where conditional distributions in source populations align with those in the target population and provide tools that can be used to characterise sets of all influence functions. When our settings overlap, their results can be used to verify that the efficient influence functions we derive are indeed the efficient influence functions. 
%To use their notation, we let the $0$th source be our complete data and take $\mathcal{D}_1^{(0)}$ to be the set of all mean-zero elements of $L_2$ and we let the $S$th source be our data $X_{S,1},\ldots,X_{S,n_S}$ and take $\mathcal{D}_1^{(S)}=H_S$ to be the set of all mean-zero elements of $L_2$ depending only on the arguments in $S$. Then, writing $\alpha_\bbS^*$ for the minimiser of $\mathcal{L}(\cdot)$ and taking $h_n^{(Q)}=a-\theta-\sum_{S \in \bbS}\alpha_S^*$ we can check using our stationarity condition~\eqref{Eq:Stationarity} that the hypothesis of their Theorem~3(c) is satisfied. 
Our work complements this unifying theory by providing, in the settings we consider, more explicit derivations of efficient influence functions and the construction of efficient estimators. While our main interest here is in nonparametric models, %we briefly mention relevant work in other missing data models. \cite{robins1994estimation} \citep[cf.][]{yu2006revisit} consider efficient estimation in parametric regression models with partially-observed covariates, giving general results that can be made more explicit in special cases such those with a monotonic missingness structure. \cite{robins1995semiparametric} give efficient estimators for a different setting in which the aim is to fit a parametric model for the conditional distribution of $Y|X$, where there is a single entry of $X$ that may be missing while the other covariates and the response $Y$ are fully observed. 
we also mention recent work on block-wise missing data \citep{xue2021integrating,song2024semi,li2024adaptive} that fits parametric models given data with non-monotonic missingness structures.%, often making MCAR assumptions.

Our work also contributes to the literature on distributional shifts and the use of density ratios. In our basic model we assume that the distributional shifts $(r_S: S\in\bbS)$ are known, as in previous literature. While primarily concerned with MCAR data, prediction-powered inference extends to settings with known distributional shifts \citep{angelopoulos2023prediction}. In a transfer learning setting with \emph{covariate shift}, \cite{ma2023optimally} construct optimal nonparametric regression estimators using knowledge of the density ratio between the target and source covariate distributions. \cite{tibshirani2019conformal} studies conformal prediction under covariate shift, similarly using knowledge of a density ratio. \cite{thams2023statistical} considers a range of hypothesis testing problems where data is drawn from a source distribution being a known shift of the target distribution, including off-policy testing for bandit problems, model selection and the detection of heterogeneity in causal discovery. Here it is common that shifts only affect a subset of the variables, whose distributions are well understood from previous data. Parametric families of shifts, such as those in Section~\ref{Sec:UnknownShifts}, are studied in statistical work such as~\cite{efron1996using,qin1998inferences,qiu2023efficient}. Density ratios themselves can be estimated directly; see~\cite{sugiyama2012density} for an introduction to this topic.

Our results show that the efficient influence functions for the problems we consider are found by minimising quadratic objective functions. In limiting cases with incomplete datasets of infinite size this minimisation problem coincides with the problem of finding \emph{generalised ANOVA decompositions} \citep[e.g.][]{stone1994use}. Such decompositions extend classical ANOVA decompositions \citep{hoeffding1948class,efron1981jackknife} beyond the setting in which the components of $X$ are independent. Work such as \citet{stone1994use} and~\citet{huang1998projection} has explored the statistical use of generalised ANOVA decompositions in problems such as nonparametric regression, where the aim is to approximate an unknown regression function by sums of simpler components. Such decompositions are also used in sensitivity analysis~\citep{chastaing2012generalized,rahman2014generalized}, where the aim is to identify subsets of variables that make a negligible contribution to an outcome variable. Classical ANOVA decompositions have also recently been used as the basis for modelling assumptions in missing data problems~\citep{sell2023nonparametric}. \cite{hooker2007generalized}, \cite{li2012general} and~\cite{rahman2014generalized} further study theoretical properties of generalised ANOVA decompositions and introduce numerical schemes for their computation.

\subsection{Notation}

Here we collect basic, commonly-used notation. We write $\bbN$ for the set of natural numbers and $\bbN_0=\bbN \cup \{0\}$. For $d \in \mathbb{N}$ write $[d]=\{1,\ldots,d\}$ and write $2^{[d]}$ for the collection of all subsets of $[d]$. Given $S \in 2^{[d]}$ and $E \subseteq \bbR$ we write $E^S = \prod_{j \in S} E$ for the marginal space of $E^d$ corresponding to those variables in $S$. For $u \in \mathbb{R}^d$ we write $\|u\|_\infty = \max_{j \in [d]} |u_j|$ and $\|u\|_1 = \sum_{j=1}^d |u_j|$, and for a function $g:\mathbb{R}^d \rightarrow \mathbb{R}$ we write $\|g\|_\infty = \sup_{x \in \mathbb{R}^d} |g(x)|$. For $x \in (\bbR \cup \{\texttt{NA}\})^d$ and $S \in 2^{[d]}$ we write $x_S \equiv (x)_S = (x_j : j \in S)$ for those entries of $x$ corresponding to the variables in $S$. For $n \in \mathbb{N}$ and $m \in [n]$ we write $(n)_m=n(n-1)\ldots(n-m+1)$ for the falling factorial and $\mathcal{I}_m=\{(i_1,\ldots,i_{m}) \in [n]^m : i_1,\ldots,i_m \text{ are distinct}\}$, so that $|\mathcal{I}_m|=(n)_m$. For $a,b \in \mathbb{R}$ we write $a \wedge b = \min(a,b)$ and $a \vee b = \max(a,b)$. For distributions $P$ and $Q$ with densities $f$ and $g$ we write $d_\mathrm{TV}(P,Q)=\frac{1}{2} \int |f-g|$ for the Total Variation distance.

\section{Problem formulation and main result}

Recall from Section~\ref{Sec:Contributions} that for some density $f$ on $\bbR^d$ we suppose that we have access to a complete dataset $X_1,\ldots,X_n \overset{\text{i.i.d.}}{\sim} f$ as well as independent incomplete data $(X_{S,i} : S \in \bbS, i\in[n_S])$ whose distribution depends on known distributional shifts $(r_S : S \in \bbS)$ and is given in~\eqref{Eq:IncompleteData}. We aim to estimate the functionals defined in~\eqref{Eq:Functional} for a known function $a : \bbR^d \rightarrow \bbR$.

Before stating our main result we gather our assumptions. We will suppose that $\|a\|_\infty < \infty$ and that there exist constants $0<c \leq C < \infty$ such that $c \leq \bar{r}_S(x_S) \leq C$ for all $x_S \in \bbR^S$ and $S \in \bbS$. For fixed parameters $\beta_1,\beta_2,\beta_3,L_1,L_2,L_3>0$ we assume the following.
\begin{enumerate}[align=left]
    \item[(A1)($\beta_1,L_1$)] For all $x,x' \in \bbR^d$ with $\|x-x'\|_\infty \leq 1$ we have
    \begin{equation}
    \label{Eq:A1}
        |a(x)-a(x')| \leq L_1 \|x-x'\|_\infty^{\beta_1}.
    \end{equation}
    \item[(A2)($\beta_2,L_2$)] For all $S \in \bbS$ and $x_S,x_S' \in \bbR^S$ with $\|x_{S}-x_{S}'\|_\infty \leq 1$ we have
    \begin{equation}
    \label{Eq:A2}
        d_\mathrm{TV}( \mathcal{L}(X_{S^c} | X_S=x_S), \mathcal{L}(X_{S^c} | X_S=x_S') ) \leq L_2 \|x_S-x_{S'}\|_\infty^{\beta_2}.
    \end{equation}
    \item[(A3)($\beta_3,L_3$)] For all $S \in \bbS$ and $x_S,x_S' \in \bbR^S$ with $\|x_{S}-x_{S}'\|_\infty \leq 1$ we have
    \begin{equation}
    \label{Eq:A3}
        |\bar{r}_S(x_S) - \bar{r}_S(x_S')| \leq L_3 \|x_S - x_S'\|_\infty^{\beta_3}.
    \end{equation}
\end{enumerate}
We assume that $n_S /n \rightarrow \lambda_S \in (0,\infty)$ as $n \rightarrow \infty$ for each $S \in \bbS$ and make the technical assumption that $|n\lambda_S /n_S - 1| \leq 1/n$ for all $n \in \bbN$.

\begin{thm}
\label{Thm:MainResult}
Suppose that the conditions of the previous paragraph hold.
\begin{itemize}
    \item[(i)] There exists an estimator $\hat{\theta} \equiv \hat{\theta}_n$ such that
\[
    \limsup_{n \rightarrow \infty} \bbE_f [n \{\hat{\theta} - \theta(f)\}^2 ] \leq \inf_{\alpha_\bbS \in H_\bbS} \mathcal{L}(\alpha_\bbS),
\]
where $\mathcal{L}(\cdot)$ is defined in~\eqref{Eq:Objective}.
    \item[(ii)] For any measurable estimator sequence $(\theta_n)$ we have
    \[
        \sup_{I \subset L^2} \liminf_{n \rightarrow \infty} \max_{h \in I} \mathbb{E}_{f_{n^{-1/2} h}} \bigl[n \bigl\{ \theta_n - \theta(f_{n^{-1/2} h}) \}^2 \bigr] \geq \inf_{\alpha_\bbS \in H_\bbS} \mathcal{L}(\alpha_\bbS),
    \]
    where the supremum is taken over all finite subsets $I$ of square-integrable functions $h$ on $\bbR^d$ and where the local perturbations $f_{h}$ of $f$ are defined in Section~\ref{Sec:ANOVALowerBound}.
\end{itemize}
\end{thm}

The estimator $\hat{\theta}$ in the above statement is introduced in Section~\ref{Sec:UpperBoundData}. The first part of Theorem~\ref{Thm:MainResult} follows from the finite-sample result Theorem~\ref{Thm:UpperBoundShifted}, but we give an asymptotic statement here for brevity. The second part of Theorem~\ref{Thm:MainResult} is a local asymptotic minimax lower bound following from Theorem~\ref{Thm:ShiftedLowerBound}. The combination of both parts of Theorem~\ref{Thm:MainResult} shows that $\hat{\theta}$ is an efficient estimator of $\theta$ and that asymptotic minimal risks in this problem are characterised by the penalised generalised ANOVA decompositions introduced in Section~\ref{Sec:ANOVADecompositions}.

The requirement that $a$ be bounded is made to simplify our statement but, as with many of our other assumptions, it can be weakened; this is discussed in the sequel, particularly in the discussion after the statement of Theorem~\ref{Thm:UpperBoundShifted}. The assumptions that $\lambda_S \in (0,\infty)$ and $c \leq r_S(x_S) \leq C$ mean that (local) sample sizes across the different datasets are comparable and are commonly made in the literature~\citep[e.g.][Assumption~3.1]{robins2017minimax}. Our assumptions (A1)$(\beta_1,L_1)$, (A2)$(\beta_2,L_2)$ and (A3)$(\beta_3,L_3)$ concern the smoothness of $a$, of the $\bar{r}_S$ and of the relevant conditional distributions of $X$, and are used to control the smoothness of shifted conditional expectations of $a$. With this in mind, the first and third assumptions are natural. The second condition is made so that we can control, for example, the difference in conditional expectations $|\bbE\{a(X) | X_S=x_S\} - \bbE\{a(X) | X_S=x_S'\}|$. Such assumptions have been made previously in other contexts; see, for example, \cite{neykov2021minimax}. % based on the Total Variance distance and $\chi^2$-divergence. As we measure the size of $a$ through $\|a\|_\infty$ in many of our results, the Total Variance distance is a natural choice for us.

We conclude this section with examples of problems covered by our setting.
\begin{eg}
Consider a study where we wish to estimate the mean response $\bbE(Y)$ given data on $X=(Z,U,Y)$. Full information is available for those subjects in the study, but the covariates in $U$ are difficult to collect and unavailable in an additional observational dataset recording $(Z,\texttt{NA},Y)$ \citep[e.g.][]{yang2019combining}. We also have access to unlabelled data on all covariates $(Z,U,\texttt{NA})$, as in semi-supervised problems. Here the minimal variance is given by
\[
    \inf_{\alpha_1,\alpha_2} \biggl[ \mathrm{Var}\bigl( Y - \alpha_1(Z,Y) - \alpha_2(Z,U) \bigr) + \bbE\biggl\{ \frac{\alpha_1(Z,Y)^2}{\lambda_1 \bar{r}_1(Z)} \biggr\} + \bbE\biggl\{ \frac{\alpha_2(Z,U)^2}{\lambda_2 \bar{r}_2(Z,U)} \biggr\} \biggr]
\]
where $\bar{r}_1,\bar{r}_2$ allow for covariate shift between the main study and the observational data.
\end{eg}
\begin{eg}
Two diagnostic tests are available to measure the severity of a health condition~\citep[e.g.][]{xue2021integrating} and a predictive algorithm $m(Z_1,Z_2)$ has been built for a response $Y$ from the two measurements. We want to estimate the squared prediction error based on data from individuals who underwent one or both of the tests, so that we have data on $(Z_1,Z_2,Y),(Z_1,\texttt{NA},Y)$ and $(\texttt{NA},Z_2,Y)$. The minimal variance here is given by
\[
    \inf_{\alpha_1,\alpha_2} \biggl[ \mathrm{Var}\bigl( \{Y-m(Z_1,Z_2)\}^2 - \alpha_1(Z_1,Y) - \alpha_2(Z_2,Y) \bigr) + \bbE\biggl\{ \frac{\alpha_1(Z_1,Y)^2}{\lambda_1 \bar{r}_1(Z_1,Y)} \biggr\} + \bbE\biggl\{ \frac{\alpha_2(Z_2,Y)^2}{\lambda_2 \bar{r}_2(Z_2,Y)} \biggr\} \biggr],
\]
where $\bar{r}_1,\bar{r}_2$ allow for changes in distributions between patients having both tests and those refusing or being advised against one.
\end{eg}

In missing data literature it is sometimes preferred to work with a fixed total sample size and data that is randomly missing. Our results adapt to such settings. We give a corollary of Theorem~\ref{Thm:MainResult} in the MCAR case below and extensions to MAR cases in Section~\ref{Sec:MAR}. Suppose that there is a collection of i.i.d., potentially unobserved data $(Y_1,\Omega_1),\ldots,(Y_N,\Omega_N)$ where $Y_1 \sim f$ and $\Omega_1$ takes values in $\{0,1\}^d$ with $\bbP(\Omega_j =1\, \forall j) >0$. We observe $Y_i \circ \Omega_i$ where
\[
    (y \circ \omega)_j = \left\{ \begin{array}{ll} y_j & \text{if } \omega_j =1 \\ \texttt{NA} & \text{otherwise} \end{array} \right. .
\]
For $S \in 2^{[d]}$ let $\mathbbm{1}_S \in \{0,1\}^d$ be the indicator vector such that $(\mathbbm{1}_S)_j=\mathbbm{1}_{\{j \in S\}}$, write $p_S = \mathbb{P}(\Omega = \mathbbm{1}_S)$ and let $\bbS=\{S \subset [d] : p_S >0\}$. Assuming that data is MCAR, in the sense that $Y \perp \!\!\! \perp \Omega$, then the conditional density of $Y_S | \{\Omega = \mathbbm{1}_S \}$ is given by $f_S$. Writing $N_S = \sum_{i=1}^N \mathbbm{1}_{\{\Omega_i=\mathbbm{1}_S\}}$, on the event that the $N_\bbS=(N_S : S \in \bbS)$ are not too small, we can apply the estimator constructed in Section~\ref{Sec:UpperBound} with $r_S \equiv 1$ and $n_S = N_S$ for each $S \in \bbS$. The following result justifies the efficiency of this procedure in this missing data model.
\begin{cor}
\label{Cor:MCAR}
Suppose that the conditions of Theorem~\ref{Thm:MainResult} hold. Let $\mathcal{L}(\cdot)$ be defined as in~\eqref{Eq:Objective} with $\lambda_S = p_S/p_{[d]}$ and $r_S \equiv 1$ for each $S \in \bbS$. Then there exists an estimator $\hat{\Theta}$ such that
\begin{equation}
\label{Eq:MCARUpperBound}
    \limsup_{n \rightarrow \infty} \bbE_f [N \{\hat{\Theta} - \theta(f)\}^2 ] \leq \frac{1}{p_{[d]}} \inf_{\alpha_\bbS \in H_\bbS} \mathcal{L}(\alpha_\bbS).
\end{equation}
\end{cor}
A corresponding lower bound is proved using similar arguments to those used to prove Theorem~\ref{Thm:ShiftedLowerBound}. Alternatively, we can specialise Theorem~\ref{Thm:MARLowerBound} below, where we consider MAR data. In the example below we apply Corollary~\ref{Cor:MCAR} to a common model for missing data.

\begin{eg}
Using the notation introduced above, suppose that each variable is missing independently, so that for some $p_1,\ldots,p_d \in (0,1]$ and for any $S \subseteq [d]$ we may write $p_S=\bbP(\Omega=\mathbbm{1}_S) = \{\prod_{j \in S} p_j\}\{\prod_{j \notin S} (1-p_j)\}$. Then the minimal asymptotic risk on the right-hand side of~\eqref{Eq:MCARUpperBound} specialises to
\[
    \inf_{\alpha_\bbS \in H_\bbS} \biggl\{ \frac{1}{p_{[d]}} \mathrm{Var} \biggl( a(X) - \sum_{S \subset [d]} \alpha_S(X_S) \biggr) + \sum_{S \subset [d]} \frac{1}{p_S} \mathrm{Var} \bigl( \alpha_S(X_S) \bigr) \biggr\}.
\]
\end{eg}

\section{Lower bound}
\label{Sec:LowerBound}

\subsection{Generalised ANOVA decompositions}
\label{Sec:ANOVADecompositions}

Generalised ANOVA decompositions~\citep[e.g.][]{stone1994use} are used to model and approximate complex functions $a(X)$ of dependent variables, extending work~\citep{hoeffding1948class,efron1981jackknife} on decompositions for functions of independent variables.
In order to formally discuss such decompositions we require some notation. For $S \subset [d]$ write
\[
    L^2 = \biggl\{ \alpha : \mathbb{R}^d \rightarrow \mathbb{R} : \int \alpha^2 f < \infty  \biggr\} \quad \text{and} \quad H_S = \biggl\{ \alpha_S \in L^2 : \alpha_S(x) \equiv \alpha_S(x_S), \int \alpha_S f = 0 \biggr\}
\]
for the set of all square-integrable functions with respect to the probability density function $f$ and the linear subspace thereof consisting of all mean-zero functions that do not depend on those $x_j$ for which $j \not\in S$. When equipped with the inner product $\langle \alpha, \beta \rangle_{L^2} = \int \alpha \beta f$ it is well known that $L^2$ is a Hilbert space. Write $H_\bbS = \prod_{S \in \bbS} H_S$ for the product of these subspaces. Given a function $a \in L^2$, we may then aim to minimise $\mathrm{Var}(a(X) - \sum_{S \in \bbS} \alpha_S(X_S) )$ over $\alpha_\bbS=(\alpha_S : S \in \bbS) \in H_\bbS$. 
%If there exist $S_1,S_2 \in \bbS$ with $S_1 \cap S_2 \neq \emptyset$ we can replace $(\alpha_{S_1},\alpha_{S_2})$ by $(\alpha_{S_1} + \alpha_{S_1 \cap S_2},\alpha_{S_2} - \alpha_{S_1 \cap S_2})$ for any $\alpha_{S_1 \cap S_2} \in H_{S_1 \cap S_2}$ without changing the value of the objective function, so minimisers are generally not unique. However, 
Under natural orthogonality conditions on $\alpha_\bbS$ and regularity conditions on $f$, it can be shown \citep[e.g.][Theorem~3.1]{stone1994use} that there is a unique minimiser $\alpha_\bbS^*$, and we call  $\int a f + \sum_{S \in \bbS} \alpha_S^*$ the generalised ANOVA decomposition of $a$.

In our setting, a modification of this optimisation problem allows us to characterise optimal estimators of $\theta(f)=\int af$. Recalling that $\lambda_S = \lim_{n \rightarrow \infty} n_S/n$, define the objective function
\begin{equation}
\label{Eq:Objective}
    \mathcal{L}(\alpha_\bbS) = \Var\biggl( a(X) - \sum_{S \in \bbS} \alpha_S(X_S) \biggr) + \sum_{S \in \bbS} \bbE \biggl[ \frac{\alpha_S(X_S)^2}{\lambda_S \bar{r}_S(X_S)} \biggr],
\end{equation}
which can be considered as a naturally penalised version of the objective function in the previous paragraph. Our first result shows that $\mathcal{L}$ has a unique minimiser whenever these penalties are non-zero and the $\bar{r}_S$ are bounded. The existence and uniqueness of a minimiser is crucial to our proof of the local asymptotic minimax lower bound in Theorem~\ref{Thm:ShiftedLowerBound}, and is convenient in the upper bound arguments in Section~\ref{Sec:UpperBound}.
\begin{prop}
\label{Prop:UniqueMinimiser}
There exists a unique $\alpha_\bbS^* \in H_\bbS$ such that $\mathcal{L}(\alpha_\bbS^*) = \inf_{\alpha_\bbS \in H_\bbS} \mathcal{L}(\alpha_\bbS)$.
\end{prop}

%The penalty terms in our objective function~\eqref{Eq:Objective} act as regularisers and mean that we can avoid the regularity conditions on $f$ that are made in the literature of generalised ANOVA decompositions. While these penalties arise naturally in our setting, it is interesting to consider the case where we allow $\lambda_S = \infty$ for some $S \in \bbS$. Write $\bbS_\infty = \{S \in \bbS : \lambda_S = \infty\}$. We may assume that $\bbS_\infty$ is \emph{hierarchical} \citep[cf.][]{stone1994use}, in the sense that if $S \in \bbS_\infty$ and $T \subseteq S$ then $T \in \bbS_\infty$, without affecting our optimisation problem. For $S \in \bbS_\infty$ we may also restrict attention to $\alpha_S \in H_S^0$, where we write $H_S^0 \subseteq H_S$ for those functions in $H_S$ that are orthogonal to all functions in $H_T$ for $T \subset S$.  If we assume that there exists $c \in (0,1]$ such that
%\begin{equation}
%\label{Eq:CompletenessCondition}
%    \bbE \biggl[ \biggl\{ \sum_{S \in \bbS_\infty} \alpha_S(X_S) \biggr\}^2 \biggr] \geq c \sum_{S \in \bbS_\infty} \bbE \bigl\{ \alpha_S^2(X_S) \bigr\}
%\end{equation}
%for all $(\alpha_S : S \in \bbS_\infty)$ such that $\alpha_S \in H_{S}^0$ for all $S \in \bbS_\infty$, it can be seen that $\sum_{S \in \bbS_\infty} H_S$ is complete. We can therefore see that the conclusions of Proposition~\ref{Prop:UniqueMinimiser} continue to hold under~\eqref{Eq:CompletenessCondition} and the assumption that $\max_{S \in \bbS \setminus \bbS_\infty} \|\bar{r}_S\|_\infty < \infty$. 

We conclude this section by giving an explicit minimiser in the case that our variables are independent. While this would be rare in practice, it provides intuition and a link to classical ANOVA decompositions. We can express the optimal $\alpha_\bbS$ in terms of suitably-centred conditional expectations of $a$ that are mutually orthogonal. To this end, for $T \subseteq [d]$ write $a_T(x_T) = \mathbb{E}\{a(X) |X_T=x_T\}$. There do not appear to be simple generalisations to nonconstant $r_\bbS$, even when they have product form, due to the fact that the penalty terms in the objective function do not have orthogonal decompositions that align with the orthogonal decomposition of the main term.

\begin{prop}
\label{Prop:ProductCase}
Suppose that $X$ has a product distribution and that $r_S \equiv 1$ for all $S \in \bbS$. For $T \subseteq [d]$ inductively define $\tilde{a}_\emptyset = \theta$ and $\tilde{a}_T = a_T - \sum_{U \subset T} \tilde{a}_U$. Then $\mathcal{L}(\alpha_\bbS)$ is minimised by taking
\begin{equation}
\label{Eq:ProductCase}
    \alpha_S = \sum_{\emptyset \neq T \subseteq S} \frac{\lambda_S}{1+\sum_{S' \in \bbS : T \subseteq S'} \lambda_{S'}} \tilde{a}_T
\end{equation}
for each $S \in \bbS$.
\end{prop}

\subsection{Local asymptotic minimax lower bound}
\label{Sec:ANOVALowerBound}

Our next result establishes a link between the minimisation of $\mathcal{L}$ and our estimation problem. Write $k(u)=1/2 + (1+e^{-4u})^{-1}$ and for $h \in L^2$ let $f_{h}$ be the perturbation of $f$ such that $f_{h}(x) \propto k(h(x))f(x)$ and $f_h$ is a density function. This is a convenient choice of $k$ since it is smooth, bounded away from zero and infinity and satisfies $k(0)=k'(0)=1$ and $k''(0)=0$.
\begin{thm}
\label{Thm:ShiftedLowerBound}
Writing $\alpha_\bbS^*$ for the minimiser of~\eqref{Eq:Objective} and $\alpha^* = a - \theta -\sum_S \alpha_S^*$, for any estimator sequence $(\theta_n)$ we have
    \[
        \sup_{I \subset \bbR} \liminf_{n \rightarrow \infty} \max_{t \in I} \,\, \mathbb{E}_{f_{n^{-1/2}t\alpha^*}} \bigl[ n \bigl\{ \theta_n - \theta(f_{n^{-1/2}t\alpha^*}) \}^2 \bigr] \geq \inf_{\alpha_\bbS \in H_\bbS} \mathcal{L}(\alpha_\bbS),
    \]
    where the supremum is taken over all finite subsets $I$ of $\bbR$.
\end{thm}
When combined with our later upper bounds this result shows that, under suitable regularity conditions, the minimal asymptotic variance is given by $n^{-1}\inf_{\alpha_\bbS \in H_\bbS} \mathcal{L}(\alpha_\bbS)$. As well as establishing our lower bound, the result shows that $\alpha^*$ is the direction in which we should perturb $f$ for a least-favourable submodel and thus informs the construction of optimal estimators to follow in Section~\ref{Sec:UpperBound}. In the language of efficiency theory, this means that $\alpha^*$ is the efficient influence function for the parameter $\theta(P)$ in our model. 

% I AM UNABLE TO PROVIDE GOOD INTUITION?
In order to provide some intuition for this efficient influence function we give formal calculations for perturbations in a general direction, restricting attention to bounded perturbations for simplicity. For a bounded function $h$ satisfying $\int hf =0$ write $P_{n,h}$ for the distribution of the data when $f$ is replaced by $f(1+n^{-1/2}h)$ and when $n$ is sufficiently large that this is non-negative. Writing $h_S(x) \equiv h_S(x_S)=\bbE\{h(X)|X_S=x_S\}$, the density of $(X_{S,1})_S$ under $P_{n,h}$ is proportional to $r_S f_S(1+n^{-1/2}h_S)$ and
%\[
%    \frac{r_S f_S(1+n^{-1/2}h_S)}{\int r_Sf_S(1+n^{-1/2}h_S)} = \frac{\bar{r}_Sf_S(1+n^{-1/2}h_S)}{1+n^{-1/2} \bbE\{\bar{r}_S(X_S)h_S(X_S)\}} = \frac{\bar{r}_Sf_S(1+n^{-1/2}h_S)}{1+n^{-1/2} \bbE\{h_S(X_{S,1})\}}.
%\]
we therefore see that
\begin{align*}
    \log \frac{dP_{n,h}}{dP_{n,0}} &= \sum_{i=1}^n \log \bigl(1+n^{-1/2}h(X_i) \bigr) + \sum_{S \in \bbS} \sum_{j=1}^{n_S} \log \biggl( \frac{1+n^{-1/2}h_S(X_{S,j})}{1+n^{-1/2}  \bbE\{h_S(X_{S,1})\}} \biggr) \\
    %& = n^{-1/2} \sum_{i=1}^n h(X_i) - \frac{1}{2n} \sum_{i=1}^n h(X_i)^2 + n^{-1/2}\sum_{S \in \mathbb{S}} \sum_{j=1}^{n_S} \bigl[ h_S(X_{S,j}) - \bbE\{ h_S(X_{S,j})\} \bigr] \\
    %& \hspace{50pt} - \frac{1}{2n} \sum_{j=1}^{n_S} \bigl[ h_S(X_{S,j})^2 - \bbE^2 \{ h_S(X_{S,j})\} \bigr] + o_p(n^{-1/2}) \\
    & = n^{-1/2} \sum_{i=1}^n h(X_i) - \frac{1}{2} \Var\{h(X_1)\} + n^{-1/2} \sum_{S \in \bbS} \sum_{j=1}^{n_S} \bigl[ h_S(X_{S,j}) - \bbE\{ h_S(X_{S,j})\} \bigr] \\
    & \hspace{50pt} - \frac{1}{2} \sum_{S \in \bbS} \lambda_S \Var \{h_S(X_{S,j})\} + o_p(n^{-1/2})
\end{align*}
as $n \rightarrow \infty$, under $P_{n,0}$. Thus, writing $\lambda_\bbS = (\lambda_S : S \in \bbS)$, if we define the inner product
\begin{equation}
\label{Eq:InnerProduct}
     \langle h,h' \rangle_{\lambda_\bbS} = \mathbb{E}\{ h(X)h'(X)\} + \sum_{S \in \bbS} \lambda_S \Cov \bigl( h_S(X_{S,1}), h_S'(X_{S,1}) \bigr),   
\end{equation}
we have that $\log \frac{dP_{n,h}}{dP_{n,0}} \overset{\mathrm{d}}{\rightarrow} \|h\|_{\lambda_\bbS} Z - (1/2) \|h\|_{\lambda_\bbS}^2$ for a standard Gaussian variable $Z$, so that our sequence of experiments in locally asymptotically normal. On the other hand, we have $\theta(f(1+n^{-1/2}h)) - \theta(f) = n^{-1/2} \int a h f$. According to local asymptotic minimax theory, we should therefore choose $h$ to maximise $\int ahf$ under the constraint $\|h\|_{\lambda_\bbS} \leq 1$. This is done by finding the adjoint of the map $\dot{\kappa} : h' \mapsto \int a f h'$, i.e.~by choosing $h$ proportional to the function $\dot{\kappa}^*$ that satisfies $\langle \dot{\kappa}^* , h' \rangle_{\lambda_\bbS} = \dot{\kappa}(h')$ for all $h'$. We show in the proof of Theorem~\ref{Thm:ShiftedLowerBound} that the stationarity conditions for minimising $\mathcal{L}(\cdot)$ ensure that $\alpha^*$ indeed satisfies this property.
% Let $H_\bbS^0$ be the set of all bounded functions $h$ with $h_S \equiv 0$ for all $S \in \bbS$. For any bounded $h$ we can write $h=h^0 + h^\perp \in H_\bbS^0 + (H_\bbS^0)^\perp$. Fixing $h^\perp$, we then think of maximising $\int a h^0 f$ over $h^0 \in H_\bbS^0$ such that $\int (h^0)^2 f$ is bounded by a given constant. Writing $a=a^0 + a^\perp$ we form the Lagrangian
% \[
%     \mathcal{L}(h,\mu_\bbS) = \int a^0 h f - \sum_{S \in \bbS} \int \mu_S(x_S) h(x) f(x) \,dx - \frac{\mu}{2} \int h^2 f.
% \]
% We see that this is maximised over bounded $h$ by taking $h(x) = (1/\mu)\{a^0(x) - \sum_{S \in \bbS} \mu_S(x_S)\}$. Substituting this into $\int a^0 hf$ and choosing $\mu$ minimally, we see that we are left to maximimise
% \[
%     \frac{\int (a^0)^2 f}{\int (a^0 - \sum_{S \in \bbS} \mu_S)^2 f }
% \]
% over $\mu_\bbS$, for which we must take $\mu_\bbS=0$. Thus, we should take $h^0 \propto a^0$, the projection of $a$ onto $H_\bbS^0$.

\subsection{Examples and special cases}
\label{Sec:ANOVAExamples}

We conclude this section with illustrative examples and special cases to aid the interpretation of the limiting variance $\inf_{\alpha_\bbS \in H_\bbS} \mathcal{L}(\alpha_\bbS)$. For the simplest structures, where $\bbS=\{S\}$ is a singleton, elementary calculations show that the optimisation problem can be solved by taking
\begin{equation}
\label{Eq:SemiSupervised}
    \alpha_S(x_S) = \frac{\lambda_S \bar{r}_S(x_S)}{1+\lambda_S \bar{r}_S(x_S)}\bbE\bigl\{ a(X) - \theta' | X_S=x_S \bigr\},
\end{equation}
for $\theta'$ such that $\bbE\{\alpha_S(X_S)\}=0$. Our next result extends these calculations to general monotonic structures of $\bbS$. 
\begin{prop}
\label{Prop:ShiftedMonotoneCase}
Suppose that $\bbS \subseteq \{[1],[2],\ldots,[d-1]\}$. Using the shorthand $Y_j=(X_1,\ldots,X_j)$ and $y_j=(x_1,\ldots,x_j)$ we inductively define the functions $R_j(y_j) = \lambda_{[j]}\bar{r}_{[j]}(y_j)$, $\mu_{d-1}(y_{d-1}) \equiv 1$, $\tilde{a}_{d-1}(y_{d-1}) = a_{[d-1]}(y_{d-1})-\theta$, $\nu_j^{j-1}(y_j) \equiv 1$,
\begin{align*}
    \mu_j(y_j)&= \bbE \biggl\{ \frac{\mu_{j+1}(Y_{j+1})}{1+R_{j+1}(Y_{j+1})\mu_{j+1}(Y_{j+1})} \biggm|Y_j = y_j \biggr\},\\
    \tilde{a}_j(y_j) &= \frac{1}{\mu_j} \bbE \biggl\{ \frac{\mu_{j+1}(Y_{j+1}) \tilde{a}_{j+1}(Y_{j+1})}{1+ R_{j+1}(Y_{j+1})\mu_{j+1}(Y_{j+1})} \biggm| Y_j=y_j \biggr\} \quad \text{and} \\
     \nu_j^k(y_j) &= \bbE \biggl\{ \frac{\nu_{j+1}^k(Y_{j+1})}{1 + R_{j+1}(Y_{j+1}) \mu_{j+1}(Y_{j+1})} \biggm| Y_j = y_j \biggr\}
\end{align*}
for $j \in [d-1]$ and $k\geq j$. For notational convenience we write $\nu_j^k = 0$ for $k<j-1$, write $\nu_0^k \equiv 0$ for all $k$ and write $\tilde{a}_0 \equiv 0$. Then $\mathcal{L}(\alpha_\bbS)$ is minimised by taking
\begin{equation}
\label{Eq:ShiftedMonotoneMinimiser}
    \alpha_{[j]} = \frac{R_j}{1/\mu_j + R_j} \sum_{k=1}^j \biggl( \prod_{m=k}^{j-1} \frac{1}{1+R_m \mu_m} \biggr) \biggl\{ \tilde{a}_k - \tilde{a}_{k-1} - \sum_{\ell=1}^{d-1} \theta_\ell \biggl( \frac{\nu_k^{\ell-1}}{\mu_k} - \frac{
\nu_{k-1}^{\ell-1}}{\mu_{k-1}} \biggr) \biggr\},
\end{equation}
for each $j \in [d-1]$ such that $[j] \in \bbS$, where $\theta_1,\ldots,\theta_{d-1} \in \bbR$ are constants such that $\theta_j=0$ if $[j] \not\in \bbS$ and such that each $\alpha_{[j]}$ has mean zero.
\end{prop}

While the optimal $\alpha_\bbS$ given in Proposition~\ref{Prop:ShiftedMonotoneCase} is rather complex, it can be written as a sum of finitely many terms, each of which can be expressed through a finite number of conditional expectations. This is due to the monotonic structure of $\bbS$, which simplifies the optimisation problem given by~\eqref{Eq:Objective} and allows for a natural sequential approach. To gain some intuition for this minimiser and the quantities involved in its definition we specialise Proposition~\ref{Prop:ShiftedMonotoneCase} to the MCAR setting in the following corollary. We also give a simple direct proof.
\begin{cor}
\label{Prop:MonotoneCase}
Suppose that $\bbS \subseteq \{[1],[2],\ldots,[d-1]\}$ and that $r_S \equiv 1$ for all $S \in \bbS$. Then $\mathcal{L}(\alpha_\bbS)$ is minimised by taking
\[
    \alpha_{[j]} = \sum_{k=1}^j  \frac{\lambda_{[j]}}{1+\sum_{\ell=k}^{d-1} \lambda_{[\ell]}} (a_{[k]} - a_{[k-1]})
\]
for each $j$ such that $[j] \in \bbS$.
\end{cor}

When we move beyond monotonic structures, it is not possible to give explicit minimisers in such generality. However, the following example exhibits a non-monotonic setting in which our results give simple expressions for optimal variances and influence functions. 
\begin{eg}
Write $X=(Y,Z)$ where $Y$ is a binary response taking values in $\{0,1\}$ and $Z$ is a covariate taking values in $\bbR^p$, with $p=d-1$, and write $\eta(z)=\bbP(Y=1|Z=z)$. Suppose that $\bbS=\{\{1\},[d]\setminus \{1\}\}$ so that we have incomplete data on $Y$ only and incomplete data on $Z$ only and write $r_Y,r_Z$ for the associated distributional shifts. Consider the problem of estimating $\theta=\bbE(Y)$. The restriction to binary responses means that the only mean-zero functions $\beta$ of $Y$ are of the form $\beta(Y)=\mu(Y-\theta)$. We therefore choose $\alpha$ and $\mu$ to minimise
\begin{align*}
    &\bbE \biggl[\{Y-\theta-\alpha(Z)-\mu(Y-\theta)\}^2 + \frac{\alpha(Z)^2}{\lambda_Z \bar{r}_Z(Z)} + \frac{\mu^2}{\lambda_Y \bar{r}_Y(Y)}(Y-\theta)^2 \biggr] \\
    %&= \bbE \biggl[ (1-\mu)^2(Y-\theta)^2 - 2(1-\mu)(Y-\theta) \alpha(Z) + \alpha(Z)^2 + \frac{\alpha(Z)^2}{\lambda_Z \bar{r}_Z(Z)} + \frac{\mu^2}{\lambda_Y \bar{r}_Y(Y)}(Y-\theta)^2 \biggr] \\
    & = \bbE \biggl[ \biggl\{ (1-\mu)^2 + \! \frac{\mu^2}{\lambda_Y \bar{r}_Y(Y)} \biggr\} (Y-\theta)^2   + \biggl\{ 1 \!+\! \frac{1}{\lambda_Z \bar{r}_Z(Z)} \biggr\} \alpha(Z)^2 - 2(1-\mu) \{\eta(Z)-\theta\} \alpha(Z) \biggr].
\end{align*}
For each fixed $\mu \in \bbR$, as for~\eqref{Eq:SemiSupervised} it can be seen that this is minimised by taking $\alpha=(1-\mu)\frac{\lambda \bar{r}_Z}{1+\lambda_Z \bar{r}_Z}(\eta-\bar{\theta})$, 
%\[
%    \alpha(z) = (1-\mu) \frac{\lambda_Z \bar{r}_Z(Z)}{1+\lambda_Z \bar{r}_Z(Z)} \{ \eta(z) - \bar{\theta}\},
%\]
where $\bar{\theta} \in \bbR$ is chosen such that $\alpha$ has mean zero. It therefore follows that in this example
\begin{align*}
    \inf_{\alpha_\bbS} \mathcal{L}(\alpha_\bbS) &= \inf_{\mu \in \bbR}  \bbE\biggl[ \biggl\{ (1-\mu)^2 + \frac{\mu^2}{\lambda_Y \bar{r}_Y(Y)} \biggr\} (Y-\theta)^2 - (1-\mu)^2 \frac{\lambda_Z \bar{r}_Z(Z)}{1+\lambda_Z \bar{r}_Z(Z)} \{\eta(Z) - \bar{\theta}\}^2 \biggr] \\
    &=\biggl[ \frac{1}{ \theta(1-\theta) - \bbE\bigl[ \frac{\lambda_Z \bar{r}_Z(Z)}{1+\lambda_Z \bar{r}_Z(Z)}\{\eta(Z) - \bar{\theta}\}^2 \bigr]  } + \frac{\lambda_Y}{\bbE\{ \frac{(Y-\theta)^2}{\bar{r}_Y(Y)} \}} \biggr]^{-1}.
\end{align*}
When $r_Z, r_Y \equiv 1$, this final expression simplifies to the harmonic mean of $\bbE \{ \mathrm{Var}(Y|Z) \} + (1+\lambda_Z)^{-1} \mathrm{Var}(\eta(Z))$ and $\lambda_Y^{-1} \mathrm{Var}(Y)$.
\end{eg}

\section{Upper bound}
\label{Sec:UpperBound}

In this section we construct estimators whose variance is approximately equal to the minimal asymptotic variance $n^{-1} \inf_{\alpha_\bbS \in H_\bbS} \mathcal{L}(\alpha_\bbS)$. We do this by finding an approximate influence function for our problem (in Section~\ref{Sec:UpperBoundApprox}), which will depend on unknown properties of $f$, and then approximating this in a data-driven fashion using cross fitting (in Section~\ref{Sec:UpperBoundData}). Finally, in Section~\ref{Sec:UpperBoundConfidence}, we will use this development to introduce confidence intervals for $\theta$. Our lower bounds in the previous section were naturally asymptotic, but the results in this section are valid with finite sample sizes. We present a simulation study in Section~\ref{Sec:Simulations}.

\subsection{Construction of an approximate influence function}
\label{Sec:UpperBoundApprox}

As shown by Theorem~\ref{Thm:ShiftedLowerBound}, the efficient influence function for our problem is given by $\alpha^*=a-\theta-\sum_{S \in \bbS} \alpha_S^*$, where $\alpha_\bbS^*$ is the minimiser of $\mathcal{L}(\cdot)$. This has no general explicit form, but we will see that it can be approximated by repeated use of a simple technique. Choosing a step size $\eta \in (0,1]$ and initialising at $\alpha_\bbS^{(0)} = 0$, we iteratively define $\alpha_\bbS^{(m)}$  by
\[
    \alpha_S^{(m+1)}(x_S) = (1-\eta) \alpha_S^{(m)}(x_S) + \eta \frac{\lambda_S \bar{r}_S(x_S)}{1+\lambda_S \bar{r}_S(x_S)} \bbE \biggl\{ a(X) - \sum_{S' \in \bbS \setminus \{S\}} \alpha_{S'}^{(m)}(X_{S'}) - \theta' \biggm| X_S = x_S \biggr\},
\]
where $\theta'$ is chosen to centre the random variable $\alpha_S^{(m+1)}(X_S)$. In view of~\eqref{Eq:SemiSupervised}, when $\eta=1$ this can be seen as applying the naive method, which treats all incomplete datasets separately, to the residuals of the current approximation to $a(\cdot)$. Our next result is proved by showing that this iterative scheme can be thought of as gradient descent in a suitable linear space. This allows us to use techniques from convex optimisation to bound $\mathcal{L}(\alpha_\bbS^{(M)}) - \mathcal{L}(\alpha_\bbS^*)$.
\begin{prop}
\label{Prop:OptimisationShifted}
Write $\lambda_\mathrm{max}=\max_{S \in \bbS} \lambda_S$ and recall that we assume that $\|\bar{r}_S\|_\infty \leq C$ for all $S \in \bbS$. Choosing $\eta = |\bbS|^{-1}$ and writing $\kappa = |\bbS|(1+C \lambda_\mathrm{max})$ we have
\[
    \mathcal{L}(\alpha_\bbS^{(M)}) - \mathcal{L}(\alpha_\bbS^*) \leq \kappa (1-1/\kappa)^{M} \mathrm{Var} \, a(X) 
\]
for all $M \in \bbN_0$. Moreover, $\mathcal{L}(\alpha_\bbS^{(M+1)}) \leq \mathcal{L}(\alpha_\bbS^{(M)})$ for all $M \in \bbN_0$.
\end{prop}
In the proof of this result we use the smoothness and strong convexity of our problem to show that the value of $\mathcal{L}(\cdot)$ decreases exponentially quickly with a rate governed by the condition number $\kappa$. The strong convexity of the problem follows from the inclusion of the penalty terms in the objective function, which leads to the dependence of $\kappa$ on $C \lambda_\mathrm{max}$. %However, if we instead assume that~\eqref{Eq:CompletenessCondition} holds with $\bbS_\infty = \bbS$ then we have

\subsection{Data-driven estimator}
\label{Sec:UpperBoundData}

We now turn to the construction of data-driven approximations of $\alpha_\bbS^{(M)}$. It will be convenient to define the approximate oracle estimator
\begin{equation}
\label{Eq:ApproxOracle}
    \theta^{*,(M)} = \frac{1}{n} \sum_{i=1}^n \biggl\{a(X_i) - \sum_{S \in \bbS} \alpha_S^{(M)}(X_i) \biggr\} + \sum_{S \in \bbS} \frac{1}{n_S} \sum_{j=1}^{n_S} \frac{\alpha_S^{(M)}(X_{S,j})}{\bar{r}_S(X_{S,j})}
\end{equation}
that we will aim to mimic with our data. By construction, each $\alpha_S^{(M)}$ has mean zero under $f$ and, since the density of $(X_{S,1})_S$ is given by $\bar{r}_S f_S$, we see that $\theta^{*,(M)}$ is unbiased for $\theta$. We can show that the variance of $\theta^{*,(M)}$ is approximately minimal using Proposition~\ref{Prop:OptimisationShifted} and arguing that $\inf_{\alpha_\bbS \in H_\bbS} \mathcal{L}(\alpha_\bbS)$ is insensitive to small changes in $\lambda_\bbS$, as we do in the proof of Proposition~\ref{Prop:CrossFitReduction} below. The latter is a technical but simple point, arising only because we have $n_S/n \rightarrow \lambda_S$ as $n \rightarrow \infty$ rather than equality.

We will approximate $\theta^{*,(M)}$ using cross fitting. Write $\mathcal{D}_1 = \{X_1,\ldots,X_{\lceil n/2 \rceil}\}$ for the first half of the complete data and $\mathcal{D}_2 = \{X_{\lceil n/2 \rceil + 1},\ldots,X_n\}$ for the second half. Let $\hat{r}_{S,(\ell)} = |\mathcal{D}_{\ell}|^{-1} \sum_{x \in \mathcal{D}_{\ell}} r_S(x)$ and let $\hat{\alpha}_{S,(\ell)}^{(M)}(\cdot)$ be a function that is calculated using the $\mathcal{D}_\ell$ data only, for $\ell=1,2$, to be constructed below. Define 
\[
    \hat{\theta}_{(\ell)}^{(M)} = \frac{1}{|\mathcal{D}_{3-\ell}|}\sum_{x \in \mathcal{D}_{3-\ell}}\biggl\{ a(x) - \sum_{S \in \bbS} \hat{\alpha}_{S,(\ell)}^{(M)}(x) \biggr\} + \sum_{S \in \bbS} \frac{1}{n_S} \sum_{j=1}^{n_S} \frac{\hat{\alpha}_{S,(\ell)}^{(M)}(X_{S,j}) \hat{r}_{S,(3-\ell)}}{ r_S(X_{S,j})}
\]
and $\hat{\theta}=(|\mathcal{D}_2| \hat{\theta}_{(1)}^{(M)} + |\mathcal{D}_1| \hat{\theta}_{(2)}^{(M)})/n$. The next result shows that $\hat{\theta}$ approximates $\theta^{*,(M)}$ well whenever the $\hat{\alpha}_{S,(\ell)}^{(M)}$ are good estimators of the $\alpha_S^{(M)}$ and establishes basic properties of $\theta^{*,(M)}$.
\begin{prop}
\label{Prop:CrossFitReduction}
Let $C>0$ be such that $\|\bar{r}_S\|_\infty \leq C$ for all $S \in \bbS$. We have that
\[
    \bbE \bigl\{ ( \hat{\theta} - \theta^{*,(M)} )^2 \bigr\} \leq \frac{8C|\bbS|}{n} \biggl[ \max_{\ell=1,2} \sum_{S \in \bbS} \bbE \biggl\{ \int 
\frac{n+n_S \bar{r}_S}{n_S \bar{r}_S} (\hat{\alpha}_{S,(\ell)}^{(M)} - \alpha_S^{(M)})^2 f_S \biggr\} + \mathrm{Var}(\theta^{*,(M)}) \biggr].
\]
Moreover,
\[
    \mathrm{Var}(\theta^{*,(M)}) \leq \frac{1}{n} \biggl( 1 + \max_{S \in \bbS} \biggl| \frac{n \lambda_S}{n_S} - 1 \biggr| \biggr) \mathcal{L}(\alpha_\bbS^{(M)})
\]
and $\theta^{*,(M)}$ is an unbiased estimator of $\theta$.
\end{prop}

It now remains to construct the $\hat{\alpha}_{\bbS,(\ell)}^{(M)}$ as estimators of $\alpha_\bbS^{(M)}$. We follow the iterative definition of $\alpha_\bbS^{(M)}$, replacing conditional expectations with nonparametric regression and employing further sample splitting to ease the analysis. Write $\mathcal{D}_{\ell,m}$ for the $m$th piece of the $\ell$th half, for $m=1,2,\ldots,M$ and $\ell=1,2$. To lessen the notational complexity we focus on the $\ell=1$ half of the data and suppress the dependence on $\ell$ in the following notation. Write $K(z)=2^{-d} \mathbbm{1}_{\{\|z\|_\infty \leq 1\}}$ for the uniform kernel. Given $S \in \bbS$ and a bandwidth $h \in (0,1)$ write $K^S(u) \equiv K^S(u_S) = \int_{\mathbb{R}^{S^c}} K(u_S,u_{S^c}') \,du_{S^c}'$ and $K_h^S(u) \equiv K_h^S(u_S) =h^{-|S|}K^S(u_S/h)$. Take $\hat{r}_S^{(m)} = |\mathcal{D}_{1,m}|^{-1} \sum_{x \in \mathcal{D}_{1,m}} r_S(x)$ and $\hat{f}_S^{(m)}(\cdot) = |\mathcal{D}_{1,m}|^{-1} \sum_{x \in \mathcal{D}_{1,m}} K_h^S(\cdot - x)$. Initialising at $\hat{\alpha}_{\bbS,(1)}^{(0)} \equiv 0$, we then define $\hat{\alpha}_{\bbS,(1)}^{(m)}$ for $m=1,\ldots,M$ via
\begin{align*}
    & \hat{\alpha}_{S,(1)}^{(m)}(x_S) = \eta \frac{\lambda_S r_S(x_S)}{\hat{r}_S^{(m)}+\lambda_S r_S(x_S)} \biggl[ |\mathcal{D}_{1,m}|^{-1} \sum_{y \in \mathcal{D}_{1,m}} \biggl\{a(y) - \sum_{S' \neq S} \hat{\alpha}_{S',(1)}^{(m-1)}(y) \biggr\} \mathbbm{1}_{B_{T}}(y) \frac{K_h^S(x_S - y_S)}{\hat{f}_S^{(m)}(x_S)} \\
    &  \hspace{50pt} -  \frac{\sum_{y \in \mathcal{D}_{1,m}} \frac{\lambda_S r_S(y)}{\hat{r}_S^{(m)}+\lambda_S r_S(y)} \{a(y) - \sum_{S' \neq S} \hat{\alpha}_{S',(1)}^{(m-1)}(y) \} \mathbbm{1}_{B_T}(y)}{\sum_{y \in \mathcal{D}_{1,m}} \frac{\lambda_S r_S(y)}{\hat{r}_S^{(m)}+\lambda_S r_S(y)}}  \biggr] + (1-\eta) \hat{\alpha}_{S,(1)}^{(m-1)}(x_S)
\end{align*}
when $\|x_S\|_\infty \leq T$ and take $\hat{\alpha}_{(\bS,S)}^{(m+1)}(x_S) = 0$ otherwise. Here we use the convention that $0/0=0$ to deal with the event that $\hat{f}_S^{(m+1)}(x_S)=0$ and take $B_T=\{x: \|x\|_\infty \leq T\}$ for some $T\geq 1$ to be chosen later. Note that, when $r_S \equiv 1$, the centring term in the iteration can be omitted, as the final estimator $\hat{\theta}$ is then invariant to constant additions to the $\hat{\alpha}_{S,(1)}^{(M)}$.
Our next result controls the error of the estimators $\hat{\alpha}_{\bbS,(1)}^{(M)}$.
\begin{prop}
\label{Prop:InductionProof}
Let $\beta_1,\beta_2,\beta_3 \in (0,1]$ and $L_1,L_2,L_3 \in (0, \infty)$, and suppose that $a$, the distribution of $X$ and $(r_S : S \in \bbS)$ satisfy (A1)($\beta_1,L_1$), (A2)($\beta_2,L_2$) and (A3)($\beta_3,L_3$), respectively, and write $\beta_\wedge = \beta_1 \wedge \beta_2 \wedge \beta_3$. Suppose further that $c \leq \bar{r}_S(x_S) \leq C$ for all $x_S \in \bbR^S$ and $S \in \bbS$. Then for any $M \in \bbN$ and  $S \in \bbS$ we have
\[
    \bbE \biggl\{ \int (\hat{\alpha}_{S,(1)}^{(M)} - \alpha_S^{(M)})^2 f_S \biggr\} \leq A^{M} B  \biggl\{ \frac{MT^d}{nh^d} + h^{2\beta_\wedge} + \mathbb{P}(\|X\|_\infty \geq T) \biggr\},
\]
where $A= 28(C/c+2^d)$ and $B = 16\max(\|a\|_\infty^2, 1) \max \{(C/c)^2, (L_1+L_2+L_3)^2,8^d\}$.
\end{prop}
The proof of this result is based on a careful analysis of iterated nonparametric regression. This relies on the assumptions (A1)$(\beta_1,L_1)$, (A2)$(\beta_2,L_2)$ and (A3)$(\beta_3,L_3)$ made in~\eqref{Eq:A1},~\eqref{Eq:A2} and~\eqref{Eq:A3}, respectively. Lemma~\ref{Lemma:Smoothness} and~\eqref{Eq:ApproxInfluence} in the proof of Proposition~\ref{Prop:InductionProof} in Section~\ref{Sec:UpperBoundProofs} show that these assumptions imply the smoothness of the $\alpha_\bbS^{(m)}$.

The combination of Propositions~\ref{Prop:OptimisationShifted},~\ref{Prop:CrossFitReduction} and~\ref{Prop:InductionProof} yields our main result on the risk of $\hat{\theta}$.
\begin{thm}
\label{Thm:UpperBoundShifted}
Suppose that the conditions of Propositions~\ref{Prop:OptimisationShifted} and~\ref{Prop:InductionProof} hold and suppose further that $|n\lambda_S /n_S - 1| \leq 1/n$ for all $S \in \bbS$. When $h \in [T(M/n)^{1/d},1]$ we have
\begin{align*}
    &n\bbE\{ (\hat{\theta}-\theta)^2 \} - \mathcal{L}(\alpha_\bbS^*) \leq A^M D  \biggl\{ \frac{MT^d}{nh^d}+ h^{2\beta_\wedge} + \mathbb{P}(\|X\|_\infty \geq T) \biggr\}^{1/2} \\
    & \hspace{250pt}+4 \kappa  \{\mathrm{Var} \,a(X)\}\exp(-M/\kappa).
\end{align*}
where $D = 200 |\bbS|^2 (C/c)(1+\lambda_\mathrm{min}^{-1}) B$ and $\lambda_\mathrm{min} = \min_{S \in \bbS} \lambda_S$. In particular, if all problem parameters are held fixed then
\[
    \bbE\{ (\hat{\theta}-\theta)^2 \} = \frac{1}{n} \mathcal{L}(\alpha_\bbS^*) + o(1/n)
\]
whenever $h,T,M$ satisfy $A^{2M}\{M T^d/(nh^d) + h^{2\beta_\wedge} + \mathbb{P}(\|X\|_\infty \geq T)\} \rightarrow 0$ and $T,M \rightarrow \infty$.
\end{thm}

If there exists $\rho>0$ such that $\mathbb{E}(\|X\|_\infty^\rho) <\infty$ then we may choose, for example, $h=1/T=n^{-1/(4d)}$ and $M=\sqrt{ \log n}$ to satisfy the final condition in Theorem~\ref{Thm:UpperBoundShifted}. However, this moment condition can be weakened by choosing $M$ to diverge more slowly. Indeed, for any fixed distribution of $X$ we have $\bbP(\|X\|_\infty \geq n^{1/(4d)}) \rightarrow 0$ as $n \rightarrow \infty$ and we can choose $M$ to diverge sufficiently slowly that $A^{2M} \bbP(\|X\|_\infty \geq n^{1/(4d)}) \rightarrow 0$. It would also be possible to weaken our smoothness assumptions to require less than H\"older smoothness, though this would require changes to the proof of Proposition~\ref{Prop:InductionProof}.

The proof of Proposition~\ref{Prop:InductionProof} reveals that our assumptions on $a,f$ and $(r_S : S \in \bbS)$ only need to hold for $x \in B_T$. We could weaken them to allow $C,c,L_1,L_2,L_3,\|a\|_\infty$ to diverge with $T$, provided we can control the error bound in Theorem~\ref{Thm:UpperBoundShifted}. For example, if we have that $\sup_{x \in \bbR^d} \frac{a(x)}{\|x\|_\infty \vee 1} \leq 1$ and all other constants remain bounded, we can make choices of $h,T,M$ such that the error converges to zero if we have $\mathbb{E}(\|X\|_\infty^\rho) < \infty$ for some $\rho>4$.

\subsection{Confidence intervals}
\label{Sec:UpperBoundConfidence}

In this section we show how the development in Sections~\ref{Sec:UpperBoundApprox} and~\ref{Sec:UpperBoundData} can be used to give confidence intervals for $\theta$ which, given the results of Section~\ref{Sec:LowerBound}, will be of asymptotically minimal width. It follows from Proposition~\ref{Prop:CrossFitReduction} and the fact that $\theta^{*,(M)}$, defined in~\eqref{Eq:ApproxOracle}, is an average of independent random variables with finite variance that our estimator $\hat{\theta}$ is approximately normally distributed. To construct confidence intervals we therefore need only an estimator of $n\mathrm{Var}(\theta^{*,(M)})$ which, by Proposition~\ref{Prop:CrossFitReduction}, approaches the optimal variance $\inf_{\alpha_\bbS \in H_\bbS} \mathcal{L}(\alpha_\bbS)$ for large $M$ and $n$. This can be achieved by using similar techniques to those used in Section~\ref{Sec:UpperBoundData}. Indeed, for $\ell=1,2$ and $S \in \bbS$ define
\[
    \hat{V}^{(\ell)} = \frac{1}{|\mathcal{D}_{3-\ell}|} \sum_{x \in \mathcal{D}_{3-\ell}} \biggl\{ a(x) - \sum_{S \in \bbS} \hat{\alpha}_{S,(\ell)}^{(M)}(x) \biggr\}^2 - (\hat{\theta})^2 \quad \text{and} \quad \hat{V}_S^{(\ell)} = \frac{n}{|\mathcal{D}_{3-\ell}|} \sum_{x \in \mathcal{D}_{3-\ell}} \frac{\hat{r}_{S,(\ell)} \hat{\alpha}_{S,(\ell)}^{(M)}(x)^2}{n_S r_S(x)},
\]
where we suppress the dependence of these estimators on $M$ and $h$ for notational convenience. Our final estimator of the variance is then
\[
    \hat{V} = \frac{1}{2} \biggl\{ \hat{V}^{(1)} + \hat{V}^{(2)} + \sum_{S \in \bbS} (\hat{V}_S^{(1)} + \hat{V}_S^{(2)}) \biggr\} \vee 0
\]
and, writing $z_{\alpha/2}$ for the upper $1-\alpha/2$ quantile of the standard normal distribution, we define the interval $C_\alpha = [\hat{\theta}-n^{-1/2}\hat{V}^{1/2}z_{\alpha/2}, \hat{\theta}+n^{-1/2}\hat{V}^{1/2}z_{\alpha/2}]$. We now prove that this is an asymptotically-valid $100(1-\alpha)\%$ confidence interval for $\theta$. Writing $d_\mathrm{K}$ for the Kolmogorov distance we have the following result.
\begin{prop}
\label{Prop:Confidence}
Suppose that the conditions of Theorem~\ref{Thm:UpperBoundShifted} hold and that $\inf_{\alpha_\bbS \in H_\bbS} \mathcal{L}(\alpha_\bbS) \geq V_0 >0$. Then we have that
\[
    d_\mathrm{K}\biggl( \frac{n^{1/2}(\hat{\theta}-\theta)}{\hat{V}^{1/2}} \mathbbm{1}_{\{\hat{V}>0\}}, Z \biggr) \leq \frac{7 A^{M/2} D}{V_0} \biggl\{ \frac{M T^d}{nh^d} + h^{2\beta_\wedge} + \mathbb{P}(\|X\|_\infty \geq T) \biggr\}^{1/4}.
\]
In particular, $\mathbb{P}(C_\alpha \ni \theta) \rightarrow 1-\alpha$ uniformly over $\alpha \in (0,1)$ whenever this right-hand side converges to zero.
\end{prop}

We see that the validity of these confidence intervals does not require large $M$. However, the proof of this result reveals that $\hat{V}$ is approximately equal to $n \mathrm{Var}(\theta^{*,(M)})$, which only converges to its minimal value as $M \rightarrow \infty$. Since the width of $C_\alpha$ is given by $2(\hat{V}/n)^{1/2} z_{\alpha/2}$, a larger value of $M$ results in asymptotically shorter confidence intervals. This result relies on the assumption that $\inf_{\alpha_\bbS \in H_\bbS} \mathcal{L}(\alpha_\bbS) \geq V_0 >0$, which is natural given that we must estimate and divide by $\inf_{\alpha_\bbS \in H_\bbS} \mathcal{L}(\alpha_\bbS)$ to standardise our estimator. The proof follows from Theorem~\ref{Thm:UpperBoundShifted} and a Berry--Esseen bound together with an analysis of the estimator $\hat{V}$.

\subsection{Numerical examples}
\label{Sec:Simulations}

We conclude this section with a brief simulation study\footnote{Code to reproduce the experiments is available on the author's website \texttt{thomasberrett.github.io}.} to demonstrate the feasibility and potential gains in accuracy of the proposed methodology. In each of three settings we suppose that our complete data are i.i.d.~copies of some $(X_1,X_2,Y)$ taking values in $[0,1]^2 \times \bbR$, that we have incomplete data of the form $(X_1,\texttt{NA},\texttt{NA})$ and $(\texttt{NA},X_2,\texttt{NA})$ under MCAR, and that we wish to estimate the expected response $\theta=\bbE(Y)$. We will compare the performance of the complete-case estimator $n^{-1} \sum_{i=1}^n Y_i$ to two estimators constructed to mimic the $M=1,2,3,4$ approximate oracle estimators~\eqref{Eq:ApproxOracle}, where one is the cross-fitting estimator constructed above and the second is a more direct approach similar to that introduced in Section~\ref{Sec:DirectEstimator} in the appendix. The estimators using $M=1$ are of a relatively standard type studied in semi-supervised settings~\citep[e.g.][]{yang2019combining,cannings2022correlation}, where the complete-case estimator is corrected by its conditional expectations, while the $M=2,3,4$ estimators exploit dependence between $X_1$ and $X_2$ to further reduce the asymptotic risk. For both approaches a bandwidth parameter $h=0.1$ is used with a Gaussian kernel in the \texttt{R} function \texttt{ksmooth}, and we choose step size $\eta=1$. As an additional baseline method we also compare to taking the empirical average of $Y$ over a multiply-imputed dataset, where Multivariate Imputation by Chained Equations (MICE) \citep{van2011mice} with default options is used for the imputation.

The data are generated as follows:
\begin{itemize}
    \item[(i)] For $\rho \in \{0,0.1,0.2,\ldots, 1\}$ let $(Z_1,Z_2)$ have standard normal marginal distributions with correlation $-\rho$ and set $(X_1,X_2) = (\Phi(Z_1), \Phi(Z_2))$, where $\Phi$ is the standard normal distribution function. Then let $Y|X \sim N(X_1-X_2,0.3^2)$.
    \item[(ii)] Let $(X_1,X_2)$ have $\mathrm{Unif}[0,1]$ marginal distributions and dependence structure given by the Clayton copula with parameter 3. For $\mu \in \{0,2,4,\ldots,20\}$ let $Y|X$ have a Bernoulli distribution with success probability $\{1+e^{-\mu(X_1+X_2-1)}\}^{-1}$.
    \item[(iii)] Generate $(X_1,X_2)$ as in Setting~(i). Then let $Y|X \sim N(-5X_1X_2, 0.3^2)$.
\end{itemize}
The complete data sample sizes are given by $n=200,600$ and the ratio of the incomplete samples sizes to $n$ is fixed at $
\lambda_1=\lambda_2=10$. Each point in the plots is the result of averaging over $10000$ repetitions of the experiment.

The results of these simulations are given in Figure~\ref{Fig:Plots}. The plots show the rescaled mean squared error against $\rho$ or $\mu$, as appropriate, with error bars showing three standard deviations. There is almost uniform improvement over the complete-case estimator, with the only exceptions being when $\mu=0,2$ in setting (ii) where the covariates $(X_1,X_2)$ are independent of or very weakly dependent on $Y$. In all cases we see that the direct estimators outperform the cross-fitting estimators, though this is less pronounced when $n=600$ as compared to $n=200$. Increasing $M$ tends to decrease the error of the estimators, as predicted by the asymptotic theory, though there are cases where the cross-fit estimators with a smaller value of $M$ perform better as samples are split into fewer pieces. Consequently, the cross-fit $M=3,4$ estimators are only shown in the $n=600$ results. While the plug-in estimator with MICE performs well in Setting (i), it is often outperformed in the other settings, with error beyond the range of the axes in some cases, mostly due to non-negligible bias; see Appendix~\ref{Sec:MICE} for further simulations on this point.

\begin{figure}
    \includegraphics[width=\textwidth]{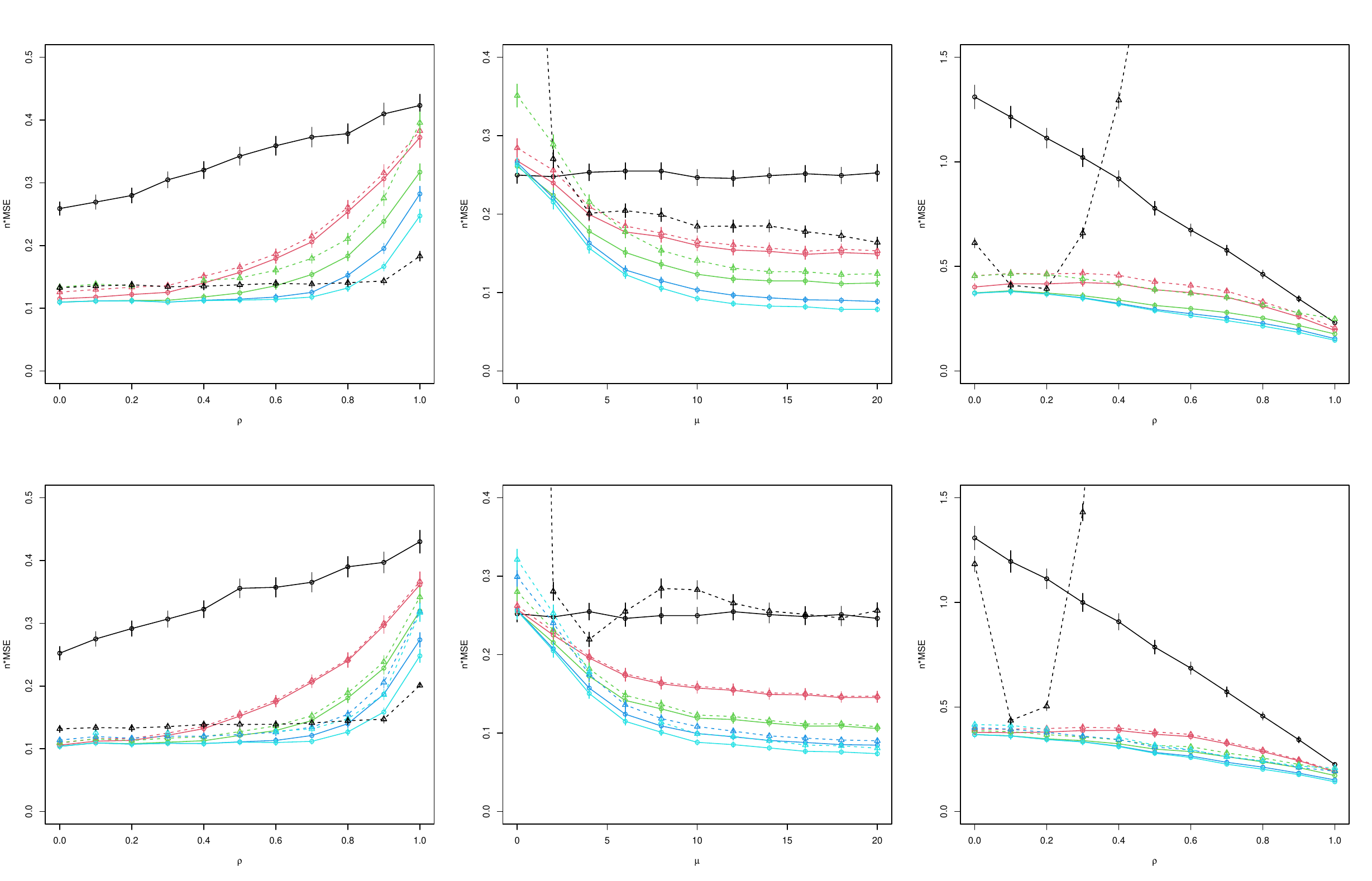}
    \caption{Estimated rescaled mean squared errors. Settings (i)--(iii) are presented from left to right, while the top row shows results for $n=200$ and the bottom row for $n=600$.  The solid black line shows results for the complete-case estimator and the dotted black line for MICE. Coloured lines show results for our proposed estimator, with $M=1,2,3,4$ corresponding to red, green, blue and cyan, respectively, solid lines corresponding to the direct approach without sample splitting and dotted lines corresponding to the cross-fit estimators. Error bars show three standard deviations.}
        \label{Fig:Plots}
\end{figure}

\section{Sampling-bias extensions}
\label{Sec:Extensions}

\subsection{Missing At Random}
\label{Sec:MAR}

It is known that when missingness structures are non-monotonic the development of models for mechanisms is challenging and that standard Missing At Random (MAR) assumptions are difficult to work with \citep{robins1997non}. In this section we work under a type of MAR assumption, sometimes referred to as Covariate Dependent Missingness \citep{li2013little}, where, conditionally on a set of fully-observed variables, the missingness is independent of the remaining variables.

We suppose that $X \sim f$, that $\Omega$ encodes the missingness with $\bbS = \{S \subseteq [d] : p_S >0\}$ being the collection of possible patterns and that we observe $N$ independent and identically distributed copies of $X \circ \Omega$. Here we assume that there exists $S_0 \subseteq [d]$ such that $S_0 \subseteq S$ for all $S \in \bbS$ and that we may write $\bbP(\Omega = \mathbbm{1}_S | X=x) = M_S(x_{S_0})$, so that the missingness mechanism $M_\bbS = (M_S : S \in \bbS)$ depends only on those variables in $S_0$, which are always observed.  We will see that our previous framework naturally extends to this MAR setting. Importantly, we do not need $M_\bbS$ to be known, only that it can be estimated sufficiently well. 

In order to state our results we write $H_S^{S_0}$ for the subspace of $H_S$ consisting of those elements $\alpha_S$ such that $\bbE\{\alpha_S(X_S) | X_{S_0}\}=0$ almost surely, we write $H_\bbS^{S_0} = \prod_{S \in \bbS} H_S^{S_0}$ and we define the new objective function
\[
    \mathcal{L}(\alpha_\bbS,M_\bbS) = \bbE \biggl[ \frac{1}{M_{[d]}(X_{S_0})} \biggl\{ a(X) - a_{S_0}(X_{S_0}) - \sum_{S \in \bbS} \alpha_S(X_S) \biggr\}^2 + \sum_{S \in \bbS} \frac{\alpha_S(X_S)^2}{M_S(X_{S_0})}  \biggr].
\]
Minimising this function over $\alpha_\bbS \in H_\bbS^{S_0}$ can be thought of as fitting generalised ANOVA decompositions separately for each value of $X_{S_0}$. The following result establishes a local asymptotic minimax lower bound in this MAR setting. This is proved very similarly to Theorem~\ref{Thm:ShiftedLowerBound}, with the main differences being the change in the direction of the worst-case perturbations and the fact that we now have random sample sizes within each missingness pattern. Although the missingness mechanism $M_\bbS$ may be unknown, it need not be perturbed to achieve a tight lower bound, and it suffices to consider perturbations of $f$ as in Section~\ref{Sec:ANOVALowerBound}.
\begin{thm}
\label{Thm:MARLowerBound}
Suppose that we observe $N$ i.i.d.~copies of $X \circ \Omega$, as described above, and that $\inf_{x_{S_0}} \min_{S \in \bbS_+} M_S(x_{S_0})>0$. For any measurable estimator sequence $(\theta_N)$ we have
    \[
        \sup_{I \subset L^2} \liminf_{N \rightarrow \infty} \max_{h \in I} \mathbb{E}_{f_{N^{-1/2} h}} \bigl[N \bigl\{ \theta_N - \theta(f_{N^{-1/2} h}) \}^2 \bigr] \geq \mathrm{Var} \bigl( a_{S_0}(X_{S_0}) \bigr) + \inf_{\alpha_\bbS \in H_\bbS^{S_0}} \mathcal{L}(\alpha_\bbS,M_\bbS),
    \]
where the supremum is taken over all finite subsets $I$ of square-integrable functions $h$ on $\bbR^d$. 
\end{thm}

In the remainder of this section we outline the construction of efficient estimators and prove a result under high-level assumptions similar to those that appear in the literature on double robustness. The form of our lower bound and the proof of Theorem~\ref{Thm:MARLowerBound} suggest that an appropriate oracle estimator is given by
\[
    \theta^* = \frac{1}{N} \sum_{i=1}^N \biggl[ a_{S_0}(X_i) + \frac{\mathbbm{1}_{\{\Omega_i=\mathbbm{1}_{[d]}\}}}{M_{[d]}(X_i)} \biggl\{ a(X_i) - a_{S_0}(X_i) - \sum_{S \in \bbS} \alpha_S^*(X_i) \biggr\} + \sum_{S \in \bbS} \frac{\mathbbm{1}_{\{\Omega_i=\mathbbm{1}_{S}\}}}{M_{S}(X_i)} \alpha_S^*(X_i) \biggr],
\]
which is unbiased and whose variance is equal to the lower bound. Suppose that we have access to independent data $\mathcal{D}$ with which we construct estimators $\hat{a}_{S_0}, \hat{M}_\bbS, \hat{\alpha}_\bbS$ of the nuisance functions in $\theta^*$. Using standard cross-fitting approaches we see that we do not need to make this assumption, as these quantities can be estimated on held-out halves of the data without losing efficiency, but for simplicity we do not make this explicit here. Both $a_{S_0}$ and $M_\bbS$ can be estimated using standard techniques, e.g. kernel nonparametric regression. One way to estimate $\alpha_\bbS$ would be to partition the sample space of $X_{S_0}$ into small bins and use the approach of Section~\ref{Sec:UpperBound} within each bin, though we omit the full and explicit construction of such an estimator. Here we simply assume that, if $X$ is independent of $\mathcal{D}$ and we write
\begin{align*}
    a_N &= \max \biggl\{ \bbE \bigl[ \bigl\{ \hat{a}_{S_0}(X_{S_0}) - a_{S_0}(X_{S_0}) \bigr\}^2 \bigr], \max_{S \in \bbS} \bbE \bigl[ \bigl\{ \hat{\alpha}_{S}(X_{S}) - \alpha_{S}^*(X_{S}) \bigr\}^2 \bigr] \biggr\} \\
    b_N &= \max_{S \in \bbS_+} \bbE \biggl[ \biggl\{ \frac{M_S(X_{S_0})}{\hat{M}_S(X_{S_0})} - 1 \biggr\}^2 \biggr],
\end{align*}
then we have $\max(a_N, b_N, N a_N b_N) \rightarrow 0$ as $N \rightarrow \infty$. Given such estimators we consider
\[
    \hat{\theta}= \frac{1}{N} \sum_{i=1}^N \biggl[ \hat{a}_{S_0}(X_i) + \frac{\mathbbm{1}_{\{\Omega_i=\mathbbm{1}_{[d]}\}}}{\hat{M}_{[d]}(X_i)} \biggl\{ a(X_i) - \hat{a}_{S_0}(X_i) - \sum_{S \in \bbS} \hat{\alpha}_S(X_i) \biggr\} + \sum_{S \in \bbS} \frac{\mathbbm{1}_{\{\Omega_i=\mathbbm{1}_{S}\}}}{\hat{M}_{S}(X_i)} \hat{\alpha}_S(X_i) \biggr].
\]
\begin{prop}
\label{Prop:MARDoubleRobustness}
Under the conditions of the previous paragraph and the assumptions that $\|a\|_\infty<\infty$ and $\inf_{x_{S_0}} \min_{S \in \bbS_+} M_S(x_{S_0})>0$, we have that
\[
    \hat{\theta} - \theta^* = O_p\biggl( \frac{a_N^{1/2}+b_N^{1/2}}{N^{1/2}} + (a_Nb_N)^{1/2} \biggr) = o_p(N^{-1/2})
\]
as $N \rightarrow \infty$.
\end{prop}
This result implies that such an estimator $\hat{\theta}$ has an asymptotic normal distribution whose variance is equal to the right-hand side of the lower bound in Theorem~\ref{Thm:MARLowerBound}, and is proved by controlling $\bbE|\hat{\theta}-\theta^*|$. By slightly strengthening the conditions we could similarly control the second moment of the difference in estimators and thus prove that the mean squared risk of $\hat{\theta}$ is asymptotically equivalent to the mean squared risk of $\theta^*$.

\subsection{Unknown shift functions}
\label{Sec:UnknownShifts}

Here we show how the distributional assumptions of our basic model can be relaxed in a different direction to that in Section~\ref{Sec:MAR}. As previously, we assume that we have access to complete data $X_1,\ldots,X_n \overset{\mathrm{i.i.d.}}{\sim} f$ from the target distribution and that we have incomplete data generated according to~\eqref{Eq:IncompleteData}. However, rather than assuming that $r_\bbS$ is known, we assume that 
\[
    r_S(x_S) = \exp \bigl( w_S^T \Psi_S(x_S) \bigr)
\]
for some unknown parameter vector $w_S \in \mathbb{R}^{k_S}$, some $k_S \in \bbN$ and some known functions $\Psi_S(x_S) = (\psi_{S,1}(x_S),\ldots,\psi_{S,k_S}(x_S))$. This is a well-studied parametric model for distributional shifts~\citep[e.g.][]{efron1996using,qin1998inferences,qiu2023efficient}. For simplicity, we will assume that the functions $\psi_{S,k}$ are all bounded, so that $\|\bar{r}_S\|_\infty < \infty$ as previously assumed. So that the model is identifiable, we assume that there are no $S \in \bbS$ and $w_S \in \bbR^{p_S}$ such that $w_S^T \Psi_S(\cdot)$ is a constant function. 

In order to state a lower bound we define the subspace $H_\bbS^\perp = \prod_{S \in \bbS} H_S^\perp$ of $H_\bbS$ given by
\[
    H_S^\perp = \bigl\{ \alpha_S \in H_S : \bbE\{ \alpha_S(X_S) \Psi_S(X_S)\} = 0 \bigr\}.
\]
Our next result shows that minimising our basic objective function $\mathcal{L}(\cdot)$, defined in~\eqref{Eq:Objective}, over $H_\bbS^\perp$ instead of $H_\bbS$ gives a lower bound in this setting. While in previous lower bounds it was sufficient to perturb $f$, here we also perturb the unknown parameter vectors $w_\bbS=(w_S : S \in \bbS)$, so we write $P_{n,t,v_\bbS}$ for the distribution of our data when the target distribution is $f_{n^{-1/2}t\alpha_\perp}$ and our distributional shifts are given by $\exp((w_S + n^{-1/2} v_S)^T \Psi_S(x_S))$ for $S \in \bbS$. We see here that, unlike the missingness mechanisms of the previous section, the parameters $w_\bbS$ generally cannot be estimated without inflating the variance in estimating $\theta$.
\begin{thm}
\label{Thm:UnknownShiftedLowerBound}
There exists a unique minimiser $\alpha_\bbS^\perp$ of $\mathcal{L}(\cdot)$ over $\alpha_\bbS \in H_\bbS^\perp$. Morever, writing $\alpha^\perp = a - \theta -\sum_S \alpha_S^\perp$ and $k_\bbS = \sum_{S \in \bbS}k_S$, for any estimator sequence $(\theta_n)$ we have
    \[
        \sup_{\substack{I \subset \bbR \\ J \subset \bbR^{k_\bbS} }} \liminf_{n \rightarrow \infty} \max_{\substack{t \in I \\ v_\bbS \in J}} \,\, \mathbb{E}_{P_{n,t,v_\bbS}} \bigl[ n \bigl\{ \theta_n - \theta(f_{n^{-1/2}t\alpha^\perp}) \}^2 \bigr] \geq \inf_{\alpha_\bbS \in H_\bbS^\perp} \mathcal{L}(\alpha_\bbS),
    \]
    where the supremum is taken over all finite subsets $I$ of $\bbR$ and finite subsets $J$ of $\bbR^{k_\bbS}$.
\end{thm}

In the remainder of this section we outline the construction of efficient estimators as in the previous section. Here, in light of the proof of Theorem~\ref{Thm:UnknownShiftedLowerBound} we aim to mimic the oracle estimator
\[
    \theta^* = \frac{1}{n} \sum_{i=1}^n \biggl\{ a(X_i) - \sum_{S \in \bbS} \alpha_S^\perp(X_i) \biggr\} + \sum_{S \in \bbS} \frac{1}{n_S} \sum_{j=1}^{n_S} \frac{\alpha_S^\perp(X_{S,j})}{\bar{r}_S(X_{S,j})},
\]
where we write $\bar{r}_S (x_S) = \exp(w_S^0 + w_S^T\Psi_S(x_S))$ and $w_S^0$ normalises the shift function. Suppose that we have access to independent data $\mathcal{D}$ with which we construct estimators $(\hat{w}_\bbS^0,\hat{w}_\bbS), \hat{\alpha}_\bbS^\perp$ of the nuisance parameters in $\theta^*$. Again, with cross-fitting one can achieve the same results without held-out data. The $(w_\bbS^0,w_\bbS)$ can be estimated using standard density ratio estimation techniques~\citep[see, e.g.][]{sugiyama2012density}. The $\alpha_\bbS^\perp$ can be estimated consistently using the gradient descent approach of Section~\ref{Sec:UpperBound}, modified to include projections onto $H_\bbS^\perp$ rather than just the spaces of mean-zero functions. We introduce the high-level assumptions that, if $X$ is independent of $\mathcal{D}$ and we write $\hat{r}_S(x_S)=\exp(\hat{w}_S^0 + \hat{w}_S^T \Psi_S(x_S))$,
\begin{align*}
    a_n &= \max_{S \in \bbS} \bbE \biggl[ \frac{1}{1 \wedge \bar{r}_S(X_S)} \bigl\{ \hat{\alpha}_{S}(X_{S}) - \alpha_{S}^\perp(X_{S}) \bigr\}^2 \biggr]  \\
    b_n &= \max_{S \in \bbS} \max \biggl\{  \bbE \biggl[ \frac{1 \vee \alpha_S^\perp(X_S)^2}{1 \wedge \bar{r}_S(X_S)} \biggl\{ \frac{\bar{r}_S(X_S)}{\hat{r}_S(X_S)} - 1 \biggr\}^2,  \\
    & \hspace{85pt} \biggl| \bbE \biggl[ \alpha_S^\perp(X_S) \biggl\{ \frac{\bar{r}_S(X_S)}{\hat{r}_S(X_S)} - 1 - w_S^0 + \hat{w}_S^0 - (w_S-\hat{w}_S)^T \Psi_S(X_S) \biggr\}  \biggr] \biggr| \biggr\},
\end{align*}
then we have $\max(a_n, b_n, na_n b_n, nb_n^2) \rightarrow 0$ as $n \rightarrow \infty$. Under appropriate regularity conditions, if the estimators of $(w_\bbS^0,w_\bbS)$ are $\sqrt{n}$-consistent then it suffices to have consistent estimators of $\alpha_\bbS^\perp$. We can now introduce the final estimator
\[
    \hat{\theta}= \frac{1}{n} \sum_{i=1}^n \biggl\{ a(X_i) - \sum_{S \in \bbS} \hat{\alpha}_S(X_i) \biggr\} + \sum_{S \in \bbS} \frac{1}{n_S} \sum_{j=1}^{n_S} \frac{\hat{\alpha}_S(X_{S,j})}{\hat{r}_S(X_{S,j})}.
\]
\begin{prop}
\label{Prop:UnknownDoubleRobustness}
Under the conditions of the previous paragraph we have that
\[
    \hat{\theta} - \theta^* = O_p\biggl( \frac{a_n^{1/2}+b_n^{1/2}}{n^{1/2}} + (a_nb_n)^{1/2} + b_n \biggr) = o_p(n^{-1/2})
\]
as $n \rightarrow \infty$.
\end{prop}

\section*{Acknowledgements}

The author was supported by Engineering and Physical Sciences Research Council New Investigator Award EP/W016117/1 and European Research Council Starting Grant 101163546.

{
\bibliographystyle{custom}
\bibliography{bib}
}

\appendix

\section{Estimation without sample splitting}
\label{Sec:DirectEstimator}

In this section we study a more direct estimator of $\theta(f)$ that does not rely on splitting our samples, designed for the MCAR case that $r_S \equiv 1$ for all $S \in \bbS$.  The analysis of such an estimator is technically more challenging and, as a result, we assume that our complete data takes values in $[0,1]^d$ and make further assumptions on $f$. While our work here concerns a more restricted setting, the explicit form of the estimator may be of independent interest. 

We will assume that there exist $\beta_1,\beta_2,L_1,L_2,c_0,C_0>0$ such that the following conditions hold.
\begin{enumerate}[align=left]
    \item[(B1)($\beta_1,L_1$)] For all $x,x' \in [0,1]^d$ we have
    \begin{equation}
    \label{Eq:aSmoothness}
    |a(x)-a(x')| \leq L_1 \|x-x'\|_\infty^{\beta_1}.
\end{equation}
    \item[(B2)($\beta_2,L_2$)] For all $x,x' \in [0,1]^d$ we have
\begin{equation}
\label{Eq:fSmoothness}
    |f(x)-f(x')| \leq L_2 \|x-x'\|_\infty^{\beta_2}.
\end{equation}
    \item[(B3)($c_0,C_0$)] For all $x \in [0,1]^d$ we have
\[
    c_0 \leq f(x) \leq C_0.
\]
\end{enumerate}
Assumption (B1)($\beta_1,L_1)$ is almost identical to the earlier (A1)($\beta_1,L_1)$, with the only difference being that it only requires smoothness on $[0,1]^d$ rather than all of $\bbR^d$. It follows from a simple calculation, similar to that in~\eqref{Eq:SmoothnessToTV}, that (B2)($\beta_2,L_2$) and (B3)($c_0,C_0$) together imply (A2)($\beta_2,L_2/c_0)$. We make these stronger assumptions to facilitate the analysis.

We now turn to the construction of our estimator. Let $K:\mathbb{R}^d \rightarrow [0,\infty)$ be a kernel function satisfying $\int K=1$, $K(u)=0$ when $\|u\|_\infty >1/2$, and $K(\xi_1u_1,\ldots,\xi_du_d)=K(u_1,\ldots,u_d)$ for any $\xi_1,\ldots,\xi_d \in \{-1,1\}$ and $u \in \mathbb{R}^d$. Given $S \in \bbS$ and a bandwidth $h \in (0,1)$ write $K^S(u)= \int_{\mathbb{R}^{S^c}} K(u_S,u_{S^c}') \,du_{S^c}'$ and $K_h^S(u)=h^{-|S|}K^S(u/h)$, as in the previous section. For $x \in \bbR^d$, $h \in (0,1)$ and $S \in \bbS$ we define the marginal density estimator
\[
    \hat{f}_{S,h}(x) \equiv \hat{f}_S(x) = \frac{1}{n+n_S} \biggl\{\sum_{i=1}^{n_S} K_h^S(X_{S,i} - x) + \sum_{i=1}^n K_h^S(X_i - x) \biggr\}.
\]
Given $\bS \in \bbS^{(m)}$ we write $v_\bS = \prod_{j=1}^m (1+n/n_{S_j})^{-1}$ and define the random function $\hat{k}_h^\bS$ on $(\bbR^d)^{m+1}$ by
\[
    \hat{k}_{h}^{\bS}(x_1,\ldots,x_{m+1}) = a(x_1)\biggl\{ \frac{K_h^{S_1}(x_2-x_1)}{\hat{f}_{S_1}(x_1)} - 1 \biggr\} \ldots \biggl\{ \frac{K_h^{S_m}(x_{m+1}-x_m)}{\hat{f}_{S_m}(x_m)} - 1 \biggr\}.
\]
Given a choice of truncation level $M \in \mathbb{N}_0$ and bandwidth $h \in (0,1)$ we take step size $\eta = |\bbS|^{-1}$ as in Section~\ref{Sec:UpperBound} and define our estimator to be
\begin{align*}
    \check{\theta}_{h}^{M} = \sum_{m=0}^{M} (-1)^m \sum_{\bS \in \bbS^{(m)}} v_{(S_1,\ldots,S_{m-1})} b_{M,\eta}(m)  \frac{1}{(n)_{m+1}} \sum_{\bi \in \mathcal{I}_{m+1}}\hat{k}_{h}^{\bS}(X_{i_1},\ldots,X_{i_{m+1}}),
\end{align*}
where we interpret the $m=0$ term as $\hat{\theta}^\mathrm{CC}$. The following is our main result of this section.
\begin{thm}
\label{Thm:UpperBound}
Let $\beta_1,\beta_2 \in (0,1]$ and $L_1,L_2,c_0,C_0>0$ and suppose that $a$ and $f$ satisfy (B1)($\beta_1,L_1$), (B2)($\beta_2,L_2$) and (B3)($c_0,C_0$). Suppose further that $\max_{S \in \bbS} \lambda_S < \infty$ and that $\max_{S \in \bbS} \|\bar{r}_S\|_\infty < \infty$. Then there exists a constant $C \equiv C(d,\beta_1,\beta_2,L_1,L_2,c_0,C_0,\|K\|_\infty)$ such that
\[
    \bbE_f [ n \{\check{\theta}_{h}^{M} - \theta(f)\}^2 ] \rightarrow \inf_{\alpha_\bbS \in H_\bbS} \mathcal{L}(\alpha_\bbS)
\]
whenever $h$ and $M$ are chosen to satisfy $M \rightarrow \infty$, $hC^M \rightarrow 0$ and $(C/h^d)^{2M+2}/n \rightarrow 0$.
\end{thm}
This result is a consequence of the finite-sample results Propositions~\ref{Prop:SmoothedUpperBound} and~\ref{Prop:FiniteMUpperBound} below and Proposition~\ref{Prop:OptimisationShifted} and the second part of Proposition~\ref{Prop:CrossFitReduction} above. The requirement on the bandwidth $h$ is stronger than for our sample-splitting estimator in Section~\ref{Sec:UpperBound}, but it is sufficient to prove that our direct estimator achieves the minimal asymptotic variance.

Theorem~\ref{Thm:UpperBound} is proved by showing that our estimator can be approximated successively by two oracle estimators. The second of these is $\theta^{*,(M)}$, defined previously in~\eqref{Eq:ApproxOracle}, but the first is new. In order to introduce this new oracle estimator we must introduce some population-level quantities. We will write $f_{S,h}(x) \equiv f_{S,h}(x_S) = \bbE \{\hat{f}_{S,h}(x)\} = (K_h^{S} \ast f)(x)$ for a smoothed version of the marginal density $f_S$. Given $\bS \in \bbS^{(m)}$ we define an oracle version of $\hat{k}_h^\bS$ by
\[
    k_h^{\bS}(x_1,\ldots,x_{m+1}) = a(x_1)\biggl\{ \frac{K_h^{S_1}(x_2-x_1)}{f_{S_1,h}(x_1)} - 1 \biggr\} \ldots \biggl\{ \frac{K_h^{S_m}(x_{m+1}-x_m)}{f_{S_m,h}(x_m)} - 1 \biggr\}
\]
and further write $\bar{k}_h^{\bS}(x) = \mathbb{E}\{ k_h^{\bS}(X_1,\ldots,X_m,x)\}$. We may now define the oracle statistic
\[
    \theta_{h}^{*,M}= \hat{\theta}^\mathrm{CC} + \sum_{m=1}^{M} (-1)^m \sum_{\bS \in \bbS^{(m)}} v_\bS b_{M,\eta}(m) \biggl\{ \frac{1}{n} \sum_{i=1}^n \bar{k}_h^{\bS}(X_i) - \frac{1}{n_{S_m}} \sum_{i=1}^{n_{S_m}} \bar{k}_h^{\bS}(X_{S_m,i}) \biggr\}.
\]
Our next result shows that this indeed provides an approximation to $\check{\theta}_{h}^{M}$.
\begin{prop}
\label{Prop:SmoothedUpperBound}
Suppose that there exists $c_0>0$ such that $f(x) \geq c_0$ for all $x \in [0,1]^d$ and suppose that
\begin{equation}
\label{Eq:nlowerbound}
    n \geq  36 (M+1)^2 \biggl(\frac{2^{d+2} \|K\|_\infty}{c_0h^d}\biggr)^2 \log\biggl( \frac{(M+1) 2^{d+2} \|K\|_\infty}{c_0h^d} \biggr).
\end{equation}
Then there exists a universal constant $C>0$ such that
\[
    \mathbb{E}\{(\check{\theta}_{h}^{M} -\theta_{h}^{*,M})^2\} \leq \frac{\|a\|_\infty^2}{\min_{S \in \mathbb{S}^+} n_S^2} \biggl( \frac{C 2^d \|K\|_\infty}{c_0h^d} \biggr)^{2M+2}.
\]
\end{prop}
This result is proved by studying the stochastic error of our density estimators and using techniques from the theory of $U$-statistics. This does not require any smoothness conditions on $a$ or $f$, though we do assume that $f$ is bounded below by a positive constant. The functions $k_h^{\bS}$ introduced above are the kernels of degenerate $U$-statistics considered in this proof. The degeneracy is seen by noting that $\mathbb{E}\{ k_h^{\bS}(x_1,\ldots,x_m,X)\} = 0$ for any $x_1,\ldots,x_m \in \bbR^d$, which will mean that these $U$-statistics can be approximated by sample means.

We next study the behaviour of $\theta_{h}^{*,M}$ for small values of $h$, which will naturally require smoothness assumptions on $a$ and $f$; we use those introduced in~\eqref{Eq:aSmoothness} and~\eqref{Eq:fSmoothness} above. We will thus see that we can approximate $\theta_{h}^{*,M}$ by $\theta^{*,(M)}$ defined in~\eqref{Eq:ApproxOracle} and taking the form
\[
    \theta^{*,(M)} = \hat{\theta}^\mathrm{CC} + \sum_{m=1}^{M} (-1)^m \sum_{\bS \in \bbS^{(m)}} b_{M,\eta}(m) \biggl\{ \frac{1}{n} \sum_{i=1}^n \bar{a}_{\bS}^{(m)}(X_i) - \frac{1}{n_{S_m}} \sum_{i=1}^{n_{S_m}} \bar{a}_{\bS}^{(m)}(X_{S_m,i}) \biggr\}.
\]
when $r_S \equiv 1$ for all $S \in \bbS$.
\begin{prop}
\label{Prop:FiniteMUpperBound}
Let $\beta_1,\beta_2 \in (0,1]$ and $L_1,L_2,c_0,C_0>0$ and suppose that $a$ and $f$ satisfy (B1)($\beta_1,L_1$), (B2)($\beta_2,L_2$) and (B3)($c_0,C_0$), suppose that~\eqref{Eq:nlowerbound} holds and that $|n\lambda_S /n_S - 1| \leq 1/n$ for all $S \in \bbS$. Then
\[
    \mathbb{E}\{( \theta_{h}^{*,M} - \theta^{*,(M)})^2\} \leq \frac{72 M^4}{ \min_{S \in \bbS^+} n_S}  \biggl(\frac{2^{d+1/2}C_0}{c_0} \biggr)^{2M} \{C_0^2 L_1^2 h^{2\beta_1}+(L_2^2 h^{2\beta_2} + C_0 d h)\|a\|_\infty^2 + n^{-2}\}.
\]
\end{prop}
The proof of this result relies on the approximation $v_\bS \bar{k}_h^{\bS}(x) \approx \bar{a}_\bS^{(m)}(x)$ that we will formalise in~\eqref{Eq:kSSmoothness} under our smoothness assumptions.

\section{Additional simulation results for MICE}
\label{Sec:MICE}

In this section we present further simulation results to investigate the poor performance of the imputation-based estimator in Section~\ref{Sec:Simulations}. We consider the same data-generating mechanisms as previously, but restrict attention to the $n=200$ settings. Here we break down the mean squared error into squared bias and variance. The best performance is in Setting~(i), where there is no bias and variance is roughly constant apart from a small increase when the covariates are singular. There is a bias in Setting (ii) that grows with $\mu$. In the extreme case that $\mu=0$ the response variable is independent of the covariates and we see a large variance. In Setting~(iii), the dominant contribution is from the bias; we omit the point from $\rho=1$ for the readability of the plot.

\begin{figure}
     \begin{subfigure}
         \centering
         \includegraphics[width=0.32\textwidth]{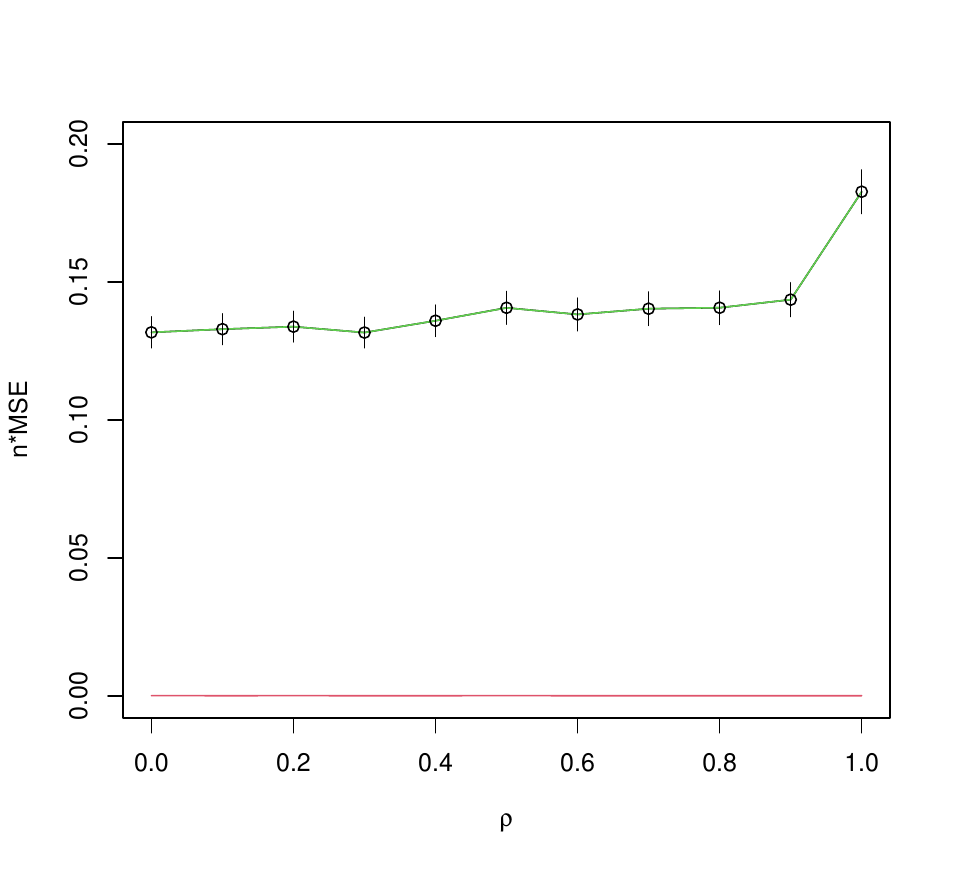}
     \end{subfigure}
     \hfill
     \begin{subfigure}
         \centering
         \includegraphics[width=0.32\textwidth]{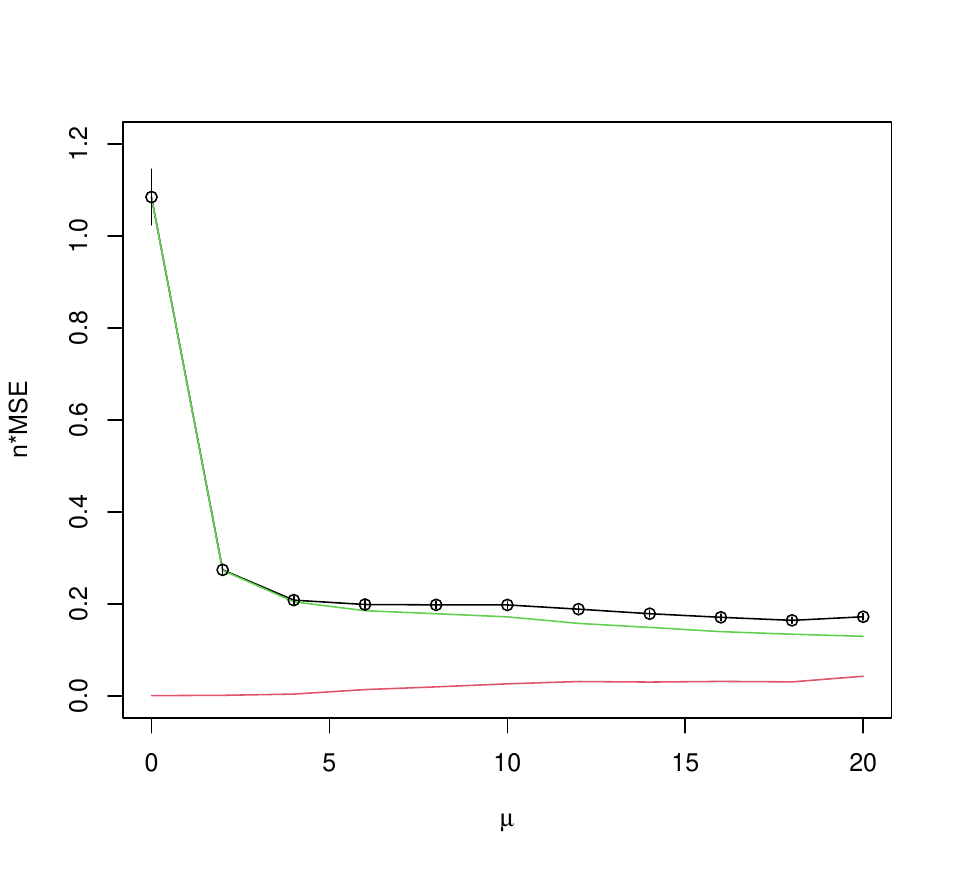}
     \end{subfigure}
     \hfill
     \begin{subfigure}
         \centering
         \includegraphics[width=0.32\textwidth]{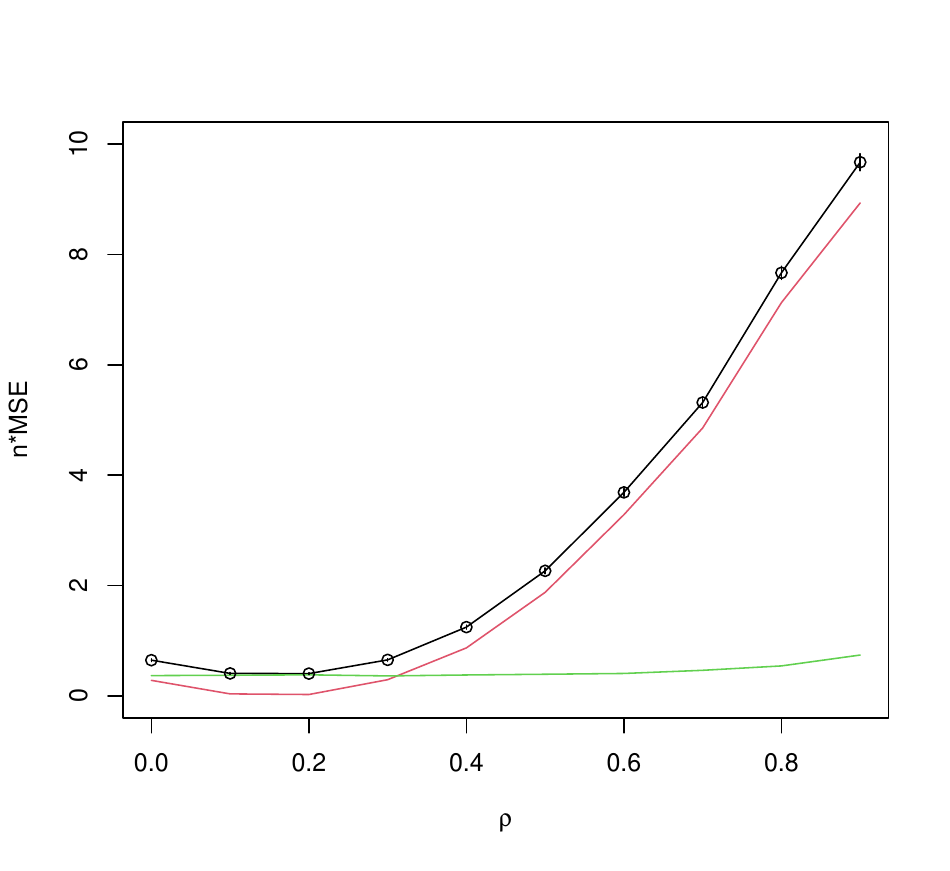}
     \end{subfigure}
        \caption{Settings (i)--(iii) are shown from left to right. Black lines show total rescaled mean squared error, while red lines show the contribution of the squared bias and green lines show that of the variance. Error bars show three standard deviations.}
        \label{Fig:Gaussian}
\end{figure}

\section{Non-linear functionals}
\label{Sec:NonLinear}

In this section we demonstrate that our framework extends to cover a much larger range of statistical tasks. We present three classes of examples, in each case providing an expression for the local asymptotic minimax lower bound and a sketch of the argument used to derive it. These are given in terms of generalised ANOVA decompositions of appropriate influence functions, which in these cases depend on unknown quantities. Following similar cross-fitting procedures to those used in Sections~\ref{Sec:UpperBound} and~\ref{Sec:MAR} it would be possible to construct efficient estimators; for the sake of brevity we focus on the lower bounds in this section.

\subsection{Smooth functionals}

Suppose that we are given data generated according to the basic model described in the introduction and wish to estimate functionals of the form
\[
    \theta(f) = \int \psi(x,f(x)) \,dx
\]
where $\psi$ is a smooth, known function. This includes the quadratic functional $\int f^2$ and the Shannon entropy $\int f \log(1/f)$, as well as other entropies, and has been studied as a class of estimands in \cite{laurent1996efficient}, for example. As the model for the data is unchanged, the same arguments as in Section~\ref{Sec:LowerBound} allow us to demonstrate the local asymptotic normality of our sequence of experiments with respect to the inner product derived in~\eqref{Eq:InnerProduct}. Now however, writing $\dot{\psi}(x,u)=(\partial \psi /\partial u)(x,u)$, we must find the adjoint of the linear map $\dot{\kappa} : h \mapsto \int f \dot{\psi} h$. It follows from our previous work, replacing $a$ by $\dot{\psi}$, that we will have an efficiency bound given by
\begin{equation}
\label{Eq:ANOVAOfInfluence}
    \inf_{\alpha_\bbS \in H_\bbS} \biggl\{ \mathrm{Var} \biggl( \dot{\psi} \bigl( X,f(X) \bigr) - \sum_{S \in \bbS} \alpha_S(X_S) \biggr) + \sum_{S \in \bbS} \bbE \biggl[ \frac{\alpha_S(X_S)^2}{\lambda_S \bar{r}_S(X_S)} \biggr] \biggr\}, 
\end{equation}
generalising Theorem~\ref{Thm:ShiftedLowerBound} beyond the case that $\psi(x,u)=a(x)u$.

\subsection{$U$-statistic targets}
\label{Sec:UStat}

Another commonly-studied class of estimands we can treat with our framework is that of the expected values of $U$-statistics. For $K \in \bbN$ and a (known) function $a : (\bbR^d)^K \rightarrow \bbR$, consider
\[
    \theta(f) = \bbE\{a(X_1,\ldots,X_K)\} = \int \ldots \int a(x_1,\ldots,x_K) f(x_1)\ldots f(x_K) \,dx_1 \,\ldots \,dx_K.
\]
This includes variances $\bbE\{\|X_1-X_2\|^2/2\}=\mathrm{tr}\, \mathrm{Cov}(X)$, goodness-of-fit statistics such as the kernel maximum mean discrepancy and measures of dependence such as Kendall's $\tau$, and has been previously studied in the semi-supervised learning setting by~\cite{cannings2022correlation} and~\cite{kim2024semi}. Writing $a_\ell(x)=\bbE\{a(X_1,\ldots,X_K) | X_\ell=x\}$, it is straightforward to see that for such functionals we have $n^{1/2}\{ \theta(f_{n^{-1/2}h}) - \theta(f)\} \rightarrow \sum_{\ell=1}^K \int f a_\ell h$. It therefore follows that the efficiency bound here is given by
\[
    \inf_{\alpha_\bbS \in H_\bbS} \biggl\{ \mathrm{Var} \biggl( \sum_{\ell=1}^K a_\ell(X) - \sum_{S \in \bbS} \alpha_S(X_S) \biggr) + \sum_{S \in \bbS} \bbE \biggl[ \frac{\alpha_S(X_S)^2}{\lambda_S \bar{r}_S(X_S)} \biggr] \biggr\},
\]
when this is positive. Similarly to~\eqref{Eq:ANOVAOfInfluence} in the previous section, we have the generalised ANOVA decomposition of the usual efficient influence function.

\subsection{Two-sample functionals}

In statistical analyses involving the comparison of distinct populations, such as two-sample testing and the estimation of treatment effects, it is common to estimate functionals of several distributions. Suppose that, as well as the data in our basic model, we have access to independent data with a similar structure on a second distribution of interest. For a density function $g$ on $\bbR^{d'}$ for some $d' \in \bbN$, a collection of subsets $\bbT \subseteq 2^{[d']}$, a collection of known distributional shifts $(s_T : T \in \bbT)$ and sample sizes $m$ and $(m_T : T \in \bbT)$, suppose that we observe 
\[
    Y_1,\ldots,Y_m \overset{\mathrm{i.i.d.}}{\sim} g  \quad \text{and} \quad Y_{T,1},\ldots,Y_{T,n_T} \overset{\mathrm{i.i.d.}}{\sim} \bar{s}_T g_T \text{ for } T \in \mathbb{T},
\]
where $g_T$ is the marginal density of $g$ on $\bbR^T$ and $\bar{s}_T(y_T) = s_T(y_T)/\int s_T g$ is a normalised shift function. Simple functionals of interest have the form $\theta(f,g)=\bbE\{a(X,Y)\}$ for a known function $a$. Taking $a(x,y)=\mathbbm{1}_{\{x_1\leq y_1\}}$, for example, gives the Mann--Whitney $U$ statistic for a difference in location between the first components of $X$ and $Y$. %Just as covariate information can be used to improve the estimation of univariate mean responses, here covariate estimation can be used to improve the estimation of a univariate measure of dependence. We could also take $a(x,y)=x^T y$ in order to estimate $\bbE(X)^T \bbE(Y)$.

We can establish local asymptotic minimax theory by perturbing of $f$ and $g$ to densities proportional to $k(n^{-1/2} h(x))f(x)$ and $k(n^{-1/2} \ell(x))g(x)$, respectively, where we recall the definition of $k$ from Section~\ref{Sec:ANOVALowerBound} and where $h,\ell \in L^2$ have mean zero. Since the datasets are independent, local asymptotic normality follows from earlier calculations with inner product
\begin{align*}
         \langle (h,\ell),(h',\ell') \rangle_{\lambda_\bbS,\mu_\bbT} &=  \mathbb{E}\{ h(X)h'(X)\} + \sum_{S \in \bbS} \lambda_S \Cov \bigl( h_S(X_{S,1}), h_S'(X_{S,1}) \bigr) \\
         & \hspace{50pt} + \tau \bbE\{ \ell(Y) \ell'(Y)\} + \tau \sum_{T \in \bbT} \mu_S \Cov \bigl( \ell_T(Y_{T,1}), \ell_T'(Y_{T,1}) \bigr) \\
         &= \langle h,h' \rangle_{\lambda_\bbS} + \tau \langle \ell,\ell' \rangle_{\mu_\bbT},
\end{align*}
where we write $\mu_T = \lim_{n \rightarrow \infty} m_T/m$ and assume that $m/n \rightarrow \tau \in (0,\infty)$. As in Section~\ref{Sec:UStat} we see that the derivative of $\theta$ with respect to these perturbations is given by $\dot{\kappa} : (h,\ell) \mapsto \int a_X hf + \int a_Y \ell g$, where $a_X(x) = \bbE\{a(x,Y)\}$ and $a_Y(y)=\bbE\{a(X,y)\}$. Using the work of Section~\ref{Sec:LowerBound} we may therefore calculate a lower bound by writing
\begin{align*}
    \sup &\biggl\{ \dot{\kappa}(h,\ell) : \|(h,\ell)\|_{\lambda_\bbS,\mu_\bbT} \leq 1 \biggr\} = \sup \biggl\{ \int a_X hf + \int a_Y \ell g : \|h\|_{\lambda_\bbS}^2 + \tau \|\ell\|_{\mu_\bbT}^2 \leq 1\biggr\} \\
    %&= \sup_{t \in [0,1]} \biggl[ t^{1/2} \sup \biggl\{ \int a_X hf : \|h\|_{\lambda_\bbS} \leq 1  \biggr\} + \{(1-t)/\tau\}^{1/2} \sup \biggl\{ \int a_Y \ell g : \|\ell\|_{\mu_\bbT} \leq 1  \biggr\} \biggr] \\
    & = \biggl[ \sup \biggl\{ \biggl( \int a_X hf \biggr)^2 : \|h\|_{\lambda_\bbS} \leq 1  \biggr\} + \frac{1}{\tau} \sup \biggl\{ \biggl(\int a_Y \ell g \biggr)^2 : \|\ell\|_{\mu_\bbT} \leq 1  \biggr\} \biggr]^{1/2} \\
    & = \inf_{\alpha_\bbS,\beta_\bbT} \biggl\{ \mathrm{Var} \biggl( a_X(X) - \sum_{S \in \bbS} \alpha_S(X_S) \biggr) + \frac{1}{\tau} \mathrm{Var} \biggl( a_Y(Y) - \sum_{T \in \bbT} \beta_T(Y_T) \biggr) \\
    & \hspace{150pt} + \sum_{S \in \bbS} \bbE \biggl[ \frac{\alpha_S(X_S)^2}{\lambda_S \bar{r}_S(X_S)} \biggr] + \sum_{T \in \bbT} \bbE \biggl[ \frac{\beta_T(Y_T)^2}{\tau \mu_T \bar{s}_T(Y_T)} \biggr] \biggr\}^{1/2}.
\end{align*}
This shows that the minimal variance in this problem is found by performing two separate generalised ANOVA decompositions.

\section{Proofs of main results}
\label{Sec:Proofs}

\subsection{Proofs for Section~\ref{Sec:LowerBound}}

\begin{proof}[Proof of Proposition~\ref{Prop:UniqueMinimiser}]
To simplify notation, we will sometimes write, for example, $\alpha_S$ for the random variable $\alpha_S(X_S)$ when this is clear from context. Since $\mathcal{L}(0) = \mathrm{Var} \{a(X)\} < \infty$ and we aim to minimise $\mathcal{L}$, we may restrict attention to $\alpha_\bbS \in H_\bbS$ such that $\alpha_S/\bar{r}_S^{1/2} \in L^2$ for all $S \in \bbS$. 

We first prove the existence of a minimiser. By elementary quadratic expansions we see that for $\mu \in [0,1]$ and $\alpha_\bbS, \beta_\bbS \in H_\bbS$ we have
\begin{align}
\label{Eq:ConvexExpansion}
    &\mathcal{L}(\mu \alpha_\bbS + (1-\mu) \beta_\bbS) \nonumber \\
    & = \Var\biggl( \mu \biggl\{ a(X) - \sum_{S \in \bbS} \alpha_S(X_S) \biggr\} + (1-\mu) \biggl\{ a(X) - \sum_{S \in \bbS} \beta_S(X_S) \biggr\} \biggr) \nonumber \\
    & \hspace{250pt} + \sum_{S \in \bbS} \int \frac{f_S}{\lambda_S \bar{r}_S} \{\mu \alpha_S + (1-\mu) \beta_S\}^2 \nonumber \\
    & = \mu \mathcal{L}(\alpha_\bbS) + (1-\mu) \mathcal{L}(\beta_\bbS) \nonumber \\
    & \hspace{50pt} + \mu(1-\mu) \biggl\{ 2\mathrm{Cov} \biggl( a - \sum_{S \in \bbS} \alpha_S ,a - \sum_{S \in \bbS} \beta_S \biggr) + 2 \sum_{S \in \bbS} \int \frac{f_S \alpha_S \beta_S}{\lambda_S \bar{r}_S} - \mathcal{L}(\alpha_\bbS) - \mathcal{L}(\beta_\bbS)\biggr\} \nonumber \\
    & = \mu \mathcal{L}(\alpha_\bbS) + (1-\mu) \mathcal{L}(\beta_\bbS) \nonumber \\
    & \hspace{50pt} - \mu(1-\mu) \biggl\{\mathrm{Var} \biggl(\sum_{S \in \bbS} \alpha_S - \sum_{S \in \bbS} \beta_S \biggr) + \sum_{S \in \bbS} \int \frac{f_S}{\lambda_S \bar{r}_S}(\alpha_S - \beta_S)^2 \biggr\} \\
    & \leq \mu \mathcal{L}(\alpha_\bbS) + (1-\mu) \mathcal{L}(\beta_\bbS) \nonumber
\end{align}
so that $\mathcal{L}$ is convex. Now let $(\alpha_\bbS^{(t)} : t \in \bbN)$ be a sequence of elements of $H_\bbS$ such that $\mathcal{L}(\alpha_\bbS^{(t)}) \rightarrow \inf_{\alpha_\bbS \in H_\bbS} \mathcal{L}(\alpha_\bbS)$ as $t \rightarrow \infty$. By the convexity of $\mathcal{L}$ we have that $\mathcal{L}(\alpha_\bbS^{(s)}/2 + \alpha_\bbS^{(t)}/2) \leq (1/2) \{ \mathcal{L}(\alpha_\bbS^{(s)}) +\mathcal{L}(\alpha_\bbS^{(t)})\} \rightarrow \inf_{\alpha_\bbS \in H_\bbS} \mathcal{L}(\alpha_\bbS)$ as $s,t \rightarrow \infty$, and hence that $\mathcal{L}(\alpha_\bbS^{(t)}) - \mathcal{L}(\alpha_\bbS^{(s)}/2 + \alpha_\bbS^{(t)}/2) \rightarrow 0$ as $s,t \rightarrow \infty$. However, we see by~\eqref{Eq:ConvexExpansion} that
\begin{align*}
    \mathcal{L}(\alpha_\bbS^{(t)}) - \mathcal{L}\biggl( \frac{\alpha_\bbS^{(s)} + \alpha_\bbS^{(t)}}{2} \biggr) &- \frac{1}{2} \mathcal{L}(\alpha_\bbS^{(t)}) + \frac{1}{2} \mathcal{L}(\alpha_\bbS^{(s)}) \\
    &= \frac{1}{4} \Var\biggl(  \sum_{S \in \bbS}(\alpha_S^{(s)} - \alpha_S^{(t)}) \biggr) + \frac{1}{4} \sum_{S \in \bbS} \int \frac{f_S}{\lambda_S \bar{r}_S} (\alpha_S^{(t)} - \alpha_S^{(s)})^2.
\end{align*}
As we know that the left-hand side of this equation converges to zero as $s,t \rightarrow \infty$, and we know that $\lambda_S < \infty$ for all $S \in \bbS$, it must be that $\int f_S (\alpha_S^{(t)} - \alpha_S^{(s)})^2/\bar{r}_S \rightarrow 0$ for each $S \in \bbS$. Since $\|\bar{r}_S\|_\infty<\infty$ we also know that $\int f_S (\alpha_S^{(t)} - \alpha_S^{(s)})^2 \rightarrow 0$ for each $S \in \bbS$. Thus, $\alpha_S^{(t)}/\bar{r}_S^{1/2}$ and $\alpha_S^{(t)}$ define Cauchy sequences in $H_S$ for each $S \in \bbS$. By the completeness of these spaces there exists $\alpha_\bbS^*$ such that $\alpha_S^{(t)} \rightarrow \alpha_S^*$ and $\alpha_S^{(t)}/\bar{r}_S^{1/2} \rightarrow \alpha_S^*/\bar{r}_S^{1/2}$ as $t \rightarrow \infty$ for each $S \in \bbS$, where the boundedness of $\bar{r}_S$ is used to see that the same $\alpha_S^*$ appears in both limits. Now by the continuity of $\mathcal{L}$ it is clear that $\mathcal{L}(\alpha_\bbS^*) = \lim_{t \rightarrow \infty} \mathcal{L}(\alpha_\bbS^{(t)})$, so that $\alpha_\bbS^*$ is a minimiser of $\mathcal{L}$.

We now justify the uniqueness of $\alpha_\bbS^*$. Suppose that $\beta_\bbS^*$ is another minimiser. Then by~\eqref{Eq:ConvexExpansion} we have
\[
    \mathcal{L}(\alpha_\bbS^*) \leq \mathcal{L}\Bigl( \frac{\alpha_\bbS^* + \beta_\bbS^*}{2} \Bigr) \leq \mathcal{L}(\alpha_\bbS^*) - \frac{1}{4} \sum_{S \in \bbS} \int \frac{f_S}{\lambda_S \bar{r}_S}(\alpha_S^* - \beta_S^*)^2,
\]
which implies that $\alpha_\bbS^* = \beta_\bbS^*$.
\end{proof}

\begin{proof}[Proof of Proposition~\ref{Prop:ProductCase}]
We start by giving a useful characterisation of minimisers of $\mathcal{L}(\alpha_\bbS)$. Specialising~\eqref{Eq:Stationarity0} and~\eqref{Eq:Stationarity} to the $r_S \equiv 1$ case we see that $\alpha_\bbS$ is a minimiser of $\mathcal{L}$ if and only if
\begin{equation}
\label{Eq:Stationarity2}
    \bbE \biggl\{ a(X) -\theta - \sum_{S' \in \bbS \setminus \{S\}} \alpha_{S'}(X_{S'}) \biggm| X_S \biggr\} =(1+1/\lambda_S)\alpha_S(X_S)
\end{equation}
almost surely, for each $S \in \bbS$. We prove our result by checking these stationarity conditions at our claimed minimiser defined in~\eqref{Eq:ProductCase}. Throughout the proof we use the shorthand $a_T$ and $\tilde{a}_T$ for $a_T(X_T)$ and $\tilde{a}_T(X_T)$, respectively, when it does not cause confusion.

We begin by proving that $\mathbb{E}(\tilde{a}_T | X_U) = 0$ for all $U \subset T$. This is related to standard ANOVA decompositions, and is well known, but we include a short inductive proof for completeness. In the base cases $T=\{j\}$ we simply have $\tilde{a}_T = a_{\{j\}} - \theta$ and clearly $\bbE(\tilde{a}_T | X_\emptyset ) = \bbE(\tilde{a}_T) = 0$. For the induction step, using the independence of the entries of $X$ and the induction hypothesis, we have for any $T \subseteq [d]$ and $U \subset T$ that
\begin{align*}
    \bbE(\tilde{a}_T | X_U) = \bbE \biggl( a_T - \sum_{U' \subset T} \tilde{a}_{U'} \biggm| X_U \biggr) &= a_U - \sum_{U' \subset T} \bbE(\tilde{a}_{U'} | X_{U \cap U'}) \\
    & = a_U - \sum_{U' \subseteq U} \tilde{a}_{U'} = \tilde{a}_U - \tilde{a}_U = 0,
\end{align*}
as claimed.

We now proceed to check that the $\alpha_\bbS$ defined in~\eqref{Eq:ProductCase} satisfies~\eqref{Eq:Stationarity2}. For each $S \in \bbS$ we have
\begin{align*}
    &(1+1/\lambda_S) \alpha_S + \sum_{S' \neq S} \bbE(\alpha_{S'} | X_S) \\
    &= \frac{1}{\lambda_S} \sum_{\emptyset \neq T \subseteq S} \frac{\lambda_S}{1+\sum_{S' \in \bbS : T \subseteq S'} \lambda_{S'}} \tilde{a}_T + \sum_{S' \in \bbS} \sum_{\emptyset \neq T \subseteq S'} \frac{\lambda_{S'}}{1+\sum_{S'' \in \bbS : T \subseteq S''}\lambda_{S''}} \bbE(\tilde{a}_T | X_S) \\
    &= \frac{1}{\lambda_S} \sum_{\emptyset \neq T \subseteq S} \frac{\lambda_S}{1+\sum_{S' \in \bbS : T \subseteq S'} \lambda_{S'}} \tilde{a}_T + \sum_{S' \in \bbS} \sum_{\emptyset \neq T \subseteq S \cap S'} \frac{\lambda_{S'}}{1+\sum_{S'' \in \bbS : T \subseteq S''}\lambda_{S''}} \tilde{a}_T \\
    &= \sum_{\emptyset \neq T \subseteq S} \frac{1}{1+\sum_{S' \in \bbS : T \subseteq S'} \lambda_{S'}} \tilde{a}_T + \sum_{\emptyset \neq T \subseteq S } \frac{\sum_{S' \in \bbS : T \subseteq S'} \lambda_{S'}}{1+\sum_{S'' \in \bbS : T \subseteq S''}\lambda_{S''}} \tilde{a}_T \\
    & = \sum_{\emptyset \neq T \subseteq S} \tilde{a}_T = a_S - \theta,
\end{align*}
as required.
\end{proof}

\begin{proof}[Proof of Theorem~\ref{Thm:ShiftedLowerBound}]
To prove the result we will appeal to Theorem~3.11.5 of~\cite{van1996weak}. Our one-dimensional submodel is indexed by $t \in \bbR$; we will use the shorthand $f_t$ for $f_{t \alpha^*}$ and write $P_{n,t}$ for the distribution of our data when the density for the complete sample is $f_{n^{-1/2}t}$. Our inner product $\langle \cdot, \cdot \rangle_{\lambda_\bbS}$ on $\bbR$ will be defined by
\[
    \langle t_1,t_2 \rangle_{\lambda_\bbS} = t_1t_2 \biggl[ \bbE\{ \alpha^*(X)^2 \} + \sum_{S \in \bbS} \lambda_S \mathrm{Var}\bigl( (\alpha^*)_S(X_{S,1}) \bigr) \biggr],
\]
where, given $\alpha \in L^2$, we write $(\alpha)_S(x_S) = \bbE\{\alpha(X) | X_S=x_S\}$.

We will require control on the behaviour of $f_t$ for small $t$. Define the normalising constant $c(t) = \{ \int k(t\alpha^*(x)) f(x) \,dx \}^{-1}$ so that $f_t(x)=c(t) k(t\alpha^*(x))f(x)$ is our perturbed density function. Since $k(0)=k'(0)=1, k''(0)=0$ and $\max(\|k-1\|_\infty,\|k'\|_\infty, \|k'''\|_\infty/8) \leq 1$ we have for any $u \in \bbR$ that
\begin{align}
\label{Eq:kexpansion1}
    \bigl|   k(u) - 1 \bigr| 
     \leq \min \{ 1, |u| \} \quad \quad \text{and} \quad \quad |k(u)  - 1 - u | \leq 2 |u| \min \{ 1, u^2 \}.
\end{align}
In particular, since $\alpha^* \in L^2$ and has mean zero, it follows from the dominated convergence theorem that
\begin{equation}
\label{Eq:kexpansion2}
    1/c(t) - 1 =  \int k( t \alpha^*(x)) f(x) \,dx - 1  =  \int \{ k( t \alpha^*(x)) -1 - t \alpha^*(x) \}f(x) \,dx = o(t^2)
\end{equation}
as $t \rightarrow 0$ and hence that $c(t)=1+o(t^2)$ as $t \rightarrow 0$. Our next calculations will allow us to control the shifted marginal distributions of $f_t$ on $\bbR^S$ for $S \in \bbS$. Using~\eqref{Eq:kexpansion1},~\eqref{Eq:kexpansion2} and the assumption that $\max_{S \in \bbS} \|\bar{r}_S\|_\infty <\infty$ we have that
\begin{align}
\label{Eq:kexpansion3}
    \biggl| \frac{\int f_t r_S}{\int f r_S} &- 1 - t \bbE\{ (\alpha^*)_S(X_{S,1}) \} \biggr| = \biggl| \int f(x) \bar{r}_S(x_S) \bigl\{ c(t) k(t \alpha^*(x)) - 1 - t \alpha^*(x) \bigr\} \,dx \biggr| \nonumber \\
    & \leq \biggl| \int f(x) \bar{r}_S(x_S) \bigl\{k(t \alpha^*(x)) - 1 - t \alpha^*(x) \bigr\} \,dx \biggr| + o(t^2) = o(t^2)
\end{align}
as $t \rightarrow 0$. Using~\eqref{Eq:kexpansion1} and writing $x=(x_S,x_{S^c})$ and $(f_t)_S$ for the marginal density of $f$ on $\bbR^S$, for any $t \in \bbR$ and $x_S \in \bbR^S$ we have that
\begin{align}
\label{Eq:MarginalExpansion1}
    \biggl| \frac{(f_t)_S(x_S)}{c(t) f_S(x_S)} - 1 \biggr| &= \biggl| \frac{\int k(t\alpha^*(x)) f(x) \,dx_{S^c}}{f_S(x_S)} - 1 \biggr| \nonumber \\
    &= \frac{| \int \{k(t\alpha^*(x)) - 1\} f(x) \,dx_{S^c} |}{f_S(x_S)} \leq \bbE \bigl[ \min\{1 , |t \alpha^*(X)|\} \bigm| X_S = x_S \bigr].
\end{align}
Similarly, for any $t \in \bbR$ and $x_S \in \bbR^S$ we have that
\begin{align}
\label{Eq:MarginalExpansion2}
    \biggl| \frac{(f_t)_S(x_S)}{c(t) f_S(x_S)} - 1 - t (\alpha^*)_S(x_S) \biggr| &= \frac{| \int \{k(t\alpha^*(x)) - 1 - t \alpha^*(x)\} f(x) \,dx_{S^c} |}{f_S(x_S)} \nonumber \\
    & \leq 2 |t| \bbE \bigl[ |\alpha^*(X)| \min\{1 , t^2 \alpha^*(X)^2 \} \bigm| X_S = x_S \bigr].
\end{align}

Equipped with the preceding bounds we will now establish the local asymptotic normality of our sequence of experiments. In the following calculation all arguments of logarithms are bounded away from $0$ and $\infty$ by the fact that $k$ takes values in $[1/2,3/2]$. We can therefore see that for any fixed $t \in \bbR$ we have
\begin{align}
\label{Eq:LAMExpansion}
    &\biggl| \log \frac{dP_{n,t}}{dP_{n,0}} - \frac{t}{n^{1/2}} \sum_{i=1}^n \alpha^*(X_i) - \frac{t}{n^{1/2}} \sum_{S \in \bbS} \sum_{i=1}^{n_S} \bigl[ (\alpha^*)_S(X_{S,i}) - \bbE \{(\alpha^*)_S(X_{S,1})\} \bigr] + \frac{1}{2} \|t\|_{\lambda_\bbS}^2 \biggr| \nonumber \\
    & = \biggl| \sum_{i=1}^n \biggl\{ \log \frac{f_{n^{-1/2}t}(X_i)}{f(X_i)} - \frac{t}{n^{1/2}} \alpha^*(X_i) \biggr\} + \frac{t^2}{2} \bbE\{ \alpha^*(X)^2\} + \frac{t^2}{2} \sum_{S \in \bbS} \lambda_S \mathrm{Var} \bigl( (\alpha^*)_S(X_{S,1}) \bigr) \nonumber \\
    & \hspace{30pt} + \sum_{S \in \bbS} \sum_{i=1}^{n_S} \biggl\{ \log \biggl( \frac{(\int f r_S) (f_{n^{-1/2}t})_S(X_{S,i})}{(\int f_{n^{-1/2}t} r_S) f_S(X_{S,i})} \biggr) - \frac{t}{n^{1/2}} \bigl[ (\alpha^*)_S(X_{S,i}) - \bbE \{(\alpha^*)_S(X_{S,1})\} \bigr] \biggr\} \biggr| \nonumber \\
    & \leq \sum_{S \in \bbS} n_S \biggl| \log \frac{\int f_{n^{-1/2}t} r_S}{\int f r_S} - \frac{t}{n^{1/2}}\bbE \{(\alpha^*)_S(X_{S,1})\} + \frac{t^2 \lambda_S}{2n_S} \bbE^2 \{(\alpha^*)_S(X_{S,1})\}  \biggr| \nonumber \\
    & \hspace{50pt} + \biggl( n + \sum_{S \in \bbS} n_S \biggr) \bigl| \log c(n^{-1/2}t) \bigr| + \frac{t^2}{2}\biggl| \frac{1}{n} \sum_{i=1}^n \alpha^*(X_i)^2 - \bbE\{ \alpha^*(X)^2\} \biggr|\nonumber  \\
    & \hspace{50pt} + \frac{t^2}{2} \sum_{S \in \bbS} \biggl| \frac{1}{n} \sum_{i=1}^{n_S} (\alpha^*)_S(X_{S,i})^2 - \lambda_S\bbE\{ (\alpha^*)_S(X_{S,1})^2 \} \biggr| \nonumber \\
    & \hspace{50pt} + \sum_{i=1}^n \biggl| \log k(n^{-1/2}t \alpha^*(X_i)) - \frac{t}{n^{1/2}} \alpha^*(X_i) + \frac{t^2}{2n} \alpha^*(X_i)^2 \biggr| \nonumber \\
    & \hspace{50pt} + \sum_{S \in \bbS} \sum_{i=1}^{n_S} \biggl| \log \biggl( \frac{(f_{n^{-1/2}t})_S(X_{S,i})}{c(n^{-1/2}t) f_S(X_{S,i})} \biggr) - \frac{t}{n^{1/2}} (\alpha^*)_S(X_{S,i}) + \frac{t^2}{2n} (\alpha^*)_S(X_{S,i})^2 \biggr| \nonumber \\
    & \lesssim o_p(1) + \frac{t^2}{n} \sum_{i=1}^n \alpha^*(X_i)^2 \min \bigl\{1, |n^{-1/2}t \alpha^*(X_i)| \bigr\} \nonumber \\
    & \hspace{130pt} + \frac{t^2}{n} \sum_{S \in \bbS} \sum_{i=1}^{n_S} \bbE \bigl[ \alpha^*(X)^2 \min\{1,|n^{-1/2}t\alpha^*(X)|\} \bigm| X_S = X_{S,i} \bigr]
\end{align}
as $n \rightarrow \infty$, under $P_{n,0}$, where the final bound follows from~\eqref{Eq:kexpansion1},~\eqref{Eq:kexpansion2},~\eqref{Eq:kexpansion3},~\eqref{Eq:MarginalExpansion1} and~\eqref{Eq:MarginalExpansion2}. Now, since
\[
    \bbE \bigl[ \alpha^*(X)^2 \min \bigl\{1, |n^{-1/2}t \alpha^*(X)| \bigr\} \bigr] = o(1)
\]
as $n \rightarrow \infty$ and since $\max_{S \in \bbS} \|\bar{r}_S\|_\infty < \infty$, we have by Markov's inequality that the final two terms in~\eqref{Eq:LAMExpansion} are $o_p(1)$ as $n \rightarrow \infty$, under $P_{n,0}$. Writing $Z \sim N(0,1)$ we therefore have that 
\[
    \log \frac{dP_{n,t}}{dP_{n,0}} \overset{\mathrm{d}}{\rightarrow} \|t\|_{\lambda_\bbS} Z - \frac{1}{2} \|t\|_{\lambda_\bbS}^2
\]
as $n \rightarrow \infty$, under $P_{n,0}$, for each $t \in \bbR$. This justifies our claim that our sequence of experiments is locally asymptotically normal.

Having establish local asymptotic normality, we now prove that the sequence of parameters $\theta(f_{n^{-1/2}t})$ is regular. Indeed, for any $t \in \bbR$ we have by~\eqref{Eq:kexpansion1},~\eqref{Eq:kexpansion2} and the dominated convergence theorem that
\begin{align*}
    \bigl| n^{1/2}&\{ \theta(f_{n^{-1/2}t}) - \theta(f) \} - t \mathbb{E}\{a(X) \alpha^*(X)\} \bigr| \\
    & = \biggl| n^{1/2} \int a(x)\{ c(n^{-1/2}t) k(n^{-1/2}t \alpha^*(x)) - 1 - n^{-1/2}t \alpha^*(x) \}f(x) \,dx  \biggr| \\
    & \leq n^{1/2} \int |a(x)\{ k(n^{-1/2}t \alpha^*(x)) - 1 - n^{-1/2}t \alpha^*(x) \}|f(x) \,dx \\
    & \hspace{100pt} + (3/2) n^{1/2}|c(n^{-1/2}t)-1| \bbE^{1/2}\{a(X)^2\} \\
    & \leq 2 |t| \int |a(x) \alpha^*(x)| \min\{1, t^2 \alpha^*(x)^2/n\} f(x) \,dx + 2 n^{1/2}|c(n^{-1/2}t)-1| \bbE^{1/2}\{a(X)^2\} \rightarrow 0
\end{align*}
as $n \rightarrow \infty$. Using the notation of~\citet[][Chapter~3]{van1996weak}, we now see that our sequence of parameters is regular with norming operator $r_n=n^{1/2}$ and $\dot{\kappa}(t) = t \mathbb{E}\{a(X) \alpha^*(X)\}$. We think of $\dot{\kappa}$ as a linear map $\dot{\kappa}: \bbR_{\lambda_\bbS} \rightarrow \mathbb{R}_\mathrm{E}$, where we write $\bbR_{\lambda_\bbS}$ for $\bbR$ equipped with the inner product $\langle \cdot, \cdot \rangle_{\lambda_\bbS}$ and $\bbR_\mathrm{E}$ for $\bbR$ equipped with the standard Euclidean inner product. 

We must now find the adjoint operator $\dot{\kappa}^* : \bbR_\mathrm{E} \rightarrow \bbR_{\lambda_\bbS}$ of $\dot{\kappa}$. For $S \in \bbS$ and $\epsilon_S \in H_S$ we will write, in a slight abuse of notation, $\alpha_\bbS^* + \epsilon_S$ for the element of $H_\bbS$ that is equal to $\alpha_\bbS^*$ except in the $S$th coordinate, where we replace $\alpha_S^*$ by $\alpha_S^* + \epsilon_S$. Whenever $\eta \in \bbR$ and  $\epsilon_S \in H_S$ satisfies $\int f_S \epsilon_S^2 /\bar{r}_S < \infty$ we have that
\begin{align}
\label{Eq:Stationarity0}
    0 & \leq \mathcal{L}(\alpha_\bbS^* + \eta \epsilon_S) - \mathcal{L}(\alpha_\bbS^*) \nonumber \\
    & = \mathrm{Var} \bigl( \alpha^*(X) - \eta \epsilon_S(X_S) \bigr) - \mathrm{Var} \bigl( \alpha^*(X) \bigr) + \int \frac{f_S}{\lambda_S \bar{r}_S} \bigl(\alpha_S^* + \eta \epsilon_S \bigr)^2 - \int \frac{f_S}{\lambda_S \bar{r}_S} (\alpha_S^*)^2 \nonumber \\
    & = - 2\eta \, \mathrm{Cov} \bigl( \alpha^*(X), \epsilon_S(X_S) \bigr) + 2\eta \int \frac{f_S}{\lambda_S \bar{r}_S} \alpha_S^* \epsilon_S + \eta^2 \biggl\{ \mathrm{Var}(\epsilon_S(X_S)) + \int \frac{f
    _S  \epsilon_S^2}{\lambda_S \bar{r}_S} \biggr\} \nonumber \\
    & = - 2 \eta \int f_S \epsilon_S \biggl\{ (\alpha^*)_S - \frac{\alpha_S^*}{\lambda_S \bar{r}_S} \biggr\} + O(\eta^2)
\end{align}
as $\eta \rightarrow 0$. After suitable centring, we may therefore conclude that
\begin{equation}
\label{Eq:Stationarity}
    (\alpha^*)_S(X_S) = \bbE \bigl\{ \alpha^*(X) \bigm| X_S \bigr\} = \frac{\alpha_S^*(X_S)}{\lambda_S \bar{r}_S(X_S)} - \int \frac{\alpha_S^* f_S}{\lambda_S \bar{r}_S}
\end{equation}
almost surely. Thus, for any $t_1,t_2 \in \bbR$ we have
\begin{align}
\label{Eq:AdjointShifted}
    \langle t_1, t_2 \rangle_{\lambda_\bbS} &= t_1t_2\biggl[ \bbE\{ \alpha^*(X)^2 \} + \sum_{S \in \bbS} \lambda_S \mathrm{Var}\bigl( (\alpha^*)_S(X_{S,1}) \bigr) \biggr] \nonumber \\
    %& =t_1t_2\biggl\{ \bbE\{ \alpha^*(X)^2 \} + \sum_{S \in \bbS} \lambda_S \mathrm{Var}\biggl( \frac{\alpha_S^*(X_{S,1})}{\lambda_S \bar{r}_S(X_{S,1})} \biggr) \biggr\} \\
    & = t_1t_2\biggl[ \bbE\{ \alpha^*(X)^2 \} + \sum_{S \in \bbS} \lambda_S \bbE\biggl\{ \frac{\alpha_S^*(X_{S,1})^2}{\lambda_S^2 \bar{r}_S(X_{S,1})^2} \biggr\} \biggr] \nonumber \\
    & = t_1t_2\biggl[ \bbE\{ \alpha^*(X)^2 \} + \sum_{S \in \bbS} \lambda_S \bbE \biggl\{ \frac{\alpha_S^*(X_{S,1})}{\lambda_S \bar{r}_S(X_{S,1})} (\alpha^*)_S(X_{S,1}) \biggr\} \biggr] \nonumber \\
    & = t_1t_2\biggl[ \bbE\{ \alpha^*(X)^2 \} + \sum_{S \in \bbS} \bbE \bigl\{ \alpha_S^*(X_{S})(\alpha^*)_S(X_{S}) \bigr\} \biggr] \nonumber \\
    & = t_1t_2 \bbE \biggl[ \alpha^*(X) \biggl\{ \alpha^*(X) + \sum_{S \in \bbS} \alpha_S^*(X_S) \biggr\} \biggr] = t_1t_2 \bbE \{ \alpha^*(X) a(X) \} = t_1 \cdot \dot{\kappa}(t_2),
\end{align}
and we see that the adjoint operator $\dot{\kappa}^* : \bbR_\mathrm{E} \rightarrow \bbR_{\lambda_\bbS}$ is simply the identity map $\dot{\kappa}^*(t_1)=t_1$.

Writing $\sigma^2 = \|\dot{\kappa}^* 1\|_{\lambda_\bbS}^2 = \bbE\{\alpha^*(X)a(X)\}$ and letting $G\sim N(0,\sigma^2)$, we may therefore apply Theorem 3.11.5 of~\cite{van1996weak},~\eqref{Eq:Stationarity} and the calculations in~\eqref{Eq:AdjointShifted} to deduce that for any estimator sequence $(\theta_n)$ we have
\begin{align*}
    \sup_{I} \liminf_{n \rightarrow \infty} \max_{t \in I} n\mathbb{E}_{f_{n^{-1/2}t}} \bigl[ \bigl\{ \theta_n - \theta(f_{n^{-1/2}t}) \}^2 \bigr] &\geq \bbE\{\alpha^*(X)a(X)\} \\
    & = \bbE\{ \alpha^*(X)^2 \} + \sum_{S \in \bbS} \frac{1}{\lambda_S} \int \frac{f_S}{\bar{r}_S} (\alpha_S^*)^2 = \mathcal{L}(\alpha_\bbS^*),
\end{align*}
where the supremum is taken over all finite subsets $I$ of $\bbR$, as required.
\end{proof}

\begin{proof}[Proof of Proposition~\ref{Prop:ShiftedMonotoneCase}]
By Proposition~\ref{Prop:UniqueMinimiser} we know that there exists a unique minimiser $\alpha_\bbS^*$ and by~\eqref{Eq:Stationarity0} we know that it must satisfy the stationarity conditions~\eqref{Eq:Stationarity}. We prove our result by showing that $\alpha_\bbS^*$ must take the form given in~\eqref{Eq:ShiftedMonotoneMinimiser}. Throughout this proof we use the shorthand $\alpha_j^0 = \alpha_{[j]}^*$ for $j \in [d-1]$ such that $[j] \in \bbS$ and $\alpha_j \equiv 0$ otherwise and identify a function $g_j$ of $y_j$ with the random variable $g_j(Y_j)$ for notational convenience. We begin the proof by arguing inductively that there exist constants $\theta_1,\ldots,\theta_{d-1} \in \bbR$ such that $\theta_j=0$ if $[j] \not\in \bbS$ and
\begin{equation}
\label{Eq:InductionHypothesis}
    \alpha_j^0 = \frac{R_j}{1+R_j \mu_j} \biggl\{ \mu_j(\tilde{a}_j - \alpha_1^0 - \ldots - \alpha_{j-1}^0) - \sum_{\ell=j}^{d-1} \theta_\ell \nu_j^{\ell-1} \biggr\}
\end{equation}
for each $j \in [d-1]$. The arguments we give here are for the case that $\bbS = \{[1],\ldots,[d-1]\}$ but they can be straightforwardly extended to cases where $\bbS \subset \{[1],\ldots,[d-1]\}$. We establish the base case $j=d-1$ by noting that by~\eqref{Eq:Stationarity} there must exist a constant $\theta_{d-1} \in \bbR$ such that
\[
    \theta_{d-1} + \frac{\alpha_{d-1}^0}{R_{d-1}} = \bbE \bigl( a - \theta - \alpha_1^0 - \ldots - \alpha_{d-1}^0 \bigm| Y_{d-1} \bigr) = \tilde{a}_{d-1} - \alpha_1^0 - \ldots - \alpha_{d-1}^0.
\]
It is now immediate that~\eqref{Eq:InductionHypothesis} holds with $j=d-1$ on recalling that $\mu_{d-1} \equiv 1$ and that $\nu_{d-1}^{d-2}\equiv 1$. Assuming that the induction hypothesis~\eqref{Eq:InductionHypothesis} holds at $j+1$ and using the stationarity condition~\eqref{Eq:Stationarity} and the tower law of expectation we see that there exists $\theta_j \in \bbR$ such that
\begin{align*}
    \theta_j + \frac{\alpha_j^0}{R_j} &= \bbE \biggl( \frac{\alpha_{j+1}^0}{R_{j+1}} \biggm| Y_j \biggr) \\
    &= \bbE \biggl[ \frac{1}{1+R_{j+1}\mu_{j+1}} \biggl\{ \mu_{j+1}(\tilde{a}_{j+1} - \alpha_1^0 - \ldots - \alpha_j^0) - \sum_{\ell=j+1}^{d-1} \theta_\ell \nu_{j+1}^{\ell-1} \biggr\} \biggm| Y_j \biggr] \\
    &= \mu_j( \tilde{a}_j - \alpha_1^0 - \ldots - \alpha_j^0) - \sum_{\ell=j+1}^{d-1} \theta_\ell \nu_j^{\ell-1}
\end{align*}
and it immediately follows that~\eqref{Eq:InductionHypothesis} holds at $j$.

We now use the fact that~\eqref{Eq:InductionHypothesis} holds for all $j \in [d-1]$ to prove by induction that $\alpha_\bbS^*$ satisfies~\eqref{Eq:ShiftedMonotoneMinimiser}. By~\eqref{Eq:InductionHypothesis} with $j=1$ we know that
\[
    \alpha_1^0 = \frac{R_1}{1/\mu_1 + R_1} \biggl( \tilde{a}_1 - \sum_{\ell=1}^{d-1} \theta_\ell \frac{\nu_1^{\ell-1}}{\mu_1} \biggr),
\]
which coincides with~\eqref{Eq:ShiftedMonotoneMinimiser} at $j=1$, establishing the base case. Using~\eqref{Eq:ShiftedMonotoneMinimiser} as an induction hypothesis we have by~\eqref{Eq:InductionHypothesis} that
\begin{align*}
    &\alpha_j^0 = \frac{R_j}{1+R_j \mu_j} \biggl\{ \mu_j(\tilde{a}_j - \alpha_1^0 - \ldots - \alpha_{j-1}^0) - \sum_{\ell=j}^{d-1} \theta_\ell \nu_j^{\ell-1} \biggr\} \\
    &= \frac{R_j \mu_j}{1+R_j \mu_j} \biggl[ \tilde{a}_j - \sum_{i=1}^{j-1} \frac{R_i \mu_i}{1 + R_i\mu_i} \sum_{k=1}^i \biggl( \prod_{m=k}^{i-1} \frac{1}{1+R_m \mu_m} \biggr) \biggl\{ \tilde{a}_k - \tilde{a}_{k-1} - \sum_{\ell=1}^{d-1} \theta_\ell \biggl( \frac{\nu_k^{\ell-1}}{\mu_k} - \frac{
\nu_{k-1}^{\ell-1}}{\mu_{k-1}} \biggr) \biggr\} \\
    & \hspace{395pt} - \sum_{\ell=j}^{d-1} \theta_\ell \frac{\nu_j^{\ell-1}}{\mu_j} \biggr] \\
    & = \frac{R_j \mu_j}{1+R_j \mu_j} \biggl[ \tilde{a}_j - \sum_{k=1}^{j-1} \biggl\{ \tilde{a}_k - \tilde{a}_{k-1} - \sum_{\ell=1}^{d-1} \theta_\ell \biggl( \frac{\nu_k^{\ell-1}}{\mu_k} - \frac{
\nu_{k-1}^{\ell-1}}{\mu_{k-1}} \biggr) \biggr\} \sum_{i=k}^{j-1} \frac{R_i \mu_i}{1 + R_i \mu_i} \biggl( \prod_{m=k}^{i-1} \frac{1}{1+R_m \mu_m} \biggr)\\
    & \hspace{395pt} - \sum_{\ell=j}^{d-1} \theta_\ell \frac{\nu_j^{\ell-1}}{\mu_j} \biggr] \\
    & = \frac{R_j \mu_j}{1+R_j \mu_j} \biggl[ \tilde{a}_j - \sum_{k=1}^{j-1} \biggl\{ \tilde{a}_k - \tilde{a}_{k-1} - \sum_{\ell=1}^{d-1} \theta_\ell \biggl( \frac{\nu_k^{\ell-1}}{\mu_k} - \frac{
\nu_{k-1}^{\ell-1}}{\mu_{k-1}} \biggr) \biggr\} \biggl( 1 - \prod_{m=k}^{j-1} \frac{1}{1+R_m \mu_m} \biggr)\\
    & \hspace{395pt} - \sum_{\ell=j}^{d-1} \theta_\ell \frac{\nu_j^{\ell-1}}{\mu_j} \biggr] \\
    &= \frac{R_j \mu_j}{1+R_j \mu_j} \biggl[ \tilde{a}_j - \tilde{a}_{j-1} + \sum_{\ell=1}^{d-1} \theta_\ell \frac{\nu_{j-1}^{\ell-1}}{\mu_{j-1}} - \sum_{\ell=j}^{d-1} \theta_\ell \frac{\nu_j^{\ell-1}}{\mu_j} \\
    & \hspace{100pt} - \sum_{k=1}^{j-1} \biggl\{ \tilde{a}_k - \tilde{a}_{k-1} - \sum_{\ell=1}^{d-1} \theta_\ell \biggl( \frac{\nu_k^{\ell-1}}{\mu_k} - \frac{
\nu_{k-1}^{\ell-1}}{\mu_{k-1}} \biggr) \biggr\} \biggl( \prod_{m=k}^{j-1} \frac{1}{1+R_m \mu_m} \biggr) \\
    & = \frac{R_j \mu_j}{1+R_j \mu_j} \sum_{k=1}^{j} \biggl( \prod_{m=k}^{j-1} \frac{1}{1+R_m \mu_m} \biggr) \biggl\{ \tilde{a}_k - \tilde{a}_{k-1} - \sum_{\ell=1}^{d-1} \theta_\ell \biggl( \frac{\nu_k^{\ell-1}}{\mu_k} - \frac{
\nu_{k-1}^{\ell-1}}{\mu_{k-1}} \biggr) \biggr\},
\end{align*}
as required.
\end{proof}

\begin{proof}[Proof of Corollary~\ref{Prop:MonotoneCase}]
We begin by simplifying the $\mu_j,\nu_j^k$ and $\tilde{a}_j$ in this MCAR setting. Since $R_j = \lambda_{[j]}$ is constant for all $j \in [d-1]$ and $\mu_{d-1}$ is constant by definition, we can see that all the $\mu_j$ are constants. For notational convenience, introduce $\lambda_{[j]}$ for all $j \in [d-1]$ by defining $\lambda_{[j]}=0$ if $[j] \notin \bbS$. It is now straightforward to see inductively that
\[
    \mu_j = \frac{1}{1 + \sum_{\ell=j+1}^{d-1} \lambda_{[\ell]}}
\]
for all $j \in [d-1]$. Similarly, we see that for any $k \geq j$ we have
\[
    \nu_j^k = \prod_{m=j}^k \nu_m^m = \prod_{m=j}^k \frac{1+ \sum_{\ell=m+2}^{d-1} \lambda_{[\ell]}}{1+ \sum_{\ell=m+1}^{d-1} \lambda_{[\ell]}} = \frac{1+ \sum_{\ell=k+2}^{d-1} \lambda_{[\ell]}}{1+ \sum_{\ell=j+1}^{d-1} \lambda_{[\ell]}}
\]
which, in particular, is constant. We also have that $\nu_j^k/\mu_j = 1 + \sum_{\ell=k+2}^{d-1} \lambda_{[\ell]}$ does not depend on $j$ except through the fact that $\nu_j^k/\mu_j=0$ if $k \leq j-2$. Finally, we have
\[
    \tilde{a}_j(X_1,\ldots,X_j) = \bbE\{ a(X) | X_1,\ldots,X_j\} - \theta = a_{[j]}(X_1,\ldots,X_j) - \theta.
\]
It now follows from~\eqref{Eq:ShiftedMonotoneMinimiser} that the minimising choice of $\alpha_{[j]}$ is of the form
\begin{align*}
    \alpha_{[j]} &= \frac{\lambda_{[j]}}{1 + \sum_{\ell=j}^{d-1} \lambda_{[\ell]}} \sum_{k=1}^j \biggl( \prod_{m=k}^{j-1} \frac{1+ \sum_{\ell=m+1}^{d-1} \lambda_{[\ell]}}{1+ \sum_{\ell=m}^{d-1} \lambda_{[\ell]}} \biggr) ( a_{[k]} - a_{[k-1]} +\theta_{k-1}) \\
    &= \sum_{k=1}^j  \frac{\lambda_{[j]} }{1+ \sum_{\ell=k}^{d-1} \lambda_{[\ell]}} ( a_{[k]} - a_{[k-1]} +\theta_{k-1}).
\end{align*}
We choose $\theta_\ell=0$ for all $\ell \in [d-1]$ in order that each $\alpha_{[j]}$ has mean zero. This completes the proof.

In fact, we can also give a simple direct proof of the corollary, without using Proposition~\ref{Prop:ShiftedMonotoneCase}. As in the proof of Proposition~\ref{Prop:ProductCase}, we check that the choice of $\alpha_\bbS$ given in the statement satisfies the stationarity conditions~\eqref{Eq:Stationarity2}. Indeed, for $j \in [d-1]$ such that $[j] \in \bbS$ we have
\begin{align*}
    (1+1/\lambda_{[j]}) &\alpha_{[j]} + \sum_{S \in \bbS \setminus \{[j]\}} \mathbb{E}( \alpha_S | X_{[j]}) \\
    & = \sum_{k=1}^j \frac{1}{1+\sum_{\ell=k}^{d-1} \lambda_{[\ell]}} (a_{[k]} - a_{[k-1]}) + \sum_{j'=1}^{d-1} \sum_{k=1}^{j'} \frac{\lambda_{[j']}}{1+\sum_{\ell=k}^{d-1} \lambda_{[\ell]}} \mathbb{E}(a_{[k]} - a_{[k-1]} | X_{[j]}) \\
    & = \sum_{k=1}^j \frac{1}{1+\sum_{\ell=k}^{d-1} \lambda_{[\ell]}} (a_{[k]} - a_{[k-1]}) + \sum_{j'=1}^{d-1} \sum_{k=1}^{j'} \frac{\lambda_{[j']}}{1+\sum_{\ell=k}^{d-1} \lambda_{[\ell]}} (a_{[k]} - a_{[k-1]} ) \mathbbm{1}_{\{ k \leq j \}} \\
    & = \sum_{k=1}^j \frac{1}{1+\sum_{\ell=k}^{d-1} \lambda_{[\ell]}} \biggl( 1 + \sum_{j'=k}^{d-1} \lambda_{[j']} \biggr) (a_{[k]} - a_{[k-1]}) = \sum_{k=1}^j (a_{[k]} - a_{[k-1]}) = a_{[j]}-\theta,
\end{align*}
almost surely, as required.
\end{proof}

\subsection{Proofs for Section~\ref{Sec:UpperBound}}
\label{Sec:UpperBoundProofs}

\begin{proof}[Proof of Proposition~\ref{Prop:OptimisationShifted}]
In this proof it will be helpful to equip the set $H_\bbS = \prod_{S \in \bbS} H_S$ with the alternative inner product
\[
    \langle \alpha_\bbS, \beta_\bbS \rangle = \sum_{S \in \bbS} \mathbb{E}\biggl\{ \frac{1+\lambda_S \bar{r}_S(X_S)}{\lambda_S \bar{r}_S(X_S)} \alpha_S(X_S) \beta_S(X_S) \biggr\}.
\]
where we write $\alpha_\bbS=(\alpha_S : S \in \bbS)$ here and throughout the proof. For each $\alpha_\bbS \in H_\bbS$ we define $\nabla \mathcal{L}(\alpha_\bbS) \in H_\bbS$ by
\begin{align*}
    \frac{1}{2}(\nabla \mathcal{L}(\alpha_\bbS))_S(x_S) &= \alpha_S(x_S) - \frac{\lambda_S \bar{r}_S(x_S)}{1+\lambda_S \bar{r}_S(x_S)}\{a_S(x_S) - \bar{\theta}_S^{(1)}\} \\
    & \hspace{65pt}+ \frac{\lambda_S \bar{r}_S(x_S)}{1+\lambda_S \bar{r}_S(x_S)} \sum_{S' \neq S} \biggl[ \bbE\{\alpha_{S'}(X_{S'}) | X_S = x_S\} - \frac{ \int \frac{\lambda_S \bar{r}_S}{1+\lambda_S \bar{r}_S} \alpha_{S'} f}{\int \frac{\lambda_S \bar{r}_S}{1+\lambda_S \bar{r}_S} f} \biggr].
\end{align*}
Note in particular that $\nabla \mathcal{L}(\alpha_\bbS^*)=0$ by the stationarity condition~\eqref{Eq:Stationarity} and the fact that $\int \alpha_S^* f=0$. For $\alpha_\bbS \in H_\bbS$ we further define $Q(\alpha_\bbS) = \mathrm{Var}( \sum_{S \in \bbS} \alpha_S(X_S) ) + \sum_S \int \frac{\alpha_S^2 f_S}{\lambda_S \bar{r}_S}$. With these definitions, for any $\alpha_\bbS,\epsilon_\bbS \in H_\bbS$ we may write
\begin{align}
\label{Eq:QuadraticExpansion}
    \mathcal{L}(\alpha_\bbS + \epsilon_\bbS) - \mathcal{L}&(\alpha_\bbS) = -2 \mathrm{Cov}\biggl( a(X) - \sum_{S \in \bbS} \alpha_S(X_S), \sum_{S \in \bbS} \epsilon_S(X_S) \biggr) + 2 \sum_{S \in \bbS} \int \frac{\alpha_S \epsilon_S f_S}{\lambda_S \bar{r}_S}  + Q(\epsilon_\bbS) \nonumber \\
    &= 2 \sum_{S \in \bbS} \bbE \biggl[ \epsilon_S(X_S) \biggl\{ \frac{\alpha_S(X_S)}{\lambda_S \bar{r}_S(X_S)} - \bbE\biggl( a(X) - \sum_{S' \in \bbS} \alpha_{S'}(X_{S'}) \biggm| X_S \biggr) \biggr\} \biggr] +  Q(\epsilon_\bbS) \nonumber \\
    & = \langle \nabla \mathcal{L}(\alpha_\bbS), \epsilon_\bbS \rangle +  Q(\epsilon_\bbS),
\end{align}
which justifies our notation and thinking of $\nabla \mathcal{L}$ as the gradient of $\mathcal{L}$. Comparing with the definition of $\alpha_\bbS^{(M)}$ in Section~\ref{Sec:UpperBoundApprox}, we see that we have the recurrence relation $\alpha_\bbS^{(M+1)} = \alpha_\bbS^{(M)} - (\eta/2) \nabla \mathcal{L}(\alpha_\bbS^{(M)})$, so that we may regard this iterative scheme as gradient descent.

To analyse the performance of this first-order optimisation algorithm we will use the smoothness and strong convexity of our problem. To this end, we now derive some useful properties of $\nabla \mathcal{L}$ and $Q$. For any $\alpha_\bbS, \epsilon_\bbS \in H_\bbS$ we have by~\eqref{Eq:QuadraticExpansion} that
\begin{align}
\label{Eq:GradientDiff1}
    \langle \nabla \mathcal{L}(\alpha_\bbS+\epsilon_\bbS) - \nabla \mathcal{L}(\alpha_\bbS), \epsilon_\bbS \rangle &= - \bigl[ \bigl\{ \mathcal{L}(\alpha_\bbS + \epsilon_\bbS - \epsilon_\bbS) - \mathcal{L}(\alpha_\bbS + \epsilon_\bbS) \bigr\} -Q(-\epsilon_\bbS) \bigr] \nonumber \\
    & \hspace{50pt} - \bigl[ \bigl\{ \mathcal{L}(\alpha_\bbS + \epsilon_\bbS) - \mathcal{L}(\alpha_\bbS) \bigr\} -Q(\epsilon_\bbS) \bigr] \nonumber \\
    & = 2Q(\epsilon_\bbS).
\end{align}
From the definition of $\nabla \mathcal{L}$ we also see, with a slight abuse of notation, that
\begin{align}
\label{Eq:GradientDiff2}
    \frac{1}{4}& \bigl\| \bigl\{ \nabla \mathcal{L}(\alpha_\bbS+\epsilon_\bbS) - \nabla \mathcal{L}(\alpha_\bbS) \bigr\} \bigr\|^2 \nonumber \\
    & = \sum_{S \in \bbS} \int \frac{1+\lambda_S \bar{r}_S}{\lambda_S \bar{r}_S} f_S \biggl\{ \epsilon_S + \frac{\lambda_S \bar{r}_S}{1+\lambda_S \bar{r}_S} \sum_{S' \neq S} \biggl( \mathbb{E}(\epsilon_{S'} | X_S) - \frac{ \int \frac{\lambda_S \bar{r}_S}{1+\lambda_S \bar{r}_S} \epsilon_{S'} f}{\int \frac{\lambda_S \bar{r}_S}{1+\lambda_S \bar{r}_S} f} \biggr) \biggr\}^2 \nonumber \\
    & = \sum_{S \in \bbS} \int f \biggl[ \frac{1+\lambda_S \bar{r}_S}{\lambda_S \bar{r}_S} \epsilon_S^2 + 2 \sum_{S' \neq S} \epsilon_S \epsilon_{S'} + \frac{\lambda_S \bar{r}_S}{1+\lambda_S \bar{r}_S} \biggl\{ \sum_{S' \neq S} \biggl( \mathbb{E}(\epsilon_{S'} | X_S) - \frac{ \int \frac{\lambda_S \bar{r}_S}{1+\lambda_S \bar{r}_S} \epsilon_{S'} f}{\int \frac{\lambda_S \bar{r}_S}{1+\lambda_S \bar{r}_S} f} \biggr) \biggr\}^2 \biggr] \nonumber \\
    & \leq \sum_{S \in \bbS} \int f \biggl[ \frac{1+\lambda_S \bar{r}_S}{\lambda_S \bar{r}_S} \epsilon_S^2 + 2 \sum_{S' \neq S} \epsilon_S \epsilon_{S'} + \frac{\lambda_S \bar{r}_S}{1+\lambda_S \bar{r}_S} \biggl\{ \sum_{S' \neq S} \mathbb{E}(\epsilon_{S'} | X_S) \biggr\}^2 \biggr] \nonumber \\
    & \leq \sum_{S \in \bbS} \int f \biggl\{ \frac{1+\lambda_S \bar{r}_S}{\lambda_S \bar{r}_S} \epsilon_S^2 + 2 \sum_{S' \neq S} \epsilon_S \epsilon_{S'} + \biggl( \sum_{S' \neq S} \epsilon_{S'} \biggr)^2 \biggr\} \nonumber \\
    & = \sum_{S \in \bbS} \biggl\{ \int \frac{f_S}{\lambda_S \bar{r}_S} \epsilon_S^2 + \mathrm{Var} \biggl( \sum_{S' \in \bbS} \epsilon_{S'} \biggr) \biggr\} \leq |\bbS| Q(\epsilon_\bbS),
\end{align}
where the first inequality follows from the fact that variances are bounded by second moments and the second inequality follows from Jensen's inequality. Using~\eqref{Eq:GradientDiff1},~\eqref{Eq:GradientDiff2} and the Cauchy--Schwarz inequality we have for any $\epsilon_\bbS \neq 0$ that
\begin{equation}
\label{Eq:SmoothQuadratic}
    Q(\epsilon_\bbS) = Q(\epsilon_\bbS)^{-1} \frac{1}{4}  \langle \nabla \mathcal{L}(\alpha_\bbS+\epsilon_\bbS) - \nabla \mathcal{L}(\alpha_\bbS), \epsilon_\bbS \rangle^2 \leq Q(\epsilon_\bbS)^{-1} \{|\bbS| Q(\epsilon_\bbS)\} \|\epsilon_\bbS\|^2 = |\bbS| \| \epsilon_\bbS\|^2,
\end{equation}
and this clearly extends to the $\epsilon_\bbS=0$ case to show that our optimisation problem is $|\bbS|$-smooth. 
We complement this with strong convexity by writing
\begin{equation}
\label{Eq:StrongConvexity}
    Q(\epsilon_\bbS) \geq \sum_{S \in \bbS} \int \frac{\epsilon_S^2 f_S}{\lambda_S \bar{r}_S}  \geq \frac{1}{1+C\lambda_\mathrm{max}} \sum_{S \in \bbS} \int \frac{1+\lambda_S \bar{r}_S}{\lambda_S \bar{r}_S} f_S \epsilon_S^2  = \frac{|\bbS|}{\kappa} \|\epsilon_\bbS\|^2,
\end{equation}
where we take condition number $\kappa=|\bbS|(1+C\lambda_\mathrm{max})$. Further, we can understand the convexity of the problem around the minimiser by writing
\begin{align}
\label{Eq:LocalConvexity}
    Q(\alpha_\bbS^*) &= \mathrm{Var} \biggl( \sum_{S \in \bbS} \alpha_S^*(X_S) \biggr) + \sum_{S \in \bbS} \frac{(\alpha_S^*)^2 f_S}{\lambda_S \bar{r}_S} = \mathrm{Var} \bigl( a(X) - \alpha^*(X) \bigr) - \mathrm{Var} \bigl( \alpha^*(X) \bigr) + \mathcal{L}(\alpha_\bbS^*) \nonumber \\
    &= \mathrm{Var} \, a(X) - 2 \bbE\{\alpha^*(X) a(X) \bigr\} + \mathcal{L}(\alpha_\bbS^*) = \mathrm{Var} \,a(X) - \mathcal{L}(\alpha_\bbS^*) \leq \mathrm{Var} \, a(X),
\end{align}
where we use~\eqref{Eq:AdjointShifted} for the final equality.

Having seen that our optimisation problem is smooth and strongly convex we are now in a position to conclude our proof. Choosing step size $\eta=|\bbS|^{-1}$ and using~\eqref{Eq:GradientDescent},~\eqref{Eq:GradientDiff1},~\eqref{Eq:GradientDiff2},~\eqref{Eq:StrongConvexity} and the fact that $\nabla \mathcal{L}(\alpha_\bbS^*)=0$ we have for any $M \geq 0$ that
\begin{align*}
    \|\alpha_\bbS^{(M+1)}&-\alpha_\bbS^* \|^2 = \|\alpha_\bbS^{(M)}-\alpha_\bbS^* \|^2 - \eta \langle \alpha_\bbS^{(M)}-\alpha_\bbS^*, \nabla \mathcal{L}(\alpha_\bbS^{(M)}) - \nabla \mathcal{L}(\alpha_\bbS^*) \rangle \nonumber \\
    & \hspace{250pt} + \frac{\eta^2}{4} \|\nabla \mathcal{L}(\alpha_\bbS^{(M)}) - \nabla \mathcal{L}(\alpha_\bbS^*) \|^2 \nonumber \\
    & = \|\alpha_\bbS^{(M)}-\alpha_\bbS^* \|^2 - 2\eta Q(\alpha_\bbS^{(M)}-\alpha_\bbS^*) + \frac{\eta^2}{4} \|\nabla \mathcal{L}(\alpha_\bbS^{(M)}) - \nabla \mathcal{L}(\alpha_\bbS^*) \|^2 \nonumber \\
    & \leq \|\alpha_\bbS^{(M)}-\alpha_\bbS^* \|^2 - 2\eta Q(\alpha_\bbS^{(M)}-\alpha_\bbS^*) + \eta^2 |\bbS| Q(\alpha_\bbS^{(M)}-\alpha_\bbS^*) \nonumber \\
    & = \|\alpha_\bbS^{(M)}-\alpha_\bbS^* \|^2 - |\bbS|^{-1}  Q(\alpha_\bbS^{(M)}-\alpha_\bbS^*) \nonumber \\
    & \leq (1-1/\kappa)\|\alpha_\bbS^{(M)}-\alpha_\bbS^* \|^2 \leq \ldots \leq (1-1/\kappa)^{M+1} \|\alpha_\bbS^{(0)}-\alpha_\bbS^* \|^2 = (1-1/\kappa)^{M+1} \|\alpha_\bbS^* \|^2.
\end{align*}
It therefore follows from~\eqref{Eq:QuadraticExpansion},~\eqref{Eq:GradientDiff2},~\eqref{Eq:SmoothQuadratic},~\eqref{Eq:StrongConvexity} and~\eqref{Eq:LocalConvexity} that
\begin{align*}
    \mathcal{L}(\alpha_\bbS^{(M)}) &- \mathcal{L}(\alpha_\bbS^*) = Q(\alpha_\bbS^{(M)} - \alpha_\bbS^*) \leq |\bbS| \|\alpha_\bbS^{(M)} - \alpha_\bbS^*\|^2  \leq |\bbS| (1-1/\kappa)^M \|\alpha_\bbS^* \|^2 \\
    & \leq \kappa (1-1/\kappa)^M Q(\alpha_\bbS^*) \leq \kappa (1-1/\kappa)^M \mathrm{Var} \,a(X),
\end{align*}
as required.

The final claim in our statement is a simple consequence of~\eqref{Eq:QuadraticExpansion} and~\eqref{Eq:SmoothQuadratic}. Indeed, we may write
\begin{align*}
    \mathcal{L}&(\alpha_\bbS^{(M+1)}) - \mathcal{L}(\alpha_\bbS^{(M)}) = \langle \nabla \mathcal{L}(\alpha_\bbS^{(M)}), \alpha_\bbS^{(M+1)} - \alpha_\bbS^{(M)} \rangle + Q(\alpha_\bbS^{(M+1)} - \alpha_\bbS^{(M)}) \\
    &= -(2/\eta) \| \alpha_\bbS^{(M+1)} - \alpha_\bbS^{(M)}\|^2 + Q(\alpha_\bbS^{(M+1)} - \alpha_\bbS^{(M)}) \leq (-2/\eta+|\bbS|) \|\alpha_\bbS^{(M+1)} - \alpha_\bbS^{(M)}\|^2 \\
    & = - |\bbS| \|\alpha_\bbS^{(M+1)} - \alpha_\bbS^{(M)}\|^2 \leq 0,
\end{align*}
as claimed.

\end{proof}

\begin{proof}[Proof of Proposition~\ref{Prop:CrossFitReduction}]
We start by proving the first statement. For $S \in \bbS$ and $\ell=1,2$ define
\begin{align*}
    R_{S,\ell} = \frac{1}{n}\sum_{x \in \mathcal{D}_{3-\ell}} \{ \alpha_S^{(M)}(x) &- \hat{\alpha}_{S,(\ell)}^{(M)}(x) \} \\
    &- \frac{|\mathcal{D}_{3-\ell}|}{n \cdot n_S} \sum_{j=1}^{n_S} \frac{1}{\bar{r}_S(X_{S,j})} \biggl[\alpha_S^{(M)}(X_{S,j}) -  \frac{\hat{r}_{S,(3-\ell)}}{\bbE\{r_S(X_S)\}} \hat{\alpha}_{S,(\ell)}^{(M)}(X_{S,j}) \biggr],
\end{align*}
so that we may write $\hat{\theta}-\theta^{*,(M)} = \sum_{S \in \bbS} (R_{S,1} + R_{S,2})$. Our goal now is to show that each of these remainder terms is negligible, by bounding $\mathbb{E}(R_{S,\ell}^2)$ for each $S \in \bbS$ and $\ell=1,2$. The $\ell=2$ terms are bounded very similarly to the $\ell=1$ terms, so we focus on the latter for notational ease. We start with the simple bound
\begin{align}
\label{Eq:CrossFitRem1}
    \bbE\biggl[ \biggl(\frac{\hat{r}_{S,(2)}}{\bbE \{r_S(X_S)\}} - 1 \biggr)^2 \biggr] = \bbE \biggl[ \biggl\{ \frac{1}{|\mathcal{D}_2|} \sum_{x \in \mathcal{D}_2} \bar{r}_S(x) - 1 \biggr\}^2 \biggr] = \frac{1}{|\mathcal{D}_2|} \mathrm{Var} \, \bar{r}_S(X) \leq \frac{C}{|\mathcal{D}_2|}.
    % \bbE\biggl[ \biggl(\frac{\hat{r}_{S,(1)}}{\bbE \{r_S(X_S)\}} - 1 \biggr)^2 \biggr]  &\leq \biggl(\frac{Cn_S}{2 |\mathcal{D}_{S,1}|} \biggr)^2 \bbE\biggl[ \biggl(\frac{\bbE \{r_S(X_S)\}}{\hat{r}_{S,(1)}} - 1 \biggr)^2 \biggr] \nonumber \\
    % &=\biggl(\frac{Cn_S}{2 |\mathcal{D}_{S,1}|} \biggr)^2 \bbE\biggl[ \biggl\{ \frac{2}{n_S} \sum_{x_S \in \mathcal{D}_{S,1}} \bar{r}_S(x_S)^{-1} - 1 \biggr\}^2 \biggr] \nonumber \\
    % & = \biggl(\frac{Cn_S}{2 |\mathcal{D}_{S,1}|} \biggr)^2 \biggl[ \biggl( \frac{2 |\mathcal{D}_{S,1}|}{n_S} - 1 \biggr)^2 + \frac{4 |\mathcal{D}_{S,1}|}{n_S^2} \mathrm{Var}\{ \bar{r}_S(X_{S,1})^{-1}\} \biggr] \nonumber \\
    % & \leq C^2 \biggl( \frac{n_S}{2 |\mathcal{D}_{S,1}|} -1 \biggr)^2 + \frac{C^2}{c |\mathcal{D}_{S,1}|} \leq \frac{C^2}{(n_S+1)^2} + \frac{2C^2}{c n_S} \leq \frac{3C^2}{cn_S}.
\end{align}
Now, by conditioning on $\mathcal{D}_1$ and using the fact that $\int \alpha_S^{(M)}f_S =0$ we have that
\begin{align}
\label{Eq:CrossFitRem2}
     \bbE \biggl[ \biggl\{ \frac{1}{n_S} \sum_{j=1}^{n_S} & \frac{\hat{\alpha}_{S,(1)}^{(M)}(X_{S,j})}{\bar{r}_S(X_{S,j})}  \biggr\}^2 \biggr] = \bbE \biggl[  \biggl( \int \hat{\alpha}_{S,(1)}^{(M)} f_S \biggr)^2 + \frac{1}{n_S} \mathrm{Var} \biggl( \frac{\hat{\alpha}_{S,(1)}^{(M)}(X_{S,1})}{\bar{r}_S(X_{S,1})} \biggm| \mathcal{D}_1 \biggr) \biggr] \nonumber \\
     & \leq  \bbE \biggl[ \biggl\{ \int (\hat{\alpha}_{S,(1)}^{(M)} - \alpha_S^{(M)}) f_S \biggr\}^2 + \frac{1}{n_S} \int \frac{(\hat{\alpha}_{S,(1)}^{(M)})^2}{\bar{r}_S} f_S \biggr] \nonumber \\
     & \leq \bbE \biggl\{ \int \biggl( 1 + \frac{2}{n_S \bar{r}_S} \biggr) (\hat{\alpha}_{S,(1)}^{(M)} - \alpha_S^{(M)})^2 f_S \biggr\} + \frac{2}{n_S} \int \frac{(\alpha_S^{(M)})^2}{\bar{r}_S} f_S.
\end{align}
The calculations in~\eqref{Eq:CrossFitRem1} and~\eqref{Eq:CrossFitRem2} will allow us to control the error arising from using the estimator $\hat{r}_{S,(2)}$ in the second term of $R_{S,1}$. The other portion of the error is bounded as follows. Conditioning on $\mathcal{D}_1$ we have
\begin{align}
\label{Eq:CrossFitRem3}
    &\bbE\biggl[ \biggl\{\frac{1}{n}\sum_{x \in \mathcal{D}_{2}} \{ \hat{\alpha}_{S,(1)}^{(M)}(x_S) - \bar{\alpha}_S^{(M)}(x_S) \} -  \frac{|\mathcal{D}_2|}{n \cdot n_S} \sum_{j=1}^{n_S} \frac{\hat{\alpha}_{S,(1)}^{(M)}(X_{S,j}) - \bar{\alpha}_S^{(M)}(X_{S,j}) }{\bar{r}_S(X_{S,j})}  \biggr\}^2 \biggr] \nonumber \\
    & = \bbE \biggl[ \frac{|\mathcal{D}_2|}{n^2} \mathrm{Var} \bigl( \hat{\alpha}_{S,(1)}^{(M)}(X) - \bar{\alpha}_S^{(M)}(X) \bigm| \mathcal{D}_1 \bigr) + \frac{|\mathcal{D}_{2}|^2}{n^2 n_S} \mathrm{Var} \biggl( \frac{ \hat{\alpha}_{S,(1)}^{(M)}(X_{S,1}) - \bar{\alpha}_S^{(M)}(X_{S,1})}{\bar{r}_S(X_{S,1})} \biggm| \mathcal{D}_1 \biggr)  \biggr] \nonumber  \\
    & \leq \frac{|\mathcal{D}_2|}{n^2} \bbE \biggl\{ \int 
\frac{n+n_S \bar{r}_S}{n_S \bar{r}_S} (\hat{\alpha}_{S,(1)}^{(M)} - \alpha_S^{(M)})^2 f_S \biggr\},
\end{align}
where the first equality relies on the fact that the quantity in braces has mean zero conditionally on $\mathcal{D}_1$. Combining~\eqref{Eq:CrossFitRem1},~\eqref{Eq:CrossFitRem2} and~\eqref{Eq:CrossFitRem3} we now have
\begin{align*}
    \bbE(R_{S,1}^2) \leq \frac{2|\mathcal{D}_2|}{n^2} &\bbE \biggl\{ \int 
\frac{n+n_S \bar{r}_S}{n_S \bar{r}_S} (\hat{\alpha}_{S,(1)}^{(M)} - \alpha_S^{(M)})^2 f_S \biggr\} \\
    & + 2\Bigl( \frac{|\mathcal{D}_2|}{n} \Bigr)^2 \frac{C}{|\mathcal{D}_2|} \biggl[ \bbE\biggl\{  \int \Bigl(1 + \frac{2}{n_S \bar{r}_S} \Bigr)(\hat{\alpha}_{S,(1)}^{(M)} - \alpha_S^{(M)})^2 f_S \biggr\} + \frac{2}{n_S} \int \frac{(\alpha_S^{(M)})^2}{\bar{r}_S} f_S \biggr] \\
    & \leq \frac{2|\mathcal{D}_2|(C+1)}{n^2} \bbE \biggl\{ \int 
\frac{n+n_S \bar{r}_S}{n_S \bar{r}_S} (\hat{\alpha}_{S,(1)}^{(M)} - \alpha_S^{(M)})^2 f_S \biggr\} + \frac{4|\mathcal{D}_2| C}{n^2n_S} \int \frac{(\alpha_S^{(M)})^2}{\bar{r}_S} f_S.
\end{align*}
It follows from the fact that $\hat{\theta}-\theta^{*,(M)} = \sum_{S \in \bbS} (R_{S,1} + R_{S,2})$ and the line above that
\begin{align*}
    \bbE\{(\hat{\theta} - \theta^{*,(M)})^2\} &\leq 2|\bbS| \sum_{S \in \bbS} \{ \bbE(R_{S,1}^2) + \bbE(R_{S,2}^2) \}  \\
    & \leq \frac{8C|\bbS|}{n}\max_{\ell=1,2} \sum_{S \in \bbS} \biggl[ \bbE \biggl\{ \int 
\frac{n+n_S \bar{r}_S}{n_S \bar{r}_S} (\hat{\alpha}_{S,(\ell)}^{(M)} - \alpha_S^{(M)})^2 f_S \biggr\} + \frac{1}{n_S} \int \frac{(\alpha_S^{(M)})^2}{\bar{r}_S} f_S \biggr] \\
    & \leq \frac{8C|\bbS|}{n} \biggl[ \max_{\ell=1,2} \sum_{S \in \bbS} \bbE \biggl\{ \int 
\frac{n+n_S \bar{r}_S}{n_S \bar{r}_S} (\hat{\alpha}_{S,(\ell)}^{(M)} - \alpha_S^{(M)})^2 f_S \biggr\} + \mathrm{Var}(\theta^{*,(M)}) \biggr],
\end{align*}
as required.

We finally turn to the second statement of the proposition, where we give a bound on $\mathrm{Var}(\theta^{*,(M)})$. Recall that we have $\bbE\{\alpha_S^{(M)}(X_{S,1})/\bar{r}_S(X_{S,1})\} =  \bbE\{\alpha_S^{(M)}(X)\} = 0$ for each $S \in \bbS$. It therefore follows that $\theta^{*,(M)}$ is unbiased and that
\begin{align*}
    n & \mathrm{Var}(\theta^{*,(M)}) = \mathrm{Var} \biggl( a(X) - \sum_{S \in \bbS} \alpha_S^{(M)}(X_S) \biggr) + \sum_{S \in \bbS} \frac{n}{n_S} \mathrm{Var} \biggl( \frac{\alpha_S^{(M)}(X_{S,1})}{\bar{r}_S(X_{S,1})} \biggr) \nonumber \\
    & = \mathcal{L}(\alpha_\bbS^{(M)}) + \sum_{S \in \bbS} \biggl( \frac{n}{n_S} - \frac{1}{\lambda_S} \biggr) \int \frac{f_S}{\bar{r}_S} (\alpha_S^{(M)})^2 \leq \biggl( 1 + \max_{S \in \bbS} \biggl| \frac{n \lambda_S}{n_S} - 1 \biggr| \biggr) \mathcal{L}(\alpha_\bbS^{(M)}),
\end{align*}
as claimed.

\end{proof}

The proof of Proposition~\ref{Prop:InductionProof} to come relies on the following result.
\begin{lemma}
\label{Lemma:Smoothness}
Let $\beta_1,\beta_2,\beta_3 \in (0,1]$ and $L_1,L_2,L_3 \in (0, \infty)$, and suppose that $a$, the distribution of $X$ and $(r_S : S \in \bbS)$ satisfy (A1)($\beta_1,L_1$), (A2)($\beta_2,L_2$) and (A3)($\beta_3,L_3$), respectively. Then for any $m \in \bbN_0, \bS \in \bbS^{(m)}$ and $x_{S_m},x_{S_m}' \in \bbR^{S_m}$ with $\|x_{S_m}-x_{S_m}'\|_\infty \leq 1$ we haves
\[
    | \bar{a}_\bS^{(m)}(x_{S_m}) - \bar{a}_\bS^{(m)}(x_{S_m}') | \leq \bigl\{ L_1 + 2(2^{m}-1)\bigl( L_2 + L_3 \bigr)\|a\|_\infty \bigr\} \|x_{S_m}-x_{S_m}'\|_\infty^{\beta_\wedge},
\]
where  we write $\beta_\wedge = \beta_1 \wedge \beta_2 \wedge \beta_3$.
\end{lemma}

\begin{proof}[Proof of Lemma~\ref{Lemma:Smoothness}]
We begin be establishing simple uniform bounds on the functions $\bar{a}_\bS^{(m)}$. For any $\bS \in \bbS^{(m)}$ and $S \in \bbS$ such that $S \neq S_m$ and any $x_S\in \bbR^S$ we have that
\begin{equation}
\label{Eq:IteratedUniformBound}
    |\bar{a}_{(\bS,S)}^{(m+1)}(x_S)| = \biggl| \frac{\lambda_S \bar{r}_S(x_S)}{1+\lambda_S \bar{r}_S(x_S)} \bbE\{ \bar{a}_{\bS}^{(m)}(X) - \bar{\theta}_{(\bS,S)}^{(m+1)} | X_S = x_S\} \biggr| \leq 2 \|\bar{a}_\bS^{(m)}\|_\infty \leq \ldots \leq 2^{m+1}\|a\|_\infty.
\end{equation}
We now prove the main claim inductively, noting that the $m=0$ base case is a simple consequence of (A1)($\beta_1,L_1$). Now suppose that the claim holds at a given $m$ and let $\bS \in \bbS^{(m)}$ and $S \in \bbS$ be such that $S \neq S_m$. For any $x_S,x_S' \in \bbR^S$ with $\|x_S-x_S'\|_\infty \leq 1$ we have
\begin{align}
\label{Eq:SmoothnessDecomp}
    |\bar{a}_{(\bS,S)}^{(m+1)}(x_S) &- \bar{a}_{(\bS,S)}^{(m+1)}(x_S') | \nonumber \\
    &= \biggl| \frac{\lambda_S \bar{r}_S(x_S)}{1+\lambda_S \bar{r}_S(x_S)} \bbE\{ \bar{a}_{\bS}^{(m)}(X) - \bar{\theta}_{(\bS,S)}^{(m+1)} | X_S = x_S\} \nonumber \\
    & \hspace{150pt} - \frac{\lambda_S \bar{r}_S(x_S')}{1+\lambda_S \bar{r}_S(x_S')} \bbE\{ \bar{a}_{\bS}^{(m)}(X) - \bar{\theta}_{(\bS,S)}^{(m+1)} | X_S = x_S'\} \biggr| \nonumber \\
    & \leq  \bigl| \bbE\{ \bar{a}_{\bS}^{(m)}(X) | X_S = x_S\} - \bbE\{ \bar{a}_{\bS}^{(m)}(X) | X_S = x_S'\} \bigr| \nonumber \\
    & \hspace{150pt}+ 2 \|\bar{a}_\bS^{(m)}\|_\infty \biggl| \frac{\lambda_S \bar{r}_S(x_S)}{1+\lambda_S \bar{r}_S(x_S)} - \frac{\lambda_S \bar{r}_S(x_S')}{1+\lambda_S \bar{r}_S(x_S')} \biggr|.
\end{align}
Using (A2)($\beta_2,L_2$), our induction hypothesis and~\eqref{Eq:IteratedUniformBound} the first term in~\eqref{Eq:SmoothnessDecomp} can be controlled by writing
\begin{align}
\label{Eq:SmoothnessTerm1}
     \bigl| \bbE\{ \bar{a}_{\bS}^{(m)}&(x_S,X_{S^c}) | X_S = x_S\} - \bbE\{ \bar{a}_{\bS}^{(m)}(x_S',X_{S^c}) | X_S = x_S'\} \bigr| \nonumber \\
     & \leq \bigl| \bbE\{ \bar{a}_{\bS}^{(m)}(x_S,X_{S^c}) | X_S = x_S\} - \bbE\{ \bar{a}_{\bS}^{(m)}(x_S,X_{S^c}) | X_S = x_S'\} \bigr|  \nonumber \\
     & \hspace{100pt} + \bigl| \bbE\bigl\{ \bar{a}_{\bS}^{(m)}(x_S,X_{S^c}) - \bar{a}_{\bS}^{(m)}(x_S',X_{S^c}) | X_S = x_S'\bigr\} \bigr| \nonumber \\
     & \leq 2 L_2\|\bar{a}_\bS^{(m)}\|_\infty \|x_S-x_S'\|_\infty^{\beta_2}  + \bigl\{ L_1 + 2(2^{m}-1)\bigl( L_2 + L_3 \bigr)\|a\|_\infty \bigr\} \|x_{S}-x_{S}'\|_\infty^{\beta_\wedge} \nonumber \\
     & \leq 2^{m+1} L_2\|a\|_\infty \|x_S-x_S'\|_\infty^{\beta_2} + \bigl\{ L_1 + 2(2^{m}-1)\bigl( L_2 +  L_3 \bigr)\|a\|_\infty \bigr\} \|x_{S}-x_{S}'\|_\infty^{\beta_\wedge}.
\end{align}
Moreover, using (A3)($\beta_3,L_3$) the second term in \eqref{Eq:SmoothnessDecomp} can be written as
\begin{align}
\label{Eq:SmoothnessTerm2}
    2 \|\bar{a}_\bS^{(m)}\|_\infty \biggl| \frac{1}{1+\lambda_S \bar{r}_S(x_S)} - \frac{1}{1+\lambda_S \bar{r}_S(x_S')} \biggr| &\leq 2^{m+1}\|a\|_\infty \frac{\lambda_S |\bar{r}_S(x_S) - \bar{r}_S(x_S')|}{\{1+\lambda_S \bar{r}_S(x_S)\}\{1+\lambda_S \bar{r}_S(x_S')\}} \nonumber \\
    & \leq 2^{m+1}\|a\|_\infty L_3 \|x_S-x_S'\|_\infty^{\beta_3}.
\end{align}
Combining~\eqref{Eq:SmoothnessDecomp},~\eqref{Eq:SmoothnessTerm1} and~\eqref{Eq:SmoothnessTerm2} we see that
\begin{align*}
    |\bar{a}_{(\bS,S)}^{(m+1)}(x_S) &- \bar{a}_{(\bS,S)}^{(m+1)}(x_S') | \\
    & \leq \bigl\{ L_1 + \{2(2^{m}-1) + 2^{m+1}\}\bigl( L_2 +  L_3 \bigr)\|a\|_\infty \bigr\} \|x_{S}-x_{S}'\|_\infty^{\beta_\wedge} \\
    & = \bigl\{ L_1 + 2(2^{m+1}-1)\bigl( L_2 + L_3 \bigr)\|a\|_\infty \bigr\} \|x_{S}-x_{S}'\|_\infty^{\beta_\wedge}.
\end{align*}
as required.
\end{proof}

\begin{proof}[Proof of Proposition~\ref{Prop:InductionProof}]
We begin by expressing $\alpha_S^{(M)}$ and $\hat{\alpha}_{S,(1)}^{(M)}$ in terms of simpler pieces. To this end, given $m \in \mathbb{N}$ write $\bbS^{(m)} = \{ (S_1,\ldots,S_m) \in \bbS^m : S_{j+1} \neq S_j \text{ for } j\in [m-1]\}$ and recursively define functions $\bar{a}_\bS^{(m)}(\cdot)$ for $m \in \bbN$ and $\bS \in \bbS^{(m)}$ as follows. We initialise by setting $\bar{a}_S^{(1)}(x_S)$ to be equal to the right-hand side of~\eqref{Eq:SemiSupervised}. For $m \in \bbN$, $\bS=(S_1,\ldots,S_m) \in \bbS^{(m)}$ and $S \in \bbS \setminus \{S_m\}$ we then set
\[
    \bar{a}_{(\bS,S)}^{(m+1)}(x) = \frac{\lambda_S \bar{r}_S(x_S)}{1+\lambda_S \bar{r}_S(x_S)} \bbE\{ \bar{a}_{\bS}^{(m)}(X) - \bar{\theta}_{(\bS,S)}^{(m+1)} | X_S = x_S\}, \quad \text{where} \quad \bar{\theta}_{(\bS,S)}^{(m+1)} = \frac{\int \frac{\lambda_S \bar{r}_S}{1+\lambda_S \bar{r}_S} \bar{a}_{\bS}^{(m)} f}{\int \frac{\lambda_S \bar{r}_S}{1+\lambda_S \bar{r}_S} f_S}.
\]
For $M \in \bbN_0, m \in \mathbb{Z}$ and $\eta \in [0,1]$ let $B \sim \mathrm{Bin}(M,\eta)$ to introduce the notation $b_{M,\eta}(m)=\mathbb{P}(B \geq m)$, where we say that $B=0$ almost surely if $M=0$. We claim that, for $S \in \bbS$, $M \in \bbN_0$ and $\eta \in (0,1]$ we may write
\begin{equation}
\label{Eq:ApproxInfluence}
    \alpha_S^{(M)}(x)  = \sum_{m=1}^{M} (-1)^{m-1} \sum_{\bS \in \bbS^{(m)} : S_m = S} b_{M,\eta}(m) \bar{a}_{\bS}^{(m)}(x),
\end{equation}
where we note that $\alpha_S^{(M)}(x)$ is a function of $x_S$ only. The $M=0$ case is trivial and we proceed by induction on $M$. Assuming the induction hypothesis, for any $M \geq 0$ we have that
\begin{align}
\label{Eq:GradientDescent}
    &\alpha_S^{(M+1)} = \alpha_S^{(M)} - (\eta/2) \{\nabla \mathcal{L}(\alpha_\bbS^{(M)})\}_S \nonumber \\
    &= \alpha_S^{(M)}  - \eta \biggl[ \alpha_S^{(M)} - \frac{\lambda_S \bar{r}_S}{1+\lambda_S \bar{r}_S} (a_S - \bar{\theta}_S^{(1)}) + \frac{\lambda_S\bar{r}_S}{1+\lambda_S\bar{r}_S} \sum_{S' \neq S} \biggl\{ \mathbb{E}(\alpha_{S'}^{(M)} | X_S) -  \frac{ \int \frac{\lambda_S \bar{r}_S}{1+\lambda_S \bar{r}_S} \alpha_{S'}^{(M)} f}{\int \frac{\lambda_S \bar{r}_S}{1+\lambda_S \bar{r}_S} f} \biggr\}\biggr] \nonumber \\
    & = (1-\eta) \sum_{m=1}^{M} (-1)^{m-1} \sum_{\bS \in \bbS^{(m)}: S_m = S} b_{M,\eta}(m) \bar{a}_{\bS}^{(m)}  \nonumber \\
    & \hspace{150pt} + \eta \bar{a}_S^{(1)}  - \eta \sum_{S' \neq S} \sum_{m=1}^{M} (-1)^{m-1} \sum_{\bS \in \bbS^{(m)}: S_m = S'} b_{M,\eta}(m) \bar{a}_{(\bS,S)}^{(m+1)} \nonumber \\
    & = \sum_{m=1}^{M+1} (-1)^{m-1} \sum_{\bS \in \bbS^{(m)}: S_m = S} \bigl\{ (1-\eta)b_{M,\eta}(m) + \eta b_{M,\eta}(m-1) \bigr\}  \bar{a}_{\bS}^{(m)} \nonumber \\ 
    & = \sum_{m=0}^{M+1} (-1)^m \sum_{\bS \in \bbS^{(m)}: S_m = S} b_{M+1,\eta}(m)\bar{a}_{\bS}^{(m)},
\end{align}
as required.

Having decomposed $\alpha_\bbS^{(M)}$, we now aim for a similar decomposition of $\hat{\alpha}_{\bbS,(1)}^{(M)}$. Initialise by setting $\hat{a}^{(0)} \equiv a$ and, for $\bS \in \bbS^{(m)}$ and $(\boldsymbol{\ell},\ell)=(\ell_1,\ldots,\ell_m,\ell)$ with $1 \leq \ell_1 < \ldots < \ell_m < \ell \leq M$, iteratively define
\begin{align*}
    \hat{a}_{(\bS,S)}^{(\boldsymbol{\ell},\ell)}(x_S) &= \frac{\lambda_S r_S(x_S)}{\hat{r}_S^{(\ell)}+\lambda_S r_S(x_S)} \biggl\{ |\mathcal{D}_{1,\ell}|^{-1} \sum_{y \in \mathcal{D}_{1,\ell}} \hat{a}_{\bS}^{(\boldsymbol{\ell})}(y) \mathbbm{1}_{B_{T}}(y) \frac{K_h^S(x_S - y_S)}{\hat{f}_S^{(\ell)}(x_S)} \\
    & \hspace{190pt} - \frac{\sum_{y \in \mathcal{D}_{1,\ell}} \frac{\lambda_S r_S(y)}{\hat{r}_S^{(\ell)}+\lambda_S r_S(y)} \hat{a}_\bS^{(\boldsymbol{\ell})}(y) \mathbbm{1}_{B_T}(y)}{\sum_{y \in \mathcal{D}_{1,\ell}} \frac{\lambda_S r_S(y)}{\hat{r}_S^{(\ell)}+\lambda_S r_S(y)}} \biggr\},
\end{align*}
when $\|x_S\|_\infty \leq T$ and $\hat{a}_{(\bS,S)}^{(m+1)}(x_S) = 0$ otherwise. We claim that we may write
\[
    \hat{\alpha}_{S,(1)}^{(M_0)}(x_S) = \sum_{m=1}^{M_0} (-1)^{m-1} \sum_{\substack{ \bS \in \bbS^{(m)} \\ S_m=S}} \sum_{\substack{ \boldsymbol{\ell}=(\ell_1,\ldots,\ell_m) \in \bbN^m \\ 1 \leq \ell_1 < \ldots < \ell_m \leq M_0 }} \eta^m (1-\eta)^{M_0-m-\ell_1+1} \, \hat{a}_\bS^{(\boldsymbol{\ell})}(x_S)
\]
for any $M_0=0,1,\ldots,M$. Trivially, the base case $M_0=0$ holds as both sides of the equality are zero. Proceeding by induction, we have that
\begin{align*}
    &\hat{\alpha}_{S,(1)}^{(M_0+1)}(x_S) =  (1-\eta) \hat{\alpha}_S^{(M_0)}(x_S) \\
    &+ \eta \frac{\lambda_S r_S(x_S)}{\hat{r}_S^{(M_0+1)}+\lambda_S r_S(x_S)} \biggl[ |\mathcal{D}_{1,M_0+1}|^{-1} \sum_{y \in \mathcal{D}_{1,M_0+1}} \biggl\{a(y) - \sum_{S' \neq S} \hat{\alpha}_{S,(1)}^{(M_0)}(y) \biggr\} \mathbbm{1}_{B_{T}}(y) \frac{K_h^S(x_S - y_S)}{\hat{f}_S^{(M_0+1)}(x_S)} \\
    &  \hspace{50pt} -  \frac{\sum_{y \in \mathcal{D}_{1,M_0+1}} \frac{\lambda_S r_S(y)}{\hat{r}_S^{(M_0+1)}+\lambda_S r_S(y)} \{a(y) - \sum_{S' \neq S} \hat{\alpha}_{S,(1)}^{(M_0)}(y) \} \mathbbm{1}_{B_T}(y)}{\sum_{y \in \mathcal{D}_{1,M_0+1}} \frac{\lambda_S r_S(y)}{\hat{r}_S^{(M_0+1)}+\lambda_S r_S(y)}}  \biggr] \\
    & = (1-\eta) \sum_{m=1}^{M_0} (-1)^{m-1} \sum_{\substack{ \bS \in \bbS^{(m)} \\ S_m=S}} \sum_{\substack{ \boldsymbol{\ell}=(\ell_1,\ldots,\ell_m) \in \bbN^m \\ 1 \leq \ell_1 < \ldots < \ell_m \leq M_0 }} \eta^m (1-\eta)^{M_0-m-\ell_1+1} \, \hat{a}_\bS^{(\boldsymbol{\ell})}(x_S) + \eta \hat{a}_S^{(M_0+1)}(x_S) \\
    & \hspace{50pt} - \eta \sum_{S' \neq S} \sum_{m=1}^{M_0} (-1)^{m-1} \sum_{\substack{ \bS \in \bbS^{(m)} \\ S_m=S'}} \sum_{\substack{ \boldsymbol{\ell}=(\ell_1,\ldots,\ell_m) \in \bbN^m \\ 1 \leq \ell_1 < \ldots < \ell_m \leq M_0 }} \eta^m (1-\eta)^{M_0-m-\ell_1+1} \, \hat{a}_{(\bS,S)}^{(\boldsymbol{\ell},M_0+1)}(x_S) \\
    & = \sum_{m=1}^{M_0} (-1)^{m-1} \sum_{\substack{ \bS \in \bbS^{(m)} \\ S_m=S}} \sum_{\substack{ \boldsymbol{\ell}=(\ell_1,\ldots,\ell_m) \in \bbN^m \\ 1 \leq \ell_1 < \ldots < \ell_m \leq M_0 }} \eta^m (1-\eta)^{M_0-m-\ell_1+2} \, \hat{a}_\bS^{(\boldsymbol{\ell})}(x_S)  \\
    & \hspace{50pt} + \sum_{m=1}^{M_0+1} (-1)^{m-1} \sum_{\substack{ \bS \in \bbS^{(m)} \\ S_{m}=S}} \sum_{\substack{ \boldsymbol{\ell}=(\ell_1,\ldots,\ell_{m}) \in \bbN^{m} \\ 1 \leq \ell_1 < \ldots < \ell_{m} = M_0+1 }} \eta^{m} (1-\eta)^{M_0-m-\ell_1+2} \, \hat{a}_{\bS}^{(\boldsymbol{\ell})}(x_S) \\
    & = \sum_{m=1}^{M_0+1} (-1)^{m-1} \sum_{\substack{ \bS \in \bbS^{(m)} \\ S_{m}=S}} \sum_{\substack{ \boldsymbol{\ell}=(\ell_1,\ldots,\ell_{m}) \in \bbN^{m} \\ 1 \leq \ell_1 < \ldots < \ell_{m} \leq M_0+1 }} \eta^{m} (1-\eta)^{(M_0+1)-m-\ell_1+1} \, \hat{a}_{\bS}^{(\boldsymbol{\ell})}(x_S),
\end{align*}
as required.

We now reduce the error of $\hat{\alpha}_{S,(1)}^{(M)}$ to the error of its constituent parts in the ball of radius $T$ about the origin. This requires an understanding of the weights in the previous decompositions. First, for a fixed $m \in \bbN$ we have that
\begin{align*}
    \sum_{\substack{ \boldsymbol{\ell}=(\ell_1,\ldots,\ell_{m}) \in \bbN^{m} \\ 1 \leq \ell_1 < \ldots < \ell_{m} \leq M }} \eta^{m} (1-\eta)^{M-m-\ell_1+1} &= \sum_{\ell_1 = 1}^{M-m+1} \binom{M - \ell_1}{m-1} \eta^{m} (1-\eta)^{M-m-\ell_1+1} \\
    & = \sum_{\ell=m-1}^{M-1} \binom{\ell}{m-1} \eta^m (1-\eta)^{\ell-m+1} = b_{M,\eta}(m),
\end{align*}
where the final identity can be checked by induction on $M$. Now, recalling that we write $B \sim \mathrm{Bin}(M,\eta)$ with $\eta=|\bbS|^{-1}$, in the case that $|\bbS| >1$ we see that
\begin{align}
\label{Eq:PGF}
    \sum_{m=1}^M \sum_{\bS \in \bbS^{(m)} : S_m =S} b_{M,\eta}(m) &\leq \sum_{m=1}^M |\bbS|^{m-1} b_{M,\eta}(m) = \sum_{j=1}^M \biggl( \sum_{m=1}^j |\bbS|^{m-1} \biggr) \bbP(B = j) \nonumber \\
    & = \sum_{j=1}^M \frac{|\bbS|^{j} - 1}{|\bbS|-1} \bbP(B=j) = \frac{\{(1-\eta)+\eta |\bbS|\}^M}{|\bbS|-1} \leq 2^M,
\end{align}
and it is straightforward to see that this bound continues to hold when $|\bbS|=1$. Now we have by the Cauchy--Schwarz inequality that
\begin{align}
\label{Eq:AlphaHatDecomp}
    &\bbE\int (\hat{\alpha}_{S,(1)}^{(M)} - \alpha_S^{(M)})^2 f_S \nonumber \\
    & = \bbE \int  \biggl[ \sum_{m=1}^{M} (-1)^{m-1} \sum_{\substack{ \bS \in \bbS^{(m)} \\ S_{m}=S}} \sum_{\substack{ \boldsymbol{\ell}=(\ell_1,\ldots,\ell_{m}) \in \bbN^{m} \\ 1 \leq \ell_1 < \ldots < \ell_{m} \leq M }} \eta^{m} (1-\eta)^{(M_0+1)-m-\ell_1+1} \bigl( \hat{a}_{\bS}^{(\boldsymbol{\ell})}  - \bar{a}_{\bS}^{(m)} \bigr) \biggr]^2 f_S \nonumber \\
    & \leq 2^{2M} \max_{m=1,\ldots,M} \max_{\bS \in \bbS^{(m)}: S_m = S} \bbE \int \bigl( \hat{a}_{\bS}^{(1,\ldots,m)} - \bar{a}_{\bS}^{(m)} \bigr)^2 f_S,
\end{align}
where the final inequality we use the fact that $\hat{a}_\bS^{(\boldsymbol{\ell})}$ has the same distribution for any $\boldsymbol{\ell}$ in the range of the sum. For notational simplicity, in the remainder of the proof we write $\hat{a}_{\bS}^{(m)}$ for $\hat{a}_{\bS}^{(1,\ldots,m)}$ For $S \in \bbS$ write $B_T^S = \{x_S' \in \mathbb{R}^S : \|x_S'\|_\infty \leq T\}$. The contribution to the error from outside this ball can be bounded using
\begin{equation}
\label{Eq:BTSComplement}
    \int_{(B_T^S)^c} \bigl\{ \hat{a}_{\bS}^{(m)}(x_S) - \bar{a}_{\bS}^{(m)}(x_S) \bigr\}^2 f_S(x_S) \,dx_S \leq 2^{2m} \mathbb{P}(\|X\|_\infty \geq T),
\end{equation}
which follows from the fact that $\hat{a}_{\bS}^{(m)}(x_S)=0$ when $x_S \not\in B_T^S$ and~\eqref{Eq:IteratedUniformBound}.

It now suffices to bound the error of the individual estimators $\hat{a}_{(\bS,S)}^{(m+1)}(x_S)$ when $\|x_S\|_\infty \leq T$ and $\bS \in \bbS^{(m-1)}$ with $S_m \neq S$. 
% We will argue inductively that for $m \in \bbN$ and $x_S \in B_{T}$ we have \tb{assuming that $\|a\|_\infty \geq1$}
% \begin{align}
% \label{Eq:InductionHypothesis}
%     \bbE& \bigl\{ \hat{a}_{(\bS,S)}^{(m+1)}(x_S) - \bar{a}_{(\bS,S)}^{(m+1)}(x_S) \bigr\}^2 \nonumber \\
%     &\leq \|a\|_\infty^2 A^{m+1} \biggl[ \frac{T^{(m+1)d}}{|\mathcal{D}_{1,m+1}| h^d} \biggl\{1 + \frac{1}{ (K_h^S \ast f_S)(x_S)} \biggr\} + h^{2 \beta_\wedge} \nonumber \\
%     & \hspace{125pt} + \mathbb{P}(\|X\|_\infty \geq T) + \frac{\bbE\{\mathbbm{1}_{B_T^c}(X) K_{(m+1)h}^S(x_S - (X)_S)\}}{(K_h^S \ast f_S)(x_S)}   \biggr],
% \end{align}
% where $A = 224\{1 + (C/c)^2 + (L_1+L_2+L_3/c)^2\}$ \red{check}. Integrating over $x_S$ will then give us an $L_2$ bound on the estimation error of the estimators $\hat{a}_{(\bS,S)}^{(m+1)}$. We use a stronger induction hypothesis as it facilitates better bounds. 
Using the fact that $\frac{\lambda_S r_S(x_S)}{\hat{r}_S^{(m+1)}+\lambda_S r_S(x_S)} \leq 1$ and the triangle inequality we have that
\begin{align}
\label{Eq:RegressionErrorDecomp}
    & \bigl| \hat{a}_{(\bS,S)}^{(m+1)}(x_S) - \bar{a}_{(\bS,S)}^{(m+1)}(x_S) \bigr| \nonumber \\
    & \leq \bigl[ |\bbE\{\bar{a}_{\bS}^{(m)}(X) | X_S=x_S)\}| + |\bar{\theta}_{(\bS,S)}^{(m+1)}| \bigr] \biggl| \frac{\lambda_S r_S(x_S)}{\hat{r}_S^{(m+1)}+\lambda_S r_S(x_S)} - \frac{\lambda_S \bar{r}_S(x_S)}{1+\lambda_S \bar{r}_S(x_S)} \biggr| \nonumber \\
    & \hspace{50pt} + \biggl| |\mathcal{D}_{1,m+1}|^{-1} \sum_{y \in \mathcal{D}_{1,m+1}} \{\mathbbm{1}_{B_{T}}(y) \hat{a}_{\bS}^{(m)}(y) - \bar{a}_{\bS}^{(m)}(y) \}\frac{K_h^S(x_S - y_S)}{\hat{f}_S^{(m+1)}(x_S)} \biggr| \nonumber \\
    & \hspace{50pt} +  \biggl| |\mathcal{D}_{1,m+1}|^{-1} \sum_{y \in \mathcal{D}_{1,m+1}} \bar{a}_{\bS}^{(m)}(y) \frac{K_h^S(x_S - y_S)}{\hat{f}_S^{(m+1)}(x_S)} - \bbE\{\bar{a}_{\bS}^{(m)}(X) | X_S=x_S\} \biggr| \nonumber \\
    & \hspace{50pt} + \biggl| \frac{\sum_{y \in \mathcal{D}_{1,m+1}} \frac{\lambda_S r_S(y)}{\hat{r}_S^{(m+1)}+\lambda_S r_S(y)} \{ \hat{a}_\bS^{(m)}(y) - \bar{a}_\bS^{(m)}(y)\} \mathbbm{1}_{B_T}(y)}{\sum_{y \in \mathcal{D}_{1,m+1}} \frac{\lambda_S r_S(y)}{\hat{r}_S^{(m+1)}+\lambda_S r_S(y)}} \biggr| \nonumber \\
    & \hspace{50pt} +  \biggl| \frac{\sum_{y \in \mathcal{D}_{1,m+1}} \{ \frac{\lambda_S r_S(y)}{\hat{r}_S^{(m+1)}+\lambda_S r_S(y)} - \frac{\lambda_S \bar{r}_S(y)}{1+\lambda_S \bar{r}_S(y)} \} \bar{a}_\bS^{(m)}(y) \mathbbm{1}_{B_T}(y) }{\sum_{y \in \mathcal{D}_{1,m+1}} \frac{\lambda_S r_S(y)}{\hat{r}_S^{(m+1)}+\lambda_S r_S(y)}} \biggr| \nonumber \\
    & \hspace{50pt} +  \biggl|\frac{\sum_{y \in \mathcal{D}_{1,m+1}} \frac{\lambda_S \bar{r}_S(y)}{1+\lambda_S \bar{r}_S(y)} \bar{a}_\bS^{(m)}(y)\mathbbm{1}_{B_T}(y)}{\sum_{y \in \mathcal{D}_{1,m+1}} \frac{\lambda_S r_S(y)}{\hat{r}_S^{(m+1)}+\lambda_S r_S(y)}} - \frac{\sum_{y \in \mathcal{D}_{1,m+1}} \frac{\lambda_S \bar{r}_S(y)}{1+\lambda_S \bar{r}_S(y)} \bar{a}_\bS^{(m)}(y) \mathbbm{1}_{B_T}(y)}{\sum_{y \in \mathcal{D}_{1,m+1}} \frac{\lambda_S \bar{r}_S(y)}{1+\lambda_S \bar{r}_S(y)}} \biggr| \nonumber \\
    & \hspace{50pt} +  \biggl| \frac{\sum_{y \in \mathcal{D}_{1,m+1}} \frac{\lambda_S \bar{r}_S(y)}{1+\lambda_S \bar{r}_S(y)} \bar{a}_\bS^{(m)}(y)\mathbbm{1}_{B_T}(y)}{\sum_{y \in \mathcal{D}_{1,m+1}} \frac{\lambda_S \bar{r}_S(y)}{1+\lambda_S \bar{r}_S(y)}} - \bar{\theta}_{(\bS,S)}^{(m+1)} \biggr| \nonumber \\
    & =: \sum_{j=1}^7 R_{(\bS,S),j}^{(m+1)}(x_S).
\end{align}
We now bound each of these error terms separately. 

\underline{To bound $R_1$:} 
% Arguing inductively, we have for any $m,\bS,S$ that
% \begin{align*}
%     \bbE\{ \bar{a}_{(\bS,S)}^{(m+1)}(X)^2 \} &\leq \bbE \biggl[ \frac{\lambda \bar{r}_S(X_S)}{1+\lambda_S \bar{r}_S(X_S)} \bigl\{ \bar{a}_\bS^{(m)}(X) - \bar{\theta}_{(\bS,S)}^{(m+1)} \bigr\}^2 \biggr] \leq \bbE \biggl[ \frac{\lambda \bar{r}_S(X_S)}{1+\lambda_S \bar{r}_S(X_S)} \bar{a}_\bS^{(m)}(X)^2 \biggr] \\
%     & \leq \bbE \bigl\{ \bar{a}_\bS^{(m)}(X)^2 \bigr\} \leq \ldots \leq \bbE \{a(X)^2\}.
% \end{align*}
% Moreover, by Jensen's inequality and the above line we have
% \begin{align*}
%     |\bar{\theta}_{(\bS,S)}^{(m+1)}|^2 \leq \frac{\int \frac{\lambda_S \bar{r}_S}{1+\lambda_S \bar{r}_S} (\bar{a}_\bS^{(m)})^2 f_S}{\int \frac{\lambda_S \bar{r}_S}{1+\lambda_S \bar{r}_S} f_S} \leq \frac{\lambda_S C}{1+\lambda_S C} \frac{1+\lambda_S c}{\lambda_S c} \bbE \bigl\{ \bar{a}_\bS^{(m)}(X)^2 \bigr\} \leq \frac{C}{c} \bbE\{a(X)^2\}.
% \end{align*}
Using the fact that $|a/(1+\epsilon+a) - a/(1+a)| \leq |\epsilon|$ for any $a \geq 0$ and $\epsilon \geq -1$, and using very similar arguments to~\eqref{Eq:CrossFitRem1} and~\eqref{Eq:IteratedUniformBound} above, we have that
\begin{align}
\label{Eq:R1bound}
    \bbE  \{R_{(\bS,S),1}^{(m+1)}(x_S)^2 \} & \leq \bigl[ |\bbE\{\bar{a}_{\bS}^{(m)}(X) | X_S=x_S)\}| + |\bar{\theta}_{(\bS,S)}^{(m+1)}| \bigr]^2 \bbE\biggl\{\biggl[ \frac{\hat{r}_S^{(m+1)}}{\bbE\{r_S(X_S)\}} -1 \biggr]^2 \biggr\} \nonumber \\
    & \leq \frac{2^{2m+2}C \|a\|_\infty^2}{|\mathcal{D}_{1,m+1}|}.
\end{align}

\underline{To bound $R_2$:} We use the notation $f_{S,h}(x_S) = (K_h^S \ast f_S)(x_S)$. By the Cauchy--Schwarz inequality and, using the fact that $K$ is uniform, Lemma~4.1(i) of~\cite{gyorfi2006distribution}, we have that
\begin{align}
\label{Eq:R2bound}
    \bbE \{ &R_{(\bS,S),2}^{(m+1)}(x_S)^2 \} = \bbE \biggl[ \biggl\{ \frac{\sum_{y \in \mathcal{D}_{1,m+1}} \{ \mathbbm{1}_{B_{T}}(y) \hat{a}_\bS^{(m)}(y) - \bar{a}_\bS^{(m)}(y) \} K_h^S(x_S - y_S)} {\sum_{y \in \mathcal{D}_{1,m+1}} K_h^S(x_S - y_S)} \biggr\}^2  \biggr] \nonumber \\
    & \leq \bbE \biggl[ \sum_{y \in \mathcal{D}_{1,m+1}} \bbE \biggl\{ \frac{ \{ \mathbbm{1}_{B_{T}}(y) \hat{a}_\bS^{(m)}(y) - \bar{a}_\bS^{(m)}(y) \}^2 K_h^S(x_S - y_S)} {\sum_{y' \in \mathcal{D}_{1,m+1}} K_h^S(x_S - y_S')} \biggm| y  \biggr\} \biggr] \nonumber \\
    & \leq \frac{1}{|\mathcal{D}_{1,m+1}| f_{S,h}(x_S)} \mathbb{E} \biggl[ \sum_{y \in \mathcal{D}_{1,m+1}} \{ \mathbbm{1}_{B_{T}}(y) \hat{a}_\bS^{(m)}(y) - \bar{a}_\bS^{(m)}(y) \}^2 K_h^S(x_S - y_S)\biggr] \nonumber \\
    & = \frac{1}{f_{S,h}(x_S)} \bbE \bigl[ \bigl\{ \mathbbm{1}_{B_{T}}(X) \hat{a}_\bS^{(m)}(X) - \bar{a}_\bS^{(m)}(X) \bigr\}^2  K_h^S(x_S - X_S) \bigr] \nonumber \\
    & = \frac{1}{f_{S,h}(x_S)} \bbE \biggl[ \biggl\{ \mathbbm{1}_{B_T}(X) \bigl\{ \hat{a}_\bS^{(m)}(X) - \bar{a}_\bS^{(m)}(X) \bigr\}^2 + \mathbbm{1}_{B_T^c}(X) \bar{a}_\bS^{(m)}(X)^2  \biggr\} K_h^S(x_S - X_S) \biggr] \nonumber \\
    & \leq \frac{1}{f_{S,h}(x_S)} \bbE \bigl[\mathbbm{1}_{B_T}(X) \bigl\{ \hat{a}_\bS^{(m)}(X) - \bar{a}_\bS^{(m)}(X) \bigr\}^2 K_h^S(x_S - X_S) \bigr] \nonumber \\
    & \hspace{175pt} + \frac{2^{2m} \|a\|_\infty^2}{f_{S,h}(x_S)} \bbE \bigl\{ \mathbbm{1}_{B_T^c}(X) K_h^S(x_S - X_S) \bigr\},
\end{align}
where the final inequality uses~\eqref{Eq:IteratedUniformBound}.

\underline{To bound $R_3$:} Our bounds on this term are based on ideas from the theory of nonparametric regression with kernels, for which a good reference is Chapter~5 of~\cite{gyorfi2006distribution}. Writing $\tilde{a}_{(\bS,S)}^{(m+1)}(x_S) = \bbE\{\bar{a}_{\bS}^{(m)}(X) | X_S=x_S\}$, we may decompose this error term by writing
\begin{align}
\label{Eq:R3decomp}
    &\bbE  \{R_{(\bS,S),3}^{(m+1)}(x_S)^2 \} \nonumber \\
    &= \bbE \biggl[ \biggl\{ |\mathcal{D}_{1,m+1}|^{-1} \sum_{y \in \mathcal{D}_{1,m+1}} \bar{a}_{\bS}^{(m)}(y) \frac{K_h^S(x_S - y_S)}{\hat{f}_S^{(m+1)}(x_S)} - \tilde{a}_{(\bS,S)}^{(m+1)}(x_S) \biggr\}^2 \mathbbm{1}_{\{ \hat{f}_S^{(m+1)}(x_S) >0 \}}\biggr] \nonumber \\
    & \hspace{50pt} + \tilde{a}_{(\bS,S)}^{(m+1)}(x_S)^2 \mathbb{P}(\hat{f}_S^{(m+1)}(x_S)=0) \nonumber \\
    & = \bbE \biggl[ \biggl\{ |\mathcal{D}_{1,m+1}|^{-1} \sum_{y \in \mathcal{D}_{1,m+1}} \{ \bar{a}_{\bS}^{(m)}(y) -  \tilde{a}_{(\bS,S)}^{(m+1)}(y_S) \} \frac{K_h^S(x_S - y_S)}{\hat{f}_S^{(m+1)}(x_S)} \nonumber \\
    & \hspace{50pt} + |\mathcal{D}_{1,m+1}|^{-1} \sum_{y \in \mathcal{D}_{1,m+1}} \{\tilde{a}_{(\bS,S)}^{(m+1)}(y_S) - \tilde{a}_{(\bS,S)}^{(m+1)}(x_S) \} \frac{K_h^S(x_S - y_S)}{\hat{f}_S^{(m+1)}(x_S)} \biggr\}^2 \biggr] \nonumber\\
    & \hspace{50pt} + \tilde{a}_{(\bS,S)}^{(m+1)}(x_S)^2 \mathbb{P}(\hat{f}_S^{(m+1)}(x_S)=0) \nonumber \\
    & = |\mathcal{D}_{1,m+1}|^{-2} \bbE \biggl[ \sum_{y \in \mathcal{D}_{1,m+1}} \{ \bar{a}_{\bS}^{(m)}(y) -  \tilde{a}_{(\bS,S)}^{(m+1)}(y_S) \}^2 \frac{K_h^S(x_S - y_S)^2}{\hat{f}_S^{(m+1)}(x_S)^2} \biggr] \nonumber \\
    & \hspace{50pt} + \bbE \biggl[ \biggl\{ |\mathcal{D}_{1,m+1}|^{-1} \sum_{y \in \mathcal{D}_{1,m+1}} \{\tilde{a}_{(\bS,S)}^{(m+1)}(y_S) - \tilde{a}_{(\bS,S)}^{(m+1)}(x_S) \} \frac{K_h^S(x_S - y_S)}{\hat{f}_S^{(m+1)}(x_S)} \biggr\}^2 \biggr] \nonumber \\
    & \hspace{50pt} + \tilde{a}_{(\bS,S)}^{(m+1)}(x_S)^2 \mathbb{P}(\hat{f}_S^{(m+1)}(x_S)=0).
\end{align}
We proceed by bounding these three terms separately. Since $K$ is the uniform kernel, we may appeal to Lemma~4.1(ii) of~\cite{gyorfi2006distribution} and~\eqref{Eq:IteratedUniformBound} to see that the first term can be bounded by
\begin{align}
\label{Eq:R3term1}
    \frac{2^{2m} \|a\|_\infty^2}{|\mathcal{D}_{1,m+1}|^2}  &\bbE \biggl[ \sum_{y \in \mathcal{D}_{1,m+1}} \frac{(2h)^{-|S|} K_h^S(x_S - y_S)}{\hat{f}_S^{(m+1)}(x_S)^2} \biggr] = \frac{2^{2m-|S|}\|a\|_\infty^2}{|\mathcal{D}_{1,m+1}|} \bbE \biggl[ \frac{h^{-|S|} \mathbbm{1}_{\{\hat{f}_S^{(m+1)}(x_S) > 0 \}}}{\hat{f}_S^{(m+1)}(x_S)} \biggr]  \nonumber \\
    & \leq \frac{2^{2m}\|a\|_\infty^2}{|\mathcal{D}_{1,m+1}| h^{|S|} f_{S,h}(x_S)}.
\end{align}
Using~\eqref{Eq:IteratedUniformBound}, the third term on the right-hand side of~\eqref{Eq:R3decomp} can be bounded by
\begin{align}
\label{Eq:R3term3}
    2^{2m} \|a\|_\infty^2 \bigl\{1 - (2h)^{|S|} f_{S,h}(x_S) \bigr\}^{|\mathcal{D}_{1,m+1}|} &\leq 2^{2m} \|a\|_\infty^2 \exp \bigl( - |\mathcal{D}_{1,m+1}| (2h)^{|S|} f_{S,h}(x_S) \bigr) \nonumber \\
    &\leq \frac{2^{2m}\|a\|_\infty^2}{|\mathcal{D}_{1,m+1}| (2h)^{|S|} f_{S,h}(x_S)}.
\end{align}
We bound the second term on the right-hand side of~\eqref{Eq:R3decomp} using the smoothness properties of $\tilde{a}_{(\bS,S)}^{(m+1)}$. It follows from Lemma~\ref{Lemma:Smoothness} that this term can be bounded by
\begin{align}
\label{Eq:R3term2}
    \bigl\{ L_1 + 2(2^{m}-1)( L_2 +  L_3 )\|a\|_\infty \bigr\}^2 h^{2\beta_\wedge}.
\end{align}
It now follows from~\eqref{Eq:R3decomp},~\eqref{Eq:R3term1},~\eqref{Eq:R3term3} and~\eqref{Eq:R3term2} that
\begin{equation}
\label{Eq:R3bound}
    \bbE  \{R_{(\bS,S),3}^{(m+1)}(x_S)^2 \} \leq \frac{3 \times 2^{2m}\|a\|_\infty^2}{|\mathcal{D}_{1,m+1}| h^{|S|} f_{S,h}(x_S)} + \bigl\{ L_1 + 2(2^{m}-1)( L_2 +  L_3 )\|a\|_\infty \bigr\}^2 h^{2\beta_\wedge}.
\end{equation}

\underline{To bound $R_4$:} This term, as well as those that follow, do not vary with $x_S$. We may use Cauchy--Schwarz to write
\begin{align}
\label{Eq:R4bound}
    \bbE  \{R_{(\bS,S),4}^{(m+1)}(x_S)^2 \}&= \bbE \biggl[ \biggl\{ \frac{\sum_{y \in \mathcal{D}_{1,m+1}} \frac{\lambda_S r_S(y)}{\hat{r}_S^{(m+1)}+\lambda_S r_S(y)} \{\hat{a}_\bS^{(m)}(y) - \bar{a}_\bS^{(m)}(y)\}\mathbbm{1}_{B_T}(y)}{\sum_{y \in \mathcal{D}_{1,m+1}} \frac{\lambda_S r_S(y)}{\hat{r}_S^{(m+1)}+\lambda_S r_S(y)} } \biggr\}^2 \biggr] \nonumber \\
    & \leq \bbE \biggl[ \frac{\sum_{y \in \mathcal{D}_{1,m+1}} \frac{\lambda_S r_S(y)}{\hat{r}_S^{(m+1)}+\lambda_S r_S(y)} \{\hat{a}_\bS^{(m)}(y) - \bar{a}_\bS^{(m)}(y)\}^2\mathbbm{1}_{B_T}(y)}{\sum_{y \in \mathcal{D}_{1,m+1}} \frac{\lambda_S r_S(y)}{\hat{r}_S^{(m+1)}+\lambda_S r_S(y)} } \biggr] \nonumber \\
    & \leq \frac{C}{c} \bbE \bigl[ \{ \hat{a}_\bS^{(m)}(X) - \bar{a}_\bS^{(m)}(X) \}^2 \mathbbm{1}_{B_T}(X) \bigr],
\end{align}
where the final inequality follows from the assumption that $c \leq r_S(y) \leq C$ for all $y \in \bbR^d$.

\underline{To bound $R_5$:} Using the fact that $|a/(1+\epsilon+a) - a/(1+a)| \leq |\epsilon| a/(1+a+\epsilon)$ whenever $a \geq 0$ and $\epsilon \geq -1$, we may write
\begin{align}
\label{Eq:R5bound}
 &  \bbE  \{R_{(\bS,S),5}^{(m+1)}(x_S)^2 \} \nonumber\\
    & = \bbE \biggl[ \biggl\{ \frac{\sum_{y \in \mathcal{D}_{1,m+1}}\{\frac{\lambda_S r_S(y)}{\hat{r}_S^{(m+1)}+\lambda_S r_S(y)} - \frac{\lambda_S \bar{r}_S(y)}{1+\lambda_S \bar{r}_S(y)}\} \bar{a}_\bS^{(m)}(y) \mathbbm{1}_{B_T}(y) }{\sum_{y \in \mathcal{D}_{1,m+1}} \frac{\lambda_S r_S(y)}{\hat{r}_S^{(m+1)}+\lambda_S r_S(y)} } \biggr\}^2 \biggr] \nonumber \\
    & \leq \bbE \biggl[\biggl\{ \frac{\hat{r}_S^{(m+1)}}{\bbE\{r_S(X_S)\}} -1 \biggr\}^2 \biggl\{ \frac{\sum_{y \in \mathcal{D}_{1,m+1}}\frac{\lambda_S r_S(y)}{\hat{r}_S^{(m+1)}+\lambda_S r_S(y)} \bar{a}_\bS^{(m)}(y) \mathbbm{1}_{B_T}(y)}{\sum_{y \in \mathcal{D}_{1,m+1}} \frac{\lambda_S r_S(y)}{\hat{r}_S^{(m+1)}+\lambda_S r_S(y)} } \biggr\}^2 \biggr] \nonumber \\
    & \leq \frac{C}{c} \frac{1}{|\mathcal{D}_{1,m+1}|^4} \bbE \biggl[ \biggl\{  \sum_{y \in \mathcal{D}_{1,m+1}} \bigl| \bar{a}_\bS^{(m)}(y) \bigr| \biggr\}^2 \biggl\{  \sum_{y \in \mathcal{D}_{1,m+1}} \bar{r}_S(y) - |\mathcal{D}_{1,m+1}| \biggr\}^2  \biggr] \nonumber \\
    & = \frac{C}{c} \frac{1}{|\mathcal{D}_{1,m+1}|^4} \bbE \biggl[ \sum_{y_1,y_2,y_3,y_4 \in \mathcal{D}_{1,m+1}}|\bar{a}_\bS^{(m)}(y_1)| |\bar{a}_\bS^{(m)}(y_2)| \{ \bar{r}_S(y_3)-1\} \{ \bar{r}_S(y_4)-1\} \biggr] \nonumber \\
    & \leq \frac{C^2}{c} \frac{3}{|\mathcal{D}_{1,m+1}|} 2^{2m} \|a\|_\infty^2,
\end{align}
where the final inequality follows from~\eqref{Eq:IteratedUniformBound} and the fact that terms in the sum where $y_1,y_2,y_3,y_4$ are distinct have mean zero.
% Arguing inductively, we have for any $m,\bS,S$ that
% \begin{align*}
%     \bbE\{ \bar{a}_{(\bS,S)}^{(m+1)}(X)^2 \} &\leq \bbE \biggl[ \frac{\lambda \bar{r}_S(X_S)}{1+\lambda_S \bar{r}_S(X_S)} \bigl\{ \bar{a}_\bS^{(m)}(X) - \bar{\theta}_{(\bS,S)}^{(m+1)} \bigr\}^2 \biggr] \leq \bbE \biggl[ \frac{\lambda \bar{r}_S(X_S)}{1+\lambda_S \bar{r}_S(X_S)} \bar{a}_\bS^{(m)}(X)^2 \biggr] \\
%     & \leq \bbE \bigl\{ \bar{a}_\bS^{(m)}(X)^2 \bigr\} \leq \ldots \leq \bbE \{a(X)^2\}.
% \end{align*}
% It now follows from~\eqref{Eq:R5bounda} that
% \begin{equation}
% \label{Eq:R5bound}
%     \bbE  \{R_{(\bS,S),5}^{(m+1)}(x_S)^2 \} \leq 6\biggl( \frac{C}{c} \biggr)^4 \frac{\|a\|_\infty^2}{|\mathcal{D}_{1,m+1}|}.
% \end{equation}

\underline{To bound $R_6$:} This term can be bounded similarly to the one above. Indeed, again using the fact that $|a/(1+\epsilon+a) - a/(1+a)| \leq |\epsilon| a/(1+a+\epsilon)$ whenever $a \geq 0$ and $\epsilon \geq -1$, we may write
\begin{align}
\label{Eq:R6bound}
    & \bbE  \{R_{(\bS,S),6}^{(m+1)}(x_S)^2 \} \nonumber\\
    & = \bbE \biggl[ \biggl\{ \frac{\sum_{y \in \mathcal{D}_{1,m+1}} \frac{\lambda_S \bar{r}_S(y)}{1+\lambda_S \bar{r}_S(y)} \bar{a}_\bS^{(m)}(y) \mathbbm{1}_{B_T}(y)}{( \sum_{y \in \mathcal{D}_{1,m+1}} \frac{\lambda_S r_S(y)}{\hat{r}_S^{(m+1)}+\lambda_S r_S(y)} )(\sum_{y \in \mathcal{D}_{1,m+1}} \frac{\lambda_S \bar{r}_S(y)}{1+\lambda_S \bar{r}_S(y)})} \biggr\}^2 \nonumber \\
    & \hspace{100pt} \times \biggl\{\sum_{y \in \mathcal{D}_{1,m+1}} \frac{\lambda_S r_S(y)}{\hat{r}_S^{(m+1)}+\lambda_S r_S(y)} - \sum_{y \in \mathcal{D}_{1,m+1}} \frac{\lambda_S \bar{r}_S(y)}{1+\lambda_S \bar{r}_S(y)}  \biggr\}^2 \biggr] \nonumber \\
    & \leq \bbE \biggl[ \biggl\{ \frac{\sum_{y \in \mathcal{D}_{1,m+1}} \frac{\lambda_S \bar{r}_S(y)}{1+\lambda_S \bar{r}_S(y)} \bar{a}_\bS^{(m)}(y) \mathbbm{1}_{B_T}(y)}{\sum_{y \in \mathcal{D}_{1,m+1}} \frac{\lambda_S \bar{r}_S(y)}{1+\lambda_S \bar{r}_S(y)}} \biggr\}^2 \biggl\{ \frac{\hat{r}_S^{(m+1)}}{\bbE\{r_S(X_S)\}} -1 \biggr\}^2 \biggr] \nonumber \\
    & \leq \frac{C}{c} \frac{1}{|\mathcal{D}_{1,m+1}|^4} \bbE \biggl[ \biggl\{  \sum_{x \in \mathcal{D}_{1,m+1}} \bigl| \bar{a}_\bS^{(m)}(x) \bigr| \biggr\}^2 \biggl\{  \sum_{x \in \mathcal{D}_{1,m+1}} \bar{r}_S(x) - |\mathcal{D}_{1,m+1}| \biggr\}^2  \biggr] \nonumber \\
    & \leq \frac{C^2}{c} \frac{3}{|\mathcal{D}_{1,m+1}|} 2^{2m} \|a\|_\infty^2,
\end{align}
where the final inequality uses from~\eqref{Eq:R5bound} above.

\underline{To bound $R_7$:} We first deal with the error incurred by truncating to $B_T$. We have
\begin{align}
\label{Eq:R7bound1}
    \biggl( \bar{\theta}_{(\bS,S)}^{(m+1)} - \frac{\int_{B_T} \frac{\lambda_S \bar{r}_S}{1+\lambda_S \bar{r}_S} \bar{a}_{\bS}^{(m)} f}{\int \frac{\lambda_S \bar{r}_S}{1+\lambda_S \bar{r}_S} f}  \biggr)^2 &= \biggl( \frac{\int_{B_T^c} \frac{\lambda_S \bar{r}_S}{1+\lambda_S \bar{r}_S} \bar{a}_{\bS}^{(m)} f_S}{\int \frac{\lambda_S \bar{r}_S}{1+\lambda_S \bar{r}_S} f_S}  \biggr)^2 \leq \frac{\int_{B_T^c} \frac{\lambda_S \bar{r}_S}{1+\lambda_S \bar{r}_S} (\bar{a}_{\bS}^{(m)})^2 f_S}{\int \frac{\lambda_S \bar{r}_S}{1+\lambda_S \bar{r}_S} f_S} \nonumber \\
    &\leq \frac{C 2^{2m}\|a\|_\infty^2}{c} \mathbb{P}(\|X\|_\infty \geq T).
\end{align}
The rest of this remainder term can be bounded by writing
\begin{align}
\label{Eq:R7bound2}
    &\bbE  \{R_{(\bS,S),7}^{(m+1)}(x_S)^2 \} - 2\biggl( \bar{\theta}_{(\bS,S)}^{(m+1)} - \frac{\int_{B_T} \frac{\lambda_S \bar{r}_S}{1+\lambda_S \bar{r}_S} \bar{a}_{\bS}^{(m)} f_S}{\int \frac{\lambda_S \bar{r}_S}{1+\lambda_S \bar{r}_S} f_S}  \biggr)^2 \nonumber \\
    &= \bbE \biggl\{ \biggl| \frac{\sum_{y \in \mathcal{D}_{1,m+1}} \frac{\lambda_S \bar{r}_S(y)}{1+\lambda_S \bar{r}_S(y)} \bar{a}_\bS^{(m)}(y)\mathbbm{1}_{B_T}(y)}{\sum_{y \in \mathcal{D}_{1,m+1}} \frac{\lambda_S \bar{r}_S(y)}{1+\lambda_S \bar{r}_S(y)}} - \bar{\theta}_{(\bS,S)}^{(m+1)} \biggr|^2 \biggr\} - 2\biggl( \bar{\theta}_{(\bS,S)}^{(m+1)} - \frac{\int_{B_T} \frac{\lambda_S \bar{r}_S}{1+\lambda_S \bar{r}_S} \bar{a}_{\bS}^{(m)} f_S}{\int \frac{\lambda_S \bar{r}_S}{1+\lambda_S \bar{r}_S} f_S}  \biggr)^2  \nonumber \\
    & \leq  2\bbE \biggl\{ \biggl| \frac{\sum_{y \in \mathcal{D}_{1,m+1}} \frac{\lambda_S \bar{r}_S(y)}{1+\lambda_S \bar{r}_S(y)} \bar{a}_\bS^{(m)}(y)\mathbbm{1}_{B_T}(y)}{\sum_{y \in \mathcal{D}_{1,m+1}} \frac{\lambda_S \bar{r}_S(y)}{1+\lambda_S \bar{r}_S(y)}} - \frac{\int_{B_T} \frac{\lambda_S \bar{r}_S}{1+\lambda_S \bar{r}_S} \bar{a}_{\bS}^{(m)} f_S}{\int \frac{\lambda_S \bar{r}_S}{1+\lambda_S \bar{r}_S} f_S} \biggr|^2 \biggr\}  \nonumber \\
    & = 2 \bbE \biggl[ \biggl\{ \frac{ \sum_{y \in \mathcal{D}_{1,m+1}} \frac{\lambda_S \bar{r}_S(y)}{1+\lambda_S \bar{r}_S(y)} \Bigl( \bar{a}_\bS^{(m)}(y)\mathbbm{1}_{B_T}(y) - \frac{\int_{B_T} \frac{\lambda_S \bar{r}_S}{1+\lambda_S \bar{r}_S} \bar{a}_{\bS}^{(m)} f_S}{\int \frac{\lambda_S \bar{r}_S}{1+\lambda_S \bar{r}_S} f_S} \Bigr)}{\sum_{x \in \mathcal{D}_{1,m+1}} \frac{\lambda_S \bar{r}_S(x)}{1+\lambda_S \bar{r}_S(x)}} \biggr\}^2 \biggr] \nonumber \\
    & \leq \frac{2}{(\frac{\lambda_S c}{1+\lambda_S c})^2} \frac{1}{|\mathcal{D}_{1,m+1}|} \mathbb{E} \biggl[ \biggl( \frac{\lambda_S \bar{r}_S(X_S)}{1+\lambda_S \bar{r}_S(X_S)} \biggl\{ \bar{a}_\bS^{(m)}(X) \mathbbm{1}_{B_T}(X) - \frac{\int_{B_T} \frac{\lambda_S \bar{r}_S}{1+\lambda_S \bar{r}_S} \bar{a}_{\bS}^{(m)} f_S}{\int \frac{\lambda_S \bar{r}_S}{1+\lambda_S \bar{r}_S} f_S} \biggr\} \biggr)^2 \biggr] \nonumber \\
    & \leq \frac{2^{2m+1} \|a\|_\infty^2}{|\mathcal{D}_{1,m+1}|} \biggl( \frac{C}{c} \biggr)^2. % \leq \frac{2}{|\mathcal{D}_{1,m}|} \biggl(\frac{C}{c} \biggr)^2 \biggl( 1 + \frac{C}{c} \biggr) \bbE\{a(X)^2\}.
\end{align}
It is an immediate consequence of~\eqref{Eq:R7bound1} and~\eqref{Eq:R7bound2} that
\begin{equation}
\label{Eq:R7bound}
    \bbE  \{R_{(\bS,S),7}^{(m+1)}(x_S)^2 \} \leq \frac{2^{2m+1} \|a\|_\infty^2}{|\mathcal{D}_{1,m+1}|} \biggl( \frac{C}{c} \biggr)^2 +2^{2m+1} \frac{C \|a\|_\infty^2}{c} \mathbb{P}(\|X\|_\infty \geq T).
\end{equation}

Having bounded all of the remainder terms we now combine our calculations. It follows from~\eqref{Eq:RegressionErrorDecomp},~\eqref{Eq:R1bound},~\eqref{Eq:R2bound},~\eqref{Eq:R3bound},~\eqref{Eq:R4bound},~\eqref{Eq:R5bound},~\eqref{Eq:R6bound} and~\eqref{Eq:R7bound} that for any $x_S$ we have
\begin{align}
\label{Eq:CombinedR}
    &\bbE \bigl[ \bigl\{ \hat{a}_{(\bS,S)}^{(m+1)}(x_S) - \bar{a}_{(\bS,S)}^{(m+1)}(x_S) \bigr\}^2 \bigr] \nonumber \\
    & \leq 7 \biggl[ \frac{2^{2m}\|a\|_\infty^2}{|\mathcal{D}_{1,m+1}|} \biggl\{ 12 \biggl(\frac{C}{c} \biggr)^2 + \frac{3}{h^{|S|} f_{S,h}(x_S)} \biggr\} + \bigl\{ L_1 + 2(2^{m}-1)( L_2 + L_3 )\|a\|_\infty \bigr\}^2 h^{2\beta_\wedge} \nonumber \\
    & \hspace{50pt} + 2^{2m+1} \|a\|_\infty^2 \frac{C}{c} \mathbb{P}(\|X\|_\infty \geq T) +\frac{2^{2m}\|a\|_\infty^2}{f_{S,h}(x_S)} \bbE \bigl\{ \mathbbm{1}_{B_T^c}(X) K_h^S( x_S - X_S) \bigr\} \biggr] \nonumber \\
    & \hspace{50pt} + 7 \bbE \biggl[ \mathbbm{1}_{B_T}(X) \biggl\{ \frac{C}{c} + \frac{K_h^S(x_S - X_S)}{f_{S,h}(x_S)} \biggr\} \bigl\{ \hat{a}_\bS^{(m)}(X) - \bar{a}_\bS^{(m)}(X) \bigr\}^2 \biggr].
\end{align}
We now aim to integrate this bound over $x_S \in B_T^S$. This requires two preliminary calculations, both of which use similar arguments to those leading up to~(5.1) in the proof of Theorem~5.1 of~\cite{gyorfi2006distribution}. To this end, writing $B_z(r)=\{x \in \mathbb{R}^S : \|x-z\|_\infty \leq r\}$, let $z_1,\ldots,z_N \in B_T^S$ be such that $B_T^S \subseteq \cup_{j=1}^N B_{z_j}(h/2)$ and such that $N \leq (4T/h)^{|S|}$. First, recalling that $f_{S,h}(x_S) = (2h)^{-|S|} \int_{B_{x_S}(h)} f_S$, we have
\begin{align}
\label{Eq:Partition1}
    \mathbb{E}\biggl\{ \frac{\mathbbm{1}_{B_T}(X)}{f_{S,h}(X_S)} \biggr\} &\leq \sum_{j=1}^N \int_{B_{z_j}(h/2)} \frac{f_S(x_S)}{f_{S,h}(x_S)} \,dx_S \nonumber \\
    &\leq \sum_{j=1}^N \int_{B_{z_j}(h/2)} \frac{f_S(x_S)}{(2h)^{-|S|} \int_{B_{z_j}(h/2)} f_S} \,dx_S = (2h)^{|S|} N \leq (8T)^{|S|}.
\end{align}
Similarly, we can partition $B_{x_S}(h)$ into $2^{|S|}$ balls of radius $h/2$ to see that
\begin{equation}
\label{Eq:Partition2}
    \mathbb{E}\biggl\{ \frac{K_h^S(X_S - x_S)}{f_{S,h}(X_S)} \biggr\} = \int_{B_{x_S}(h)} \frac{f_S(x_S')}{\int_{B_{x_S'}(h)}f_S} \,dx_S' \leq 2^{|S|}.
\end{equation}
We recall that we write $\hat{a}^{(0)} \equiv a$ and note that the arguments leading to~\eqref{Eq:CombinedR} go through for the $m=0$ case by taking $\bar{a}^{(0)} \equiv a$. Using the bound $|\mathcal{D}_{1,m+1}| \geq n/(4M)$ it now follows from~\eqref{Eq:CombinedR},~\eqref{Eq:Partition1}~\eqref{Eq:Partition2} and inductive reasoning that
\begin{align*}
    &\bbE \bigl[ \mathbbm{1}_{B_T}(X) \bigl\{ \hat{a}_{(\bS,S)}^{(m+1)}(X) - \bar{a}_{(\bS,S)}^{(m+1)}(X) \bigr\}^2 \bigr]  \nonumber \\
    & \leq 7 \biggl[ \frac{2^{2m+2} M\|a\|_\infty^2}{n} \biggl\{ 12 \biggl(\frac{C}{c} \biggr)^2 + \frac{3 (8T)^{d}}{h^{d}} \biggr\} + \bigl\{ L_1 + 2(2^{m}-1)( L_2 +L_3)\|a\|_\infty \bigr\}^2 h^{2\beta_\wedge} \nonumber \\
    & \hspace{50pt} + 2^{2m+1} \|a\|_\infty^2 \biggl(\frac{C}{c} + 2^{d-1} \biggr) \mathbb{P}(\|X\|_\infty \geq T) \biggr] \nonumber \\
    & \hspace{50pt} + 7\biggl( \frac{C}{c} + 2^{d} \biggr) \bbE \bigl[ \mathbbm{1}_{B_T}(X)  \bigl\{ \hat{a}_\bS^{(m)}(X) - \bar{a}_\bS^{(m)}(X) \bigr\}^2 \bigr] \nonumber \\
    & \leq 7 \sum_{j=0}^m \biggl\{ 7\biggl( \frac{C}{c} + 2^{d} \biggr) \biggr\}^{m-j} \biggl[ \frac{2^{2j+2}M\|a\|_\infty^2}{n} \biggl\{ 12 \biggl(\frac{C}{c} \biggr)^2 + \frac{3 (8T)^{d}}{h^{d}} \biggr\}  \nonumber \\
    & \hspace{50pt} + \bigl\{ L_1 + 2^{j+1}( L_2 + L_3 )\|a\|_\infty \bigr\}^2 h^{2\beta_\wedge} + 2^{2j+1} \|a\|_\infty^2 \biggl(\frac{C}{c} + 2^{d-1} \biggr) \mathbb{P}(\|X\|_\infty \geq T) \biggr] \nonumber \\ 
    & \leq 14 \max\bigl( \|a\|_\infty^2 ,1 \bigr) \biggl[ \frac{2M}{n} \biggl\{ 12 \biggl(\frac{C}{c} \biggr)^2 + \frac{3 (8T)^{d}}{h^{d}} \biggr\}  + 2( L_1 +  L_2 + L_3)^2 h^{2\beta_\wedge} \nonumber \\
    & \hspace{50pt} +  \biggl(\frac{C}{c} + 2^{d-1} \biggr) \mathbb{P}(\|X\|_\infty \geq T) \biggr] \sum_{j=0}^m 2^{2j} \biggl\{ 7\biggl( \frac{C}{c} + 2^{d} \biggr) \biggr\}^{m-j} \nonumber \\
    & \leq 14 \biggl\{ 7 \biggl( \frac{C}{c} + 2^d \biggr) \biggr\}^{m+1} \max\bigl( \|a\|_\infty^2 ,1 \bigr) \max\{ (C/c)^2, (L_1+L_2+L_3)^2,8^d\} \nonumber \\
    & \hspace{50pt} \times \biggl[ \frac{M}{n} \Bigl( \frac{T}{h} \Bigr)^d  + h^{2\beta_\wedge} + \mathbb{P}(\|X\|_\infty \geq T) \biggr].
\end{align*}
Substituting this bound and~\eqref{Eq:BTSComplement} into~\eqref{Eq:AlphaHatDecomp} and leads to
\begin{align*}
    & \bbE\int (\hat{\alpha}_{S,(1)}^{(M)} - \bar{\alpha}_S^{(M)})^2 f_S \\
    &\leq 2^{2M+4} 7^{M} \biggl( \frac{C}{c} + 2^d \biggr)^{M} \max\bigl( \|a\|_\infty^2 ,1 \bigr) \max\{ (C/c)^2, (L_1+L_2+L_3)^2,8^d\} \nonumber \\
    & \hspace{50pt} \times \biggl[ \frac{M}{n} \Bigl( \frac{T}{h} \Bigr)^d  + h^{2\beta_\wedge} + \mathbb{P}(\|X\|_\infty \geq T) \biggr] 
\end{align*}
and the result follows.
\end{proof}

\begin{proof}[Proof of Theorem~\ref{Thm:UpperBoundShifted}]
We begin by writing
\begin{align*}
    &n\bbE\{ (\hat{\theta}-\theta)^2 \} - \mathcal{L}(\alpha_\bbS^*) \\
    &= n\mathrm{Var}(\theta^{*,(M)} ) - \mathcal{L}(\alpha_\bbS^*) + 2n \bbE \bigl\{ ( \hat{\theta} - \theta^{*,(M)} )(\theta^{*,(M)} - \theta) \bigr\} + n\bbE\bigl\{(\hat{\theta} - \theta^{*,(M)})^2 \bigr\}  \\
    & \leq n\mathrm{Var}(\theta^{*,(M)} ) -\mathcal{L}(\alpha_\bbS^*) + 2n \mathrm{Var}^{1/2}(\theta^{*,(M)} ) \bbE^{1/2} \bigl\{ ( \hat{\theta} - \theta^{*,(M)} )^2 \bigr\} + n\bbE\bigl\{(\hat{\theta} - \theta^{*,(M)})^2 \bigr\}  \\
    & \leq 2 \mathcal{L}^{1/2}(\alpha_\bbS^*) \bbE^{1/2} \bigl\{ n( \hat{\theta} - \theta^{*,(M)} )^2 \bigr\} + \biggl[ \bigl\{n\mathrm{Var}(\theta^{*,(M)} ) -\mathcal{L}(\alpha_\bbS^*)\bigr\}^{1/2} + \bbE^{1/2} \bigl\{ n( \hat{\theta} - \theta^{*,(M)} )^2 \bigr\}  \biggr]^2 \\
    & \leq 2 \mathrm{Var}^{1/2}\{a(X)\} \bbE^{1/2} \bigl\{ n( \hat{\theta} - \theta^{*,(M)} )^2 \bigr\} + 2\bbE \bigl\{ n( \hat{\theta} - \theta^{*,(M)} )^2 \bigr\} + 2 \bigl\{n\mathrm{Var}(\theta^{*,(M)} ) -\mathcal{L}(\alpha_\bbS^*)\bigr\}.
\end{align*}
Now the first two terms can be bounded using Propositions~\ref{Prop:OptimisationShifted},~\ref{Prop:CrossFitReduction} and~\ref{Prop:InductionProof} while the final term is bounded immediately by Propositions~\ref{Prop:OptimisationShifted} and~\ref{Prop:CrossFitReduction}. Indeed, using the shorthand $n_{\min} = \min_{S \in \bbS} n_S$, we have
\begin{align}
\label{Eq:OracleApprox}
    \bbE &\bigl\{ n( \hat{\theta} - \theta^{*,(M)} )^2 \bigr\} \leq 8C|\bbS| \biggl[ \max_{\ell=1,2} \sum_{S \in \bbS} \bbE \biggl\{ \int 
\frac{n+n_S \bar{r}_S}{n_S \bar{r}_S} (\hat{\alpha}_{S,(\ell)}^{(M)} - \alpha_S^{(M)})^2 f_S \biggr\} + \mathrm{Var}(\theta^{*,(M)}) \biggr] \nonumber \\
    & \leq \frac{8C|\bbS|^2(n+cn_{\min} )}{cn_{\min} } \max_{\substack{\ell=1,2 \\ S \in \bbS}} \bbE \biggl\{ \int  (\hat{\alpha}_{S,(\ell)}^{(M)} - \alpha_S^{(M)})^2 f_S \biggr\} + \frac{8C|\bbS|}{n} \biggl(1 + \max_{S \in \bbS} \biggl| \frac{n\lambda_S}{n_S} -1\biggr| \biggr) \mathcal{L}(\alpha_\bbS^{(M)}) \nonumber \\
    & \leq \frac{8C|\bbS|^2(n+cn_{\min} )}{cn_{\min} } A^{M} B  \biggl\{ \frac{M}{n} \Bigl( \frac{T}{h} \Bigr)^d + h^{2\beta_\wedge} + \mathbb{P}(\|X\|_\infty \geq T) \biggr\} + \frac{16C|\bbS| \mathrm{Var}\{a(X)\}}{n}
\end{align}
where the first and second inequalities follow from Proposition~\ref{Prop:CrossFitReduction} and the third inequality follows from Propositions~\ref{Prop:OptimisationShifted} and~\ref{Prop:InductionProof} and our technical assumption that $|n\lambda_S /n_S - 1| \leq 1/n$. It is straightforward from Propositions~\ref{Prop:OptimisationShifted} and~\ref{Prop:CrossFitReduction} that
\begin{align*}
    n\mathrm{Var}(\theta^{*,(M)} ) -\mathcal{L}(\alpha_\bbS^*) &\leq (1 + 1/n ) \mathcal{L}(\alpha_\bbS^{(M)}) - \mathcal{L}(\alpha_\bbS^*) \\
    &\leq (1+1/n) \bigl\{ \mathcal{L}(\alpha_\bbS^*) + \kappa(1-1/\kappa)^M \mathrm{Var} \, a(X) \bigr\} -\mathcal{L}(\alpha_\bbS^*) \\
    &\leq \bigl\{2 \kappa(1-1/\kappa)^M + 1/n \bigr\} \mathrm{Var} \, a(X).
\end{align*}
Recalling the values of $A$ and $B$ from the statement of Proposition~\ref{Prop:InductionProof}, straightforward but tedious calculations show that
\begin{align*}
    & n\bbE\{ (\hat{\theta}-\theta)^2 \} - \mathcal{L}(\alpha_\bbS^*) \\
    & \leq 200 |\bbS|^2 (C/c) (1+\lambda_\mathrm{min}^{-1}) A^M B \biggl\{ \frac{M T^d}{nh^d}+ h^{2\beta_\wedge} + \mathbb{P}(\|X\|_\infty \geq T) \biggr\}^{1/2} \!\!\!\!\! + 4\kappa \exp(-M/\kappa) \mathrm{Var} \, a(X)
\end{align*}
and the result follows.
\end{proof}

\begin{proof}[Proof of Proposition~\ref{Prop:Confidence}]
The first component of the proof is in establishing the consistency of the variance estimator $\hat{V}$. It is sufficient to prove the consistency of the individual estimators $\hat{V}^{(\ell)}$ and $\hat{V}_S^{(\ell)}$ for $S \in \bbS$ and $\ell=1,2$. It will be convenient to introduce the oracle estimators
\[
    V^{*,(\ell)} = \frac{1}{|\mathcal{D}_{3-\ell}|} \sum_{x \in \mathcal{D}_{3-\ell}} \biggl\{ a(x) - \sum_{S \in \bbS} \alpha_{S}^{(M)}(x) \biggr\}^2 - \theta^2 \quad \text{and} \quad V_S^{*,(\ell)} = \frac{n}{|\mathcal{D}_{3-\ell}|} \sum_{x \in \mathcal{D}_{3-\ell}} \frac{\alpha_{S}^{(M)}(x)^2}{n_S \bar{r}_S(x)},
\]
to which we will compare our data-driven estimators. By symmetry we may restrict attention to the $\ell=1$ estimators. First, we see that
\begin{align}
\label{Eq:VarianceCons1}
    \bbE | &\hat{V}^{(1)} - V^{*,(1)}| \leq \bbE \biggl|  \biggl\{ a(X_n) - \sum_{S \in \bbS} \hat{\alpha}_{S,(1)}^{(M)}(X_n) \biggr\}^2 -  \biggl\{ a(X_n) - \sum_{S \in \bbS} \alpha_{S}^{(M)}(X_n) \biggr\}^2 + \theta^2 - (\hat{\theta})^2 \biggr| \nonumber \\
    & = \bbE \biggl| \biggl\{ 2a(X_n) - \sum_{S \in \bbS} \bigl(\hat{\alpha}_{S,(1)}^{(M)}(X_n) + \alpha_{S}^{(M)}(X_n) \bigr) \biggr\} \sum_{S \in \bbS} \bigl\{ \hat{\alpha}_{S,(1)}^{(M)}(X_n) - \alpha_{S}^{(M)}(X_n) \bigr\}  \nonumber \\
    & \hspace{325pt} + (\hat{\theta} - \theta)(\hat{\theta}+\theta) \biggr| \nonumber \\
    & \leq \bbE \biggl[ 2\biggl|\biggl\{ a(X_n) - \sum_{S \in \bbS}\alpha_{S}^{(M)}(X_n) \biggr\} \sum_{S \in \bbS}\{ \hat{\alpha}_{S,(1)}^{(M)}(X_n) - \alpha_{S}^{(M)}(X_n) \} \biggr| + 2\theta|\hat{\theta}-\theta| \nonumber \\
    & \hspace{150pt} + \biggl\{ \sum_{S \in \bbS}\{ \hat{\alpha}_{S,(1)}^{(M)}(X_n) - \alpha_{S}^{(M)}(X_n) \} \biggr\}^2 + (\hat{\theta}-\theta)^2   \biggr] \nonumber \\
    & \leq 2 \mathcal{L}^{1/2}(\alpha_\bbS^{(M)}) \sum_{S \in \bbS} \bbE^{1/2} [ \{\hat{\alpha}_{S,(1)}^{(M)}(X_n) - \alpha_{S}^{(M)}(X_n)\}^2 ] + 2|\theta| \bbE|\hat{\theta} - \theta|\nonumber \\
    & \hspace{150pt} + |\bbS| \sum_{S \in \bbS} \bbE [ \{\hat{\alpha}_{S,(1)}^{(M)}(X_n) - \alpha_{S}^{(M)}(X_n)\}^2 ] + \bbE\{ (\hat{\theta}-\theta)^2\},
\end{align}
where each of these terms will be bounded later by our previous work. Now, for each $S \in \bbS$ we have
\begin{align}
\label{Eq:VarianceCons2}
    \frac{n_S}{n} \bbE | \hat{V}_S^{(1)} &- V_S^{*,(1)} | \leq  \bbE\biggl[ \frac{1}{\bar{r}_S(X_n)} \biggl| \frac{\hat{r}_{S,(1)}}{\bbE{\{r_S(X)\}}} \hat{\alpha}_{S,(1)}^{(M)}(X_n)^2 - \alpha_S^{(M)}(X_n)^2 \biggr| \biggr] \nonumber \\
    & =  \bbE\biggl| \frac{\hat{r}_{S,(1)}}{r_S(X_n)} \{ \hat{\alpha}_{S,(1)}^{(M)}(X_n)^2 - \alpha_S^{(M)}(X_n)^2\} + \frac{\alpha_S^{(M)}(X_n)^2}{\bar{r}_S(X_n)}\biggl\{\frac{\hat{r}_{S,(1)}}{\bbE\{r_S(X)\}} - 1 \biggr\}  \biggr| \nonumber \\
    & \leq C \bbE \biggl[ \frac{1}{\bar{r}_S(X_n)} \bigl|  \hat{\alpha}_{S,(1)}^{(M)}(X_n)^2 - \alpha_S^{(M)}(X_n)^2 \bigr| \biggr] + \bbE\biggl[ \frac{\alpha_S^{(M)}(X)^2}{\bar{r}_S(X)}\biggr] \bbE \biggl| \frac{\hat{r}_{S,(1)}}{\bbE\{r_S(X)\}} - 1 \biggr| \nonumber \\
    & \leq C \bbE \biggl[ \frac{2|\alpha_S^{(M)}(X_n)|}{\bar{r}_S(X_n)}\bigl|  \hat{\alpha}_{S,(1)}^{(M)}(X_n) - \alpha_S^{(M)}(X_n) \bigr| +\bigl\{  \hat{\alpha}_{S,(1)}^{(M)}(X_n) - \alpha_S^{(M)}(X_n) \bigr\}^2 \biggr] \nonumber \\
    & \hspace{200pt} + \bbE\biggl[ \frac{\alpha_S^{(M)}(X)^2}{\bar{r}_S(X)}\biggr] |\mathcal{D}_1|^{-1/2} \mathrm{Var}^{1/2} \{ \bar{r}_S(X) \} \nonumber \\
    & \leq \frac{2C}{c^{1/2}} \bbE^{1/2} \biggl[ \frac{\alpha_S^{(M)}(X)^2}{\bar{r}_S(X)}\biggr] \bbE^{1/2}\bigl[ \bigl\{  \hat{\alpha}_{S,(1)}^{(M)}(X_n) - \alpha_S^{(M)}(X_n) \bigr\}^2 \bigr] + \frac{C^{1/2}}{|\mathcal{D}_1|^{1/2}}\bbE\biggl[ \frac{\alpha_S^{(M)}(X)^2}{\bar{r}_S(X)}\biggr]  \nonumber \\
    & \hspace{200pt} + C \bbE\bigl[ \bigl\{  \hat{\alpha}_{S,(1)}^{(M)}(X_n) - \alpha_S^{(M)}(X_n) \bigr\}^2 \bigr],
\end{align}
where, again, these terms are controlled by our earlier work. Indeed, combining~\eqref{Eq:VarianceCons1} and~\eqref{Eq:VarianceCons2} we have
\begin{align}
\label{Eq:VarianceOracleApprox}
    \bbE &\biggl| \hat{V}^{(1)} + \sum_{S \in \bbS} \hat{V}_S^{(1)} - V^{*,(1)} - \sum_{S \in \bbS} V_S^{*,(1)} \biggr| \nonumber \\
    & \leq 2 \biggl\{ |\bbS| \mathcal{L}^{1/2}(\alpha_\bbS^{(M)}) + \frac{C}{c^{1/2}} \sum_{S \in \bbS} \frac{n}{n_S} \bbE^{1/2} \biggl[ \frac{\alpha_S^{(M)}(X)^2}{\bar{r}_S(X)}\biggr] \biggr\} \max_{S \in \bbS} \bbE^{1/2}\bigl[ \bigl\{  \hat{\alpha}_{S,(1)}^{(M)}(X_n) - \alpha_S^{(M)}(X_n) \bigr\}^2 \bigr] \nonumber \\
    & \hspace{50pt} + |\bbS|\biggl(|\bbS|+\frac{Cn}{\min_{S \in \bbS} n_S} \biggr) \max_{S \in \bbS} \bbE \bigl[ \bigl\{  \hat{\alpha}_{S,(1)}^{(M)}(X_n) - \alpha_S^{(M)}(X_n) \bigr\}^2 \bigr] + 2|\theta| \bbE^{1/2}\{(\hat{\theta}-\theta)^2\} \nonumber \nonumber \\
    & \hspace{50pt} + \bbE\{(\hat{\theta}-\theta)^2\} + \frac{C^{1/2}}{|\mathcal{D}_1|^{1/2}} \sum_{S \in \bbS} \frac{n}{n_S} \bbE\biggl[ \frac{\alpha_S^{(M)}(X)^2}{\bar{r}_S(X)}\biggr]  \nonumber \\
    & \leq 2|\bbS|\biggl(1 + \frac{C}{c^{1/2}} \max_{S \in \bbS} \frac{n\lambda_S^{1/2}}{n_S} \biggr) \mathcal{L}^{1/2}(\alpha_\bbS^{(M)})  \max_{S \in \bbS} \bbE^{1/2}\bigl[ \bigl\{  \hat{\alpha}_{S,(1)}^{(M)}(X_n) - \alpha_S^{(M)}(X_n) \bigr\}^2 \bigr] \nonumber \\
    & \hspace{50pt} + |\bbS|\biggl(|\bbS|+\frac{Cn}{\min_{S \in \bbS} n_S} \biggr) \max_{S \in \bbS} \bbE \bigl[ \bigl\{  \hat{\alpha}_{S,(1)}^{(M)}(X_n) - \alpha_S^{(M)}(X_n) \bigr\}^2 \bigr] + 2|\theta| \bbE^{1/2}\{(\hat{\theta}-\theta)^2\} \nonumber \\
    & \hspace{50pt} + \bbE\{(\hat{\theta}-\theta)^2\} + \frac{C^{1/2}}{|\mathcal{D}_1|^{1/2}} \biggl( \max_{S \in \bbS} \frac{\lambda_S n}{n_S} \biggr) \mathcal{L}(\alpha_\bbS^{(M)})  \nonumber \\
    & \leq \frac{3C|\bbS|\mathrm{Var}^{1/2} a(X)}{c^{1/2} \lambda_\mathrm{min}^{1/2} } \max_{S \in \bbS} \bbE^{1/2} \biggl\{ \int (\hat{\alpha}_{S,(1)}^{(M)} - \alpha_S^{(M)})^2 f_S \biggr\} + \|a\|_\infty^{1/2} \bbE^{1/2}\{(\hat{\theta}-\theta)^2\} \nonumber \\
    & \hspace{50pt} + \frac{4C|\bbS|^2}{\lambda_\mathrm{min}} \max_{S \in \bbS} \bbE \biggl\{ \int (\hat{\alpha}_{S,(1)}^{(M)} - \alpha_S^{(M)})^2 f_S \biggr\} + \bbE\{(\hat{\theta}-\theta)^2\} + \frac{4 C^{1/2} \mathrm{Var} \, a(X)}{n^{1/2}} \nonumber\\
    & \leq 8 A^{M/2} D \biggl\{ \frac{M T^d}{nh^d} + h^{2\beta_\wedge} + \mathbb{P}(\|X\|_\infty \geq T) \biggr\}^{1/2},
\end{align}
where the penultimate inequality follows from Proposition~\ref{Prop:OptimisationShifted} and the final inequality follows from Proposition~\ref{Prop:InductionProof} and Theorem~\ref{Thm:UpperBoundShifted}. 

We have now shown that our variance estimators can be approximated by their oracle counterparts, and the consistency of these estimators is thus established by showing that the oracle estimators $V^{*,(\ell)} + \sum_{S \in \bbS} V_S^{*,(\ell)}$ are consistent. It is clear that they are unbiased estimators of $n\mathrm{Var}(\theta^{*,(M)})$.  Now using~\eqref{Eq:IteratedUniformBound} and an analogous argument to~\eqref{Eq:PGF} we have that 
\begin{align}
\label{Eq:UniformAlphaBound}
    \|\alpha_S^{(M)}\|_\infty &\leq \sum_{m=1}^{M} \sum_{\bS \in \bbS^{(m)} : S_m = S} b_{M,\eta}(m) \|a_{\bS}^{(m)}\|_\infty \leq 2 \|a\|_\infty \sum_{m=1}^{M} (2|\bbS|)^{m-1} b_{M,\eta}(m) \leq 2\|a\|_\infty 3^M.
\end{align}
We therefore have that
\begin{align}
\label{Eq:ConsistentVariance2}
    \mathrm{Var} \biggl( V^{*,(1)} + \sum_{S \in \bbS} V_S^{*,(1)} \biggr) &= \frac{1}{|\mathcal{D}_2|}\mathrm{Var}\biggl( \biggl\{ a(X) - \sum_{S \in \bbS} \alpha_S^{(M)}(X) \biggr\}^2 + \sum_{S \in \bbS} \frac{n}{n_S} \frac{\alpha_S^{(M)}(X)^2}{\bar{r}_S(X)} \biggr) \nonumber \\
    & \leq \frac{\|a\|_\infty^4}{|\mathcal{D}_2|} \biggl\{ \bigl( 1 +  2 |\bbS| 3^{M} \bigr)^2 + \sum_{S \in \bbS} \frac{2n 3^{M}}{c n_S} \biggr\}^2 \nonumber \\
    & \leq \frac{1184 \|a\|_\infty^4 |\bbS|^4 3^{4M}}{n\lambda_\mathrm{min}^2} \leq \frac{3^{4M}D^2}{n}.
\end{align}
Combining~\eqref{Eq:VarianceOracleApprox} and~\eqref{Eq:ConsistentVariance2} with Markov's inequality we thus see that we have
\begin{align}
\label{Eq:VhatL1}
    \bbE \bigl| \hat{V} - n\mathrm{Var}(\theta^{*,(M)}) \bigr| & \leq \bbE \biggl| \hat{V}^{(1)} + \sum_{S \in \bbS} \hat{V}_S^{(1)} - n \mathrm{Var}(\theta^{*,(M)}) \biggr| \nonumber \\
    & \leq 8 A^{M/2} D \biggl\{ \frac{M T^d}{nh^d} + h^{2\beta_\wedge} + \mathbb{P}(\|X\|_\infty \geq T) \biggr\}^{1/2} + \frac{3^{2M}D}{n^{1/2}}   \nonumber \\ 
    & \leq 9 A^{M/2} D \biggl\{ \frac{M T^d}{nh^d} + h^{2\beta_\wedge} + \mathbb{P}(\|X\|_\infty \geq T) \biggr\}^{1/2}
\end{align}
so that our variance estimators are consistent.

With control over $\hat{V}$ we now proceed to prove the approximate normality of the standardised estimator $\hat{\theta}$. To do this we will need to bound the variance of $\theta^{*,(M)}$ below. Using the shorthand $v_M = n\mathrm{Var}(\theta^{*,(M)})$ we have
\begin{align}
\label{Eq:VarianceLowerBound}
    v_M &= \mathrm{Var}\biggl( a(X) - \sum_{S \in \bbS} \alpha_S^{(M)}(X) \biggr) + \sum_{S \in \bbS} \frac{n}{n_S} \bbE \biggl\{ \frac{\alpha_S^{(M)}(X)^2}{\bar{r}_S(X)} \biggr\} \nonumber \\
    & = \mathcal{L}(\alpha_\bbS^{(M)}) + \sum_{S \in \bbS} \frac{1}{\lambda_S} \biggl( \frac{n \lambda_S}{n_S} - 1 \biggr) \bbE \biggl\{ \frac{\alpha_S^{(M)}(X)^2}{\bar{r}_S(X)} \biggr\} \geq \biggl(1 - \max_{S \in \bbS} \biggl| \frac{n \lambda_S}{n_S} - 1 \biggr| \biggr) \mathcal{L}(\alpha_\bbS^{(M)}) \nonumber \\
    & \geq \frac{n-1}{n} \mathcal{L}(\alpha_\bbS^{(M)}) \geq \frac{n-1}{n} \mathcal{L}(\alpha_\bbS^*) \geq \frac{n-1}{n} V_0.
\end{align}
For $\epsilon \in (0,1/2]$ and $\tau>0$ we have by elementary calculations that
\begin{align*}
    \sup_{x \in \bbR} &\biggl\{ \mathbb{P} \biggl( \frac{n^{1/2}(\hat{\theta}-\theta)}{\hat{V}^{1/2}} \mathbbm{1}_{\{\hat{V}>0\}} \leq x \biggr) - \Phi(x) \biggr\} \nonumber \\
    & \leq \sup_{x \in \bbR} \biggl\{\mathbb{P} \biggl( \frac{n^{1/2}(\hat{\theta}-\theta)}{v_M^{1/2}} \leq x(1+\epsilon) \biggr) - \Phi(x) \biggr\} + \mathbb{P}\biggl( \frac{\hat{V}}{v_M} \geq (1+\epsilon)^2 \biggr) + \mathbb{P}(\hat{V} \leq 0) \nonumber \\
    & \leq \sup_{x \in \bbR} \biggl\{\mathbb{P} \biggl( \frac{n^{1/2}(\theta^{*,(M)}-\theta)}{v_M^{1/2}} \leq x(1+\epsilon) + \tau \biggr) - \Phi(x) \biggr\} + \mathbb{P}(n^{1/2}|\hat{\theta}-\theta^{*,(M)}| \geq \tau v_M^{1/2} ) \nonumber \\
    & \hspace{150pt}+ \mathbb{P}\biggl( \frac{\hat{V}}{v_M} \geq (1+\epsilon)^2 \biggr) + \mathbb{P}(\hat{V} \leq 0) \nonumber \\
    & \leq \sup_{x \in \bbR} \biggl\{\mathbb{P} \biggl( \frac{n^{1/2}(\theta^{*,(M)}-\theta)}{v_M^{1/2}} \leq x \biggr) - \Phi(x) \biggr\} + \frac{\tau}{(2\pi)^{1/2}} + 2\epsilon \nonumber \\
    & \hspace{150pt} + \mathbb{P}(n^{1/2}|\hat{\theta}-\theta^{*,(M)}| \geq \tau v_M^{1/2} ) + \mathbb{P}\bigl( |\hat{V}-v_M| \geq \epsilon v_M \bigr).
\end{align*}
With similar control over the upper tail probabilities, writing $Z \sim N(0,1)$ and taking $\tau=[\bbE\{n(\hat{\theta}-\theta^{*,(M)})^2\}/v_M]^{1/3}$ and $\epsilon=\min\{1/2,\bbE|\hat{V} - v_M|/ v_M\}$, we see by~\eqref{Eq:OracleApprox},~\eqref{Eq:VhatL1} and~\eqref{Eq:VarianceLowerBound} that
\begin{align}
\label{Eq:KolmTerm1}
    d_\mathrm{K}\biggl( \frac{n^{1/2}(\hat{\theta}-\theta)}{\hat{V}^{1/2}} \mathbbm{1}_{\{\hat{V}>0\}},& Z \biggr) - d_\mathrm{K}\biggl( \frac{n^{1/2}(\theta^{*,(M)}-\theta)}{v_M^{1/2}}, Z \biggr) \nonumber \\
    &\leq  2(\tau+\epsilon)+ \frac{\bbE\{n(\hat{\theta}-\theta^{*,(M)})^2\}}{\tau^2 v_M} + 2 \frac{\bbE|\hat{V} - v_M|}{\epsilon v_M} \nonumber  \\
    & \leq 3 \biggl( \frac{\bbE\{n(\hat{\theta}-\theta^{*,(M)})^2\}}{v_M} \biggr)^{1/3} + 2 \biggl( \frac{\bbE|\hat{V} - v_M|}{v_M} \biggr) + 4 \biggl( \frac{\bbE|\hat{V} - v_M|}{v_M} \biggr)^{1/2} \nonumber \\
    & \leq\frac{6 A^{M/2} D}{V_0} \biggl\{ \frac{M T^d}{nh^d} + h^{2\beta_\wedge} + \mathbb{P}(\|X\|_\infty \geq T) \biggr\}^{1/4}.
\end{align}

It now remains to prove that $\theta^{*,(M)}$ is approximately normally distributed. Since this is the sum of independent random variables we can appeal to the Berry--Esseen theorem~\citep[e.g.][]{shevtsova2010improvement}, for which we bound the third absolute moments of the summands. For $i\in[n]$ we have by~\eqref{Eq:UniformAlphaBound} that
\begin{align*}
    \beta_i &:= \bbE \biggl[ \biggl| \frac{1}{(nv_M)^{1/2}} \biggl\{ a(X_i) - \theta - \sum_{S \in \bbS} \alpha_S^{(M)}(X_i) \biggr\} \biggr|^3 \biggr] \\
    & \leq \frac{2(1+3^{M}|\bbS|)\|a\|_\infty}{(nv_M)^{3/2}} \mathrm{Var} \biggl( a(X) - \sum_{S \in \bbS} \alpha_S^{(M)}(X) \biggr).
\end{align*}
Similarly, we have for $S \in \bbS$ and $j \in [n_S]$ that
\begin{align*}
    \beta_{S,j}:= \bbE \biggl[ \biggl| \frac{n^{1/2}}{n_S v_M^{1/2}} \frac{\alpha_S^{(M)}(X_{S,j})}{\bar{r}_S(X_{S,j})} \biggr|^3 \biggr] \leq \biggl( \frac{n}{n_S^2 v_M} \biggr)^{3/2} \frac{2\|a\|_\infty 3^{M}}{c} \mathrm{Var} \biggl( \frac{\alpha_S^{(M)}(X_{S,1})}{\bar{r}_S(X_{S,1})} \biggr).
\end{align*}
Thus, we have
\begin{align*}
    d_\mathrm{K}\biggl( \frac{n^{1/2}(\theta^{*,(M)}-\theta)}{v_M^{1/2}}, Z \biggr) &\leq \sum_{i=1}^n \beta_i + \sum_{S \in \bbS} \sum_{j=1}^{n_S} \beta_{S,j} \\
    & \leq \frac{2\|a\|_\infty}{(nv_M)^{1/2}} \max \biggl\{ 1+3^{M}|\bbS|, \frac{3^{M}}{c} \max_{S \in \bbS} \frac{n}{n_S} \biggr\} \\
    & \leq \frac{4\|a\|_\infty}{(nv_M)^{1/2}} 3^{M}|\bbS| \biggl( 1 + \frac{1}{c \lambda_\mathrm{min}} \biggr).
\end{align*}
It now follows from the above display and~\eqref{Eq:KolmTerm1} that
\begin{align*}
    d_\mathrm{K}\biggl( \frac{n^{1/2}(\hat{\theta}-\theta)}{\hat{V}^{1/2}} \mathbbm{1}_{\{\hat{V}>0\}}, Z \biggr) &\leq \frac{6 A^{M/2} D}{V_0} \biggl\{ \frac{M T^d}{nh^d} + h^{2\beta_\wedge} + \mathbb{P}(\|X\|_\infty \geq T) \biggr\}^{1/4}\\
    & \hspace{150pt}  + \frac{8\|a\|_\infty}{(nV_0)^{1/2}} 3^{M}|\bbS| \biggl( 1 + \frac{1}{c \lambda_\mathrm{min}} \biggr) \\
    & \leq \frac{7 A^{M/2} D}{V_0} \biggl\{ \frac{M T^d}{nh^d} + h^{2\beta_\wedge} + \mathbb{P}(\|X\|_\infty \geq T) \biggr\}^{1/4},
\end{align*}
as claimed.

\end{proof}

\subsection{Proofs for Section~\ref{Sec:Extensions}}

\begin{proof}[Proof of Theorem~\ref{Thm:MARLowerBound}]
The existence and uniqueness of $\alpha_\bbS^*$, the minimiser of $\mathcal{L}(\cdot,M_\bbS)$, is established by using very similar bounds to those used in the proof of Proposition~\ref{Prop:UniqueMinimiser}. Given this minimiser, we may define
\[
    \alpha^*(x) = \frac{1}{M_{[d]}(x_{S_0})} \biggl\{ a(x) - a_{S_0}(x_{S_0}) - \sum_{S \in \bbS} \alpha_S^*(x_S) \biggr\}
\]
and $h^*(x) = a_{S_0}(x_{S_0}) - \theta + \alpha^*(x)$. Later we will use the fact that for $S \in \bbS$ we have
\begin{equation}
\label{Eq:ConditionalStationarity}
    \bbE\{ \alpha^*(X) | X_S\} = \frac{\alpha_S^*(X_S)}{M_S(X_{S_0})}
\end{equation}
almost surely, which is a conditional version of our stationarity condition~\eqref{Eq:Stationarity}. 

Next we establish the local asymptotic normality of the one-dimensional submodel $\{f_t : t \in \bbR\}$, where we use the shorthand $f_t$ for $f_{th^*}$, and write $P_{N,t}$ for the distribution of all of our data under $f_{N^{-1/2}t}$. Our inner product $\langle \cdot, \cdot \rangle_{M_{\bbS}}$ is defined by
\[
    \langle t_1, t_2 \rangle_{M_\bbS} = t_1 t_2 \bbE \biggl[ \{a_{S_0}(X_{S_0}) - \theta\}^2 + \frac{\alpha^*(X)^2}{M_{[d]}(X_{S_0})} + \sum_{S \in \bbS} \frac{\alpha_S^*(X_S)^2}{M_S(X_{S_0})}  \biggr].
\]
Almost exactly as in the proof of Theorem~\ref{Thm:ShiftedLowerBound}, we have control over the behaviour of $f_{t}$ for small values of $t$ and, under $P_{N,0}$, we see that
\begin{align*}
    \log \frac{dP_{N,t}}{dP_{N,0}} &= \sum_{i=1}^N \sum_{S \in \bbS_+} \mathbbm{1}_{\{\Omega_i = S\}} \log \bigl( 1 + N^{-1/2}t (h^*)_S(X_{i,S}) \bigr) + o_p(1) \\
    & = N^{-1/2} t \sum_{i=1}^N \sum_{S \in \bbS_+} \mathbbm{1}_{\{\Omega_i = S\}} (h^*)_S(X_{i,S}) - \frac{t^2}{2} \bbE\biggl\{ \sum_{S \in \bbS_+} \mathbbm{1}_{\{\Omega = S\}}(h^*)_S(X_{S})^2 \biggr\}+ o_p(1).
\end{align*}
Hence, we see that $\log \frac{dP_{N,t}}{dP_{N,0}} \overset{\mathrm{d}}{\rightarrow} \sigma Z - (1/2) \sigma^2$ where $Z \sim N(0,1)$ and, using~\eqref{Eq:ConditionalStationarity}
\begin{align}
\label{Eq:MARAdjoint}
    \sigma^2 &= t^2 \bbE\biggl\{ \sum_{S \in \bbS_+} \mathbbm{1}_{\{\Omega = S\}}(h^*)_S(X_{S})^2 \biggr\} = t^2 \sum_{S \in \bbS_+} \bbE \Bigl[ M_S(X_{S_0}) \bbE\bigl\{ (h^*)_S(X_S)^2 | X_{S_0} \bigr\} \Bigr] \nonumber \\
    &= t^2 \sum_{S \in \bbS_+} \bbE \biggl\{ M_S(X_{S_0}) \biggl[ (h^*)_{S_0}(X_{S_0})^2 + \bigl\{ (h^*)_{S}(X_S) - (h^*)_{S_0}(X_{S_0}) \bigr\}^2 \biggr] \biggr\} \nonumber \\
    & = t^2 \bbE \biggl[ \bigl\{a_{S_0}(X_{S_0})-\theta\bigr\}^2 + \sum_{S \in \bbS_+} M_S(X_{S_0})(\alpha^*)_S(X_S)^2 \biggr] \nonumber \\
    & = t^2 \bbE \biggl[\bigl\{a_{S_0}(X_{S_0})-\theta\bigr\} h^*(X) + M_{[d]}(X_{S_0}) \alpha^*(X)^2 + \sum_{S \in \bbS} \alpha_S^*(X_S) \alpha^*(X) \biggr] \nonumber \\
    & = t^2 \bbE \biggl[\bigl\{a_{S_0}(X_{S_0})-\theta\bigr\} h^*(X) + \alpha^*(X) \bigl\{ a(X) - a_{S_0}(X_{S_0}) \bigr\} \biggr] \nonumber \\
    & = t^2 \bbE \biggl[\bigl\{a_{S_0}(X_{S_0})-\theta\bigr\} h^*(X) + h^*(X) \bigl\{ a(X) - a_{S_0}(X_{S_0}) \bigr\} \biggr] = t^2 \bbE\{ a(X) h^*(X) \}.
\end{align}

The regularity of our sequence of parameters $\theta(f_{N^{-1/2}t})$ follows from almost identical arguments to those used in the proof of Theorem~\ref{Thm:ShiftedLowerBound} and, using the notation of~\cite[][Chapter~3]{van1996weak}, we may write $\dot{\kappa}(t)=t\bbE\{a(X) h^*(X)\}$ and $r_N=N^{1/2}$, where $\dot{\kappa}$ is thought of as a map from $(\bbR, \langle \cdot, \rangle_{M_\bbS})$ to $\bbR$ equipped with the usual Euclidean metric. The calculation~\eqref{Eq:MARAdjoint} then shows that the adjoint $\dot{\kappa}^*$ of $\dot{\kappa}$ is simply the identity. The result now follows on applying~\citet[][Theorem~3.11.5]{van1996weak} and noting that
\begin{align*}
    \| \dot{\kappa}^* &1 \|_{M_\bbS}^2 = \bbE\{a(X) h^*(X)\} \\
    &= \mathrm{Var} \, a_{S_0}(X_{S_0}) + \bbE \biggl[ \frac{1}{M_{[d]}(X_{S_0})} \biggl\{ a(X) - a_{S_0}(X_{S_0}) - \sum_{S \in \bbS} \alpha_S^*(X_S) \biggr\}^2  + \sum_{S \in \bbS} \frac{\alpha_S^*(X_S)^2}{M_S(X_{S_0})}  \biggr]
\end{align*}
which follows from the third line of~\eqref{Eq:MARAdjoint}.

\end{proof}

\begin{proof}[Proof of Proposition~\ref{Prop:MARDoubleRobustness}]
Using the fact that $\bbE\{\alpha_S^*(X_S) | X_{S_0}\} =0$ almost surely, we have that
\begin{align}
\label{Eq:BiasBound}
    \bbE &\bigl| \bbE(\hat{\theta} - \theta^* | \mathcal{D}) \bigr| = \bbE \bigl| \bbE \bigl(\hat{\theta} - a_{S_0}(X_{S_0}) \bigm| \mathcal{D} \bigr) \bigr| \nonumber \\
    & \leq \bbE \biggl| \bbE \biggl[ \hat{a}_{S_0}(X) - a_{S_0}(X_{S_0}) + \frac{\mathbbm{1}_{\{\Omega_i=\mathbbm{1}_{[d]}\}}}{\hat{M}_{[d]}(X_{S_0})}\biggl\{ a(X) - \hat{a}_{S_0}(X_{S_0}) - \sum_{S \in \bbS} \hat{\alpha}_S(X_S) \biggr\} \nonumber\\
    & \hspace{250pt} + \sum_{S \in \bbS} \frac{\mathbbm{1}_{\{\Omega_i=\mathbbm{1}_{S}\}}}{\hat{M}_{S}(X_{S_0})} \hat{\alpha}_S(X_S) \biggm| \mathcal{D} \biggr] \biggr| \nonumber\\
    & = \bbE \biggl| \bbE \biggl[ \bigl\{ \hat{a}_{S_0}(X) - a_{S_0}(X_{S_0}) \bigr\} \biggl\{1-\frac{M_{[d]}(X_{S_0})}{\hat{M}_{[d]}(X_{S_0})} \biggr\} \nonumber\\
    & \hspace{100pt} + \sum_{S \in \bbS}\biggl\{ \frac{M_S(X_{S_0})}{\hat{M}_{S}(X_{S_0})} - \frac{M_{[d]}(X_{S_0})}{\hat{M}_{[d]}(X_{S_0})} \biggr\}\bigl\{  \hat{\alpha}_S(X_S) - \alpha_S^*(X_S) \bigr\} \biggm| \mathcal{D} \biggr] \biggr| \nonumber\\
    & \leq (1+2|\bbS|) (a_Nb_N)^{1/2},
\end{align}
where the final inequality follows by the triangle and Cauchy--Schwarz inequalities.

Having bounded the conditional mean of $\hat{\theta}-\theta^*$ we now turn to its conditional absolute deviation. With the notation $\hat{g}=a - \hat{a}_{S_0} - \sum_{S \in \bbS} \hat{\alpha}_S$ and $g^* = a - a_{S_0} - \sum_{S \in \bbS} \alpha_S^*$, we start by defining
\begin{align*}
    L &= \frac{1}{N} \sum_{i=1}^N \biggl[\hat{a}_{S_0}(X_i) - a_{S_0}(X_i) + \frac{\mathbbm{1}_{\{\Omega_i = \mathbbm{1}_{[d]}\}}}{M_{[d]}(X_i)} \biggl\{ \hat{g}(X_i) - g^*(X_i) + g^*(X_i) \biggl( \frac{M_{[d]}(X_i)}{\hat{M}_{[d]}(X_i)} - 1 \biggr) \biggr\} \\
    & \hspace{140pt} + \sum_{S \in \bbS} \frac{\mathbbm{1}_{\{\Omega_i = \mathbbm{1}_{S}\}}}{M_{S}(X_i)} \biggl\{ \hat{\alpha}_S(X_i) - \alpha_S^*(X_i) + \alpha_S^*(X_i) \biggl( \frac{M_{S}(X_i)}{\hat{M}_{S}(X_i)} - 1 \biggr) \biggr\} \biggr]
\end{align*}
and
\begin{align*}
    Q &= \frac{1}{N} \sum_{i=1}^N \biggl[ \frac{\mathbbm{1}_{\{\Omega_i = \mathbbm{1}_{[d]}\}}}{M_{[d]}(X_i)} \bigl\{ \hat{g}(X_i) - g^*(X_i) \bigr\}\biggl\{ \frac{M_{[d]}(X_i)}{\hat{M}_{[d]}(X_i)} - 1 \biggr\} \\
    & \hspace{170pt}+ \sum_{S \in \bbS}  \frac{\mathbbm{1}_{\{\Omega_i = \mathbbm{1}_{S}\}}}{M_{S}(X_i)} \bigl\{ \hat{\alpha}_S(X_i) - \alpha_S^*(X_i)\bigr\} \biggl\{ \frac{M_{S}(X_i)}{\hat{M}_{S}(X_i)} - 1 \biggr\} \biggr]
\end{align*}
so that we have $\hat{\theta} - \theta^* = L+Q$. Similarly to in our bound on the conditional mean above, we have
\begin{equation}
\label{Eq:QBound}
    \bbE \bigl| Q - \bbE(Q | \mathcal{D}) \bigr| \leq 2 \bbE |Q| \leq (1+2|\bbS|) (a_Nb_N)^{1/2}.
\end{equation}
It remains to control
\begin{align}
\label{Eq:LBound}
    N &\bigl\{ \bbE \bigl| L - \bbE(L | \mathcal{D}) \bigr| \bigr\}^2 \leq N \bbE \bigl\{ \mathrm{Var}(L | \mathcal{D} ) \bigr\} \nonumber \\
    & = \bbE \biggl[ \mathrm{Var} \biggl(\hat{a}_{S_0}(X_{S_0}) - a_{S_0}(X_{S_0}) + \frac{\mathbbm{1}_{\{\Omega = \mathbbm{1}_{[d]}\}}}{M_{[d]}(X_{S_0})} \biggl\{ \hat{g}(X) - g^*(X) + g^*(X) \biggl( \frac{M_{[d]}(X_{S_0})}{\hat{M}_{[d]}(X_{S_0})} - 1 \biggr) \biggr\} \nonumber \\
    & \hspace{50pt} + \sum_{S \in \bbS} \frac{\mathbbm{1}_{\{\Omega = \mathbbm{1}_{S}\}}}{M_{S}(X_{S_0})} \biggl\{ \hat{\alpha}_S(X_S) - \alpha_S^*(X_S) + \alpha_S^*(X_S) \biggl( \frac{M_{S}(X_{S_0})}{\hat{M}_{S}(X_{S_0})} - 1 \biggr) \biggr\} \biggm| \mathcal{D}, \mathbf{X}_{S_0} \biggr) \biggr] \nonumber \\
    & \leq (2+5|\bbS|) \bbE \biggl[ \biggl\{1 + \frac{1}{M_{[d]}(X_{S_0})}\biggr\} \bigl\{\hat{a}_{S_0}(X_{S_0}) - a_{S_0}(X_{S_0})\bigr\}^2 + \frac{g^*(X)^2}{M_{[d]}(X_{S_0})} \biggl\{ \frac{M_{[d]}(X_{S_0})}{\hat{M}_{[d]}(X_{S_0})} - 1 \biggr\}^2  \nonumber \\
    & \hspace{5pt}  + \sum_{S \in \bbS} \biggl\{ \frac{1}{M_{[d]}(X_{S_0})} +  \frac{1}{M_{S}(X_{S_0})} \biggr\} \bigl\{ \hat{\alpha}_S(X_S) - \alpha_S^*(X_S) \bigr\}^2 + \sum_{S \in \bbS} \frac{\alpha_S^*(X)^2}{M_{S}(X_{S_0})} \biggl\{ \frac{M_{S}(X_{S_0})}{\hat{M}_{S}(X_{S_0})} - 1 \biggr\}^2 \biggr] \nonumber \\
    & \leq  \frac{2(2+5|\bbS|)(1+|\bbS|) a_N}{\inf_{x_{S_0}} \min_{S \in \bbS_+} M_S(x_{S_0})} + \bbE \biggl[ 
  \biggl\{\frac{g^*(X)^2}{M_{[d]}(X_{S_0})} + \sum_{S \in \bbS} \frac{\alpha_S^*(X)^2}{M_{S}(X_{S_0})} \biggr\} \max_{S \in \bbS_+} \biggl\{ \frac{M_{S}(X_{S_0})}{\hat{M}_{S}(X_{S_0})} - 1 \biggr\}^2 \biggr].
\end{align}
Now, since $\alpha_\bbS^*$ is the minimiser of $\mathcal{L}(\cdot, M_\bbS)$, and since it is minimised for each value of $X_{S_0}$ separately, we have that
\[
    \bbE \biggl\{\frac{g^*(X)^2}{M_{[d]}(X_{S_0})} + \sum_{S \in \bbS} \frac{\alpha_S^*(X)^2}{M_{S}(X_{S_0})} \biggm| X_{S_0} \biggr\} \leq \bbE \biggl[ \frac{\{a(X)-a_{S_0}(X_{S_0})\}^2}{M_{[d]}(X_{S_0})} \biggm| X_{S_0} \biggr],
\]
which is the value that we would get by choosing $\alpha_\bbS(\cdot,X_{S_0})=0$. It therefore follows from~\eqref{Eq:BiasBound},~\eqref{Eq:QBound} and~\eqref{Eq:LBound} that
\begin{align*}
    \bbE| \hat{\theta} - \theta^*| &\leq \bbE \bigl| \bbE(\hat{\theta} - \theta^* | \mathcal{D}, \mathbf{X}_{S_0} ) \bigr| + \bbE \bigl| Q - \bbE(Q | \mathcal{D}, \mathbf{X}_{S_0}) \bigr| + \bbE \bigl| L - \bbE(L | \mathcal{D}, \mathbf{X}_{S_0}) \bigr| \\
    & \leq 2(1+2|\bbS|) (a_Nb_N)^{1/2} + \frac{1}{\inf_{x_{S_0}} \min_{S \in \bbS_+} M_S(x_{S_0})^{1/2}} \biggl\{ \frac{(2+5|\bbS|)a_N^{1/2}}{N^{1/2}} + \frac{\|a\|_\infty b_N^{1/2}}{N^{1/2}} \biggr\},
\end{align*}
as required.

\end{proof}

\begin{proof}[Proof of Theorem~\ref{Thm:UnknownShiftedLowerBound}]
Since $H_\bbS^\perp$ is a closed subspace of $H_\bbS$, the existence and uniqueness of the minimiser $\alpha_\bbS^\perp$ over $H_\bbS^\perp$ follows from almost identical arguments to those given in the proof of Proposition~\ref{Prop:UniqueMinimiser} for minimisation over $H_\bbS$. We begin by establishing an analogue of our stationarity condition~\eqref{Eq:Stationarity} that will be useful later in the proof. For any $S \in \bbS, \epsilon_S \in H_S^\perp$ and $\eta \in \bbR$, we may consider perturbing $\alpha_\bbS^\perp$ by replacing $\alpha_S^\perp$ by $\alpha_S^\perp + \eta \epsilon_S$. By the minimality of $\alpha_\bbS^\perp$ we must have
\begin{align*}
    0 &\leq \mathrm{Var}\bigl( \alpha^\perp(X) - \eta \epsilon_S(X_S) \bigr) - \mathrm{Var}\bigl( \alpha^\perp(X) \bigr) + \bbE \biggl[ \frac{\{\alpha_S^\perp(X_S) + \eta \epsilon_S(X_S)\}^2 - \{\alpha_S^\perp(X_S)\}^2 }{\lambda_S \bar{r}_S(X_S)} \biggr] \\
    &= -2\eta \bbE\biggl[ \epsilon_S(X_S) \biggl\{ \alpha^\perp(X) - \frac{\alpha_S^\perp(X_S)}{\lambda_S \bar{r}_S(X_S)} \biggr\} \biggr] + O(\eta^2) \\
    &= -2\eta \bbE\biggl[ \epsilon_S(X_S) \biggl\{ (\alpha^\perp)_S(X_S) - \frac{\alpha_S^\perp(X_S)}{\lambda_S \bar{r}_S(X_S)} \biggr\} \biggr] + O(\eta^2),
\end{align*}
as $\eta \rightarrow 0$, where we recall our notational convention that $(\alpha^\perp)_S(X_S) = \bbE\{\alpha^\perp(X) | X_S\}$. As $\epsilon_S \in H_S^\perp$ was arbitrary, using the fact that the orthogonal complement of $H_S^\perp$ within $H_S$ is given by the span of $\Psi_S$, we see that there exist $\mu \in \bbR$ and $\mu_S \in \bbR^{k_S}$ such that
\begin{equation}
\label{Eq:StationarityUnknown}
    (\alpha^\perp)_S(X_S) - \frac{\alpha_S^\perp(X_S)}{\lambda_S \bar{r}_S(X_S)} = \mu + \mu_S^T \Psi_S(X_S)
\end{equation}
almost surely.

By almost identical arguments to those applied in the proof of Theorem~\ref{Thm:ShiftedLowerBound}, we have control over the behaviour of $f_{t \alpha^\perp}$ for small values of $t$. We therefore see that, under $P_{n,0,0}$, we have
\begin{align*}
    \log & \frac{dP_{n,t,v_\bbS}}{dP_{n,0,0}}  = \sum_{i=1}^n \log \bigl( 1 + n^{-1/2} t \alpha^\perp(X_i) \bigr) \\
    & \hspace{20pt} + \sum_{S \in \bbS} \sum_{j=1}^{n_S} \log \frac{\exp\bigl(n^{-1/2}v_S^T \Psi_S(X_{S,j})\bigr)\bigl\{1+n^{-1/2}t(\alpha^\perp)_S(X_{S,j})\bigr\}\int e^{w_S^T\Psi_S}f}{\int e^{(w_S+n^{-1/2}v_S)^T\Psi_S}f(1+tn^{-1/2}\alpha^\perp)} +o_p(1) \\
    & = n^{-1/2} \sum_{i=1}^n t \alpha^\perp(X_i) - \frac{t^2}{2} \bbE\{ \alpha^\perp(X)^2\} - \frac{1}{2} \sum_{S \in \bbS} \lambda_S \mathrm{Var} \bigl( v_S^T \Psi_S(X_{S,1}) + t(\alpha^\perp)_S(X_{S,1}) \bigr) \\
    & \hspace{5pt}+ n^{-1/2} \sum_{S \in \bbS} \sum_{j=1}^{n_S} \biggl[ v_S^T \Psi_S(X_{S,j}) + t(\alpha^\perp)_S(X_{S,j}) - \bbE\bigl\{ v_S^T \Psi_S(X_{S,j}) + t(\alpha^\perp)_S(X_{S,j}) \bigr\}\biggr] +o_p(1).
\end{align*}
We therefore see that $\log \frac{dP_{n,t,v_\bbS}}{dP_{n,0,0}}\overset{\mathrm{d}}{\rightarrow} \|(t,v_\bbS)\|_{\Psi_\bbS} Z - (1/2) \|(t,v_\bbS)\|_{\Psi_\bbS}^2$, where $Z \sim N(0,1)$ and where $\|\cdot\|_{\Psi_\bbS}$ is the norm associated to the inner product
\begin{align*}
    &\langle (t,v_\bbS), (s,u_\bbS) \rangle_{\Psi_\bbS} \\
    &= st \bbE\{ \alpha^\perp(X)^2\} + \sum_{S \in \bbS} \lambda_S \mathrm{Cov}\bigl( v_S^T \Psi_S(X_{S,1}) + t(\alpha^\perp)_S(X_{S,1}), u_S^T \Psi_S(X_{S,1}) + s(\alpha^\perp)_S(X_{S,1}) \bigr),
\end{align*}
which establishes the local asymptotic normality of this sequence of experiments.

The regularity of our sequence of parameters $\theta(f_{n^{-1/2}t \alpha^\perp})$ follows from almost identical arguments to those used in the proof of Theorem~\ref{Thm:ShiftedLowerBound} and, using the notation of~\citet[][Chapter~3]{van1996weak}, we may write $\dot{\kappa}(t,v_\bbS)=t\bbE\{a(X) \alpha^\perp(X)\}$ and $r_N=n^{1/2}$, where $\dot{\kappa}$ is thought of as a map from $(\bbR^{1+k_\bbS}, \langle \cdot, \rangle_{\Psi_\bbS})$ to $\bbR$ equipped with the usual Euclidean metric. To find the adjoint $\dot{\kappa}^*$ of $\dot{\kappa}$ we aim to find $(t^*,v_\bbS^*) \in \bbR^{1+k_\bbS}$ such that $\langle(t^*,v_\bbS^*), (t,v_\bbS) \rangle_{\Psi_\bbS} =\dot{\kappa}(t,v_\bbS)$ for all $(t,v_\bbS) \in \bbR^{1+k_\bbS}$. Since we assume that the functions $\Psi_S$ are linearly independent, for each $S \in \bbS$ we may define
\[
    v_S^* = - \mathrm{Cov}\bigl( \Psi_S(X_{S,1}) \bigr)^{-1} \mathrm{Cov} \bigl( \Psi_S(X_{S,1}), (\alpha^\perp)_S(X_{S,1}) \bigr)
\]
so that $\mathrm{Cov}((\alpha^\perp)_S(X_{S,1}) + (v_S^*)^T \Psi_S(X_{S,1}), \Psi_S(X_{S,1}))=0$. Using  our stationarity condition~\eqref{Eq:StationarityUnknown} and the fact that $\alpha_\bbS^\perp \in H_\bbS^\perp$, we have for almost all $x_S$ that
\begin{align*}
    & (\alpha^\perp)_S(x_S) + (v_S^*)^T \Psi_S(x_S) = \frac{\alpha_S^\perp(x_S)}{\lambda_S \bar{r}_S(x_S)} + \mu + \mu_S^T \Psi_S(x_S)  \\
    & \hspace{50pt}- \mathrm{Cov}\bigl( \Psi_S(X_{S,1}) \bigr)^{-1} \mathrm{Cov} \biggl( \Psi_S(X_{S,1}), \frac{\alpha_S^\perp(X_{S,1})}{\lambda_S \bar{r}_S(X_{S,1})} + \mu + \mu_S^T \Psi_S(X_{S,1}) \biggr) \Psi_S(x_S) \\
    & = \frac{\alpha_S^\perp(x_S)}{\lambda_S \bar{r}_S(x_S)} + \mu + \mu_S^T \Psi_S(x_S)  - \frac{1}{\lambda_S} \mathrm{Cov}\bigl( \Psi_S(X_{S,1}) \bigr)^{-1} \bbE\bigl\{ \Psi_S(X_S) \alpha_S^\perp(X_S) \bigr\} - \mu_S^T \Psi_S(x_S) \\
    & = \frac{\alpha_S^\perp(x_S)}{\lambda_S \bar{r}_S(x_S)} + \mu
\end{align*}
Thus, for any $(t,v_\bbS) \in \bbR^{1+k_\bbS}$ we may write that
\begin{align*}
    \langle (1,v_\bbS^*), (t,v_\bbS) \rangle_{\Psi_\bbS} &= t \bbE\{ \alpha^\perp(X)^2\} + \sum_{S \in \bbS} \lambda_S \mathrm{Cov} \bigl( (\alpha^\perp)_S(X_{S,1}) + (v_S^*)^T \Psi_S(X_{S,1}), t (\alpha^\perp)_S(X_{S,1}) \bigr) \\
    &= t \bbE\{ \alpha^\perp(X)^2\} + t \sum_{S \in \bbS} \lambda_S \mathrm{Cov} \biggl( \frac{\alpha_S^\perp(X_{S,1})}{\lambda_S \bar{r}_S(X_{S,1})}, (\alpha^\perp)_S(X_{S,1})\biggr) \\
    & = t \bbE\{ \alpha^\perp(X)^2\} + t \sum_{S \in \bbS} \bbE \bigl\{ \alpha_S^\perp(X_{S}) (\alpha^\perp)_S(X_S) \bigr\} \\
    & = t \bbE\{ \alpha^\perp(X)^2\} + t \sum_{S \in \bbS} \bbE \bigl\{ \alpha_S^\perp(X_{S}) \alpha^\perp(X) \bigr\} = t \bbE\{ \alpha^\perp(X) a(X)\} = \dot{\kappa}(t,v_\bbS),
\end{align*}
as required. This shows that the adjoint is given by $\dot{\kappa}^* : t \mapsto (t,tv_\bbS^*)$. The result now follows on using the above calculation and~\eqref{Eq:StationarityUnknown} to see that 
\begin{align*}
    \|(1,v_\bbS^*)\|_{\Psi_S}^2 &= \dot{\kappa}(1,v_\bbS^*) = \bbE\{ \alpha^\perp(X)^2\} + \sum_{S \in \bbS} \bbE \bigl\{ \alpha_S^\perp(X_{S}) (\alpha^\perp)_S(X_S) \bigr\} \\
    & =\bbE\{ \alpha^\perp(X)^2\} + \sum_{S \in \bbS} \bbE \biggl[\alpha_S^\perp(X_{S}) \biggl\{ \frac{\alpha_S^\perp(X_S)}{\lambda_S \bar{r}_S(X_S)} + \mu - (v_S^*)^T \Psi_S(X_S) \biggr\} \biggr] \\
    & = \bbE\{ \alpha^\perp(X)^2\} + \sum_{S \in \bbS} \bbE \biggl\{ \frac{\alpha_S^\perp(X_S)^2}{\lambda_S \bar{r}_S(X_S)} \biggr\}
\end{align*}
and applying~\citet[][Theorem~3.11.5]{van1996weak}.
\end{proof}

\begin{proof}[Proof of Proposition~\ref{Prop:UnknownDoubleRobustness}]
Since $\theta^*$ is unbiased and independent of $\mathcal{D}$, and since $\alpha_\bbS \in H_\bbS^\perp$, we may control the first moment by writing
\begin{align*}
    \bbE \bigl| \bbE(\hat{\theta} - \theta^* | \mathcal{D}) \bigr| &= \bbE\biggl| \sum_{S \in \bbS}  \bbE \biggl[ - \frac{1}{n} \sum_{i=1}^n \hat{\alpha}_S(X_i) + \frac{1}{n_S} \sum_{j=1}^{n_S} \frac{\hat{\alpha}_S(X_{S,j})}{\hat{r}_S(X_{S,j})}  \biggm| \mathcal{D} \biggr] \biggr| \\
    & = \bbE \biggl| \sum_{S \in \bbS} \bbE\biggl[ \biggl\{ \frac{\bar{r}_S(X_S)}{\hat{r}_S(X_S)} - 1 \biggr\} \hat{\alpha}_S(X_S) \biggr] \biggr| \\
    &=  \bbE \biggl| \sum_{S \in \bbS} \bbE\biggl[ \biggl\{ \frac{\bar{r}_S(X_S)}{\hat{r}_S(X_S)} - 1 \biggr\} \bigl\{ \hat{\alpha}_S(X_S) - \alpha_S^\perp(X_S) \bigr\} \\
    & \hspace{75pt} + \biggl\{ \frac{\bar{r}_S(X_S)}{\hat{r}_S(X_S)} - 1 - w_S^0 + \hat{w}_S^0 - (w_S-\hat{w}_S)^T \Psi_S(X_S) \biggr\} \alpha_S^\perp(X_S) \biggr] \biggr| \\
    & \leq |\bbS|\{ (a_nb_n)^{1/2} + b_n \},
\end{align*}
which is negligible under our hypotheses. It now remains to control the fluctuations of $\hat{\theta}-\theta^*$ about its conditional mean. To this end, we define
\begin{align*}
    L = \sum_{S \in \bbS} \biggl[ \frac{1}{n_S}  \sum_{j=1}^{n_S} \biggl\{ \frac{\hat{\alpha}_S(X_{S,j})-\alpha_S^\perp(X_{S,j})}{\bar{r}_S(X_{S,j})} + \alpha_S^\perp(X_{S,j}) \biggl( \frac{1}{\hat{r}_S(X_{S,j})} &- \frac{1}{\bar{r}_S(X_{S,j})} \biggr) \biggr\} \\
    & - \frac{1}{n} \sum_{i=1}^n \bigl\{ \hat{\alpha}_S(X_i) - \alpha_S^\perp(X_i) \bigr\}  \biggr]
\end{align*}
and
\[
    Q = \sum_{S \in \bbS} \frac{1}{n_S}  \sum_{j=1}^{n_S} \bigl\{ \hat{\alpha}_S(X_{S,j})-\alpha_S^\perp(X_{S,j}) \bigr\}  \biggl( \frac{1}{\hat{r}_S(X_{S,j})} - \frac{1}{\bar{r}_S(X_{S,j})} \biggr)
\]
so that $\hat{\theta}-\theta^* = L+Q$. With similar arguments to those employed for the conditional mean, we see that
\[
    \bbE| Q - \bbE(Q|\mathcal{D})| \leq 2 \bbE |Q| \leq 2 \sum_{S \in \bbS} \bbE \biggl| \bigl\{ \hat{\alpha}_S(X_{S})-\alpha_S^\perp(X_{S}) \bigr\}  \biggl( \frac{\bar{r}_S(X_{S})}{\hat{r}_S(X_{S})} - 1 \biggr) \biggr| \leq 2|\bbS| (a_nb_n)^{1/2}
\]
is negligible. On the other hand, we have that
\begin{align*}
    \bigl( \bbE| &L - \bbE(L|\mathcal{D})| \bigr)^2 \leq \bbE \{ \mathrm{Var}(L | \mathcal{D}) \} \\
    & \leq |\bbS| \sum_{S \in \bbS} \bbE \biggl[ \frac{1}{n_S} \mathrm{Var} \biggl\{ \frac{\hat{\alpha}_S(X_{S,1})-\alpha_S^\perp(X_{S,1})}{\bar{r}_S(X_{S,1})} + \alpha_S^\perp(X_{S,1}) \biggl( \frac{1}{\hat{r}_S(X_{S,1})} - \frac{1}{\bar{r}_S(X_{S,1})} \biggr) \biggm| \mathcal{D} \biggr\} \\
    & \hspace{150pt} + \frac{1}{n} \mathrm{Var} \bigl( \hat{\alpha}_S(X) - \alpha_S^\perp(X) \bigm| \mathcal{D} \bigr) \biggr] \\
    & \leq |\bbS| \sum_{S \in \bbS}  \bbE \biggl[ \biggl\{ \frac{2}{n_S \bar{r}_S(X_S)} + \frac{1}{n} \biggr\} \{\hat{\alpha}_S(X_{S})-\alpha_S^\perp(X_{S})\}^2  + \frac{\alpha_S^\perp(X_S)^2}{n_S \bar{r}_S(X_S)} \biggl\{ \frac{\bar{r}_S(X_S)}{\hat{r}_S(X_S)} - 1 \biggr\}^2 \biggr] \\
    & \leq |\bbS|^2 \biggl\{ \biggl( \frac{2}{n_S} + \frac{1}{n} \biggr) a_n + \frac{b_n}{n_S} \biggr\},
\end{align*}
as required. This completes the proof.
\end{proof}

\subsection{Proofs for Section~\ref{Sec:DirectEstimator}}

\begin{proof}[Proof of Proposition~\ref{Prop:SmoothedUpperBound}]
Write
\begin{align*}
    \check{\theta}^{\bS} &= \frac{1}{(n)_{m+1}} \sum_{\bi \in \mathcal{I}_{m+1}} \hat{k}_{h}^{\bS}(X_{i_1},\ldots,X_{i_{m+1}}) \\
    & =\frac{1}{(n)_{m+1}} \sum_{\bi \in \mathcal{I}_{m+1}} a(X_{i_1})\biggl\{ \frac{K_h^{S_1}(X_{i_2}-X_{i_1})}{\hat{f}_{S_1,h}(X_{i_1})} - 1 \biggr\} \ldots \biggl\{ \frac{K_h^{S_m}(X_{i_{m+1}}-X_{i_m})}{\hat{f}_{S_m,h}(X_{i_m})} - 1 \biggr\}
\end{align*}
so that $\check{\theta}_{h}^{M} = \hat{\theta}^\mathrm{CC} + \sum_{m=1}^M (-1)^m \sum_{\bS \in \bbS^{(m)}} v_{(S_1,\ldots,S_{m-1})}b_{M,\eta}(m)\check{\theta}^\bS$. In this proof we will first consider $\check{\theta}^\bS$ for each $m \in [M]$ and $\bS \in \bbS^{(m)}$ separately, providing progressively simpler approximations, before combining these approximations to prove the result.

The main technical difficulty is in accounting for the random fluctuations of the density estimators $\hat{f}_{S,h}$ about their means. To ease notation we will write $K_{i,j}^S=\frac{K_h^S(X_i-X_j)}{f_{S,h}(X_i)}-1$ for any $i,j \in [n]$ and $S \in \bbS$ and $L_{i,j}^S=\frac{K_h^S(X_i-X_{S,j})}{f_{S,h}(X_i)}-1$ for any $i \in [n]$, $S \in \bbS$ and $j \in [n_S]$. For $i \in [n]$ and $S \in \bbS$, we now write
\begin{align*}
    \epsilon_i^{S} = \frac{\hat{f}_{S,h}(X_{i})}{f_{S,h}(X_{i})} - 1 &= \frac{1}{n+n_S} \biggl[ \sum_{j=1}^n \biggl\{ \frac{K_h^S(X_i-X_j)}{f_{S,h}(X_i)} -1 \biggr\} + \sum_{j=1}^{n_S} \biggl\{ \frac{K_h^S(X_i-X_{S,j})}{f_{S,h}(X_i)} -1 \biggr\} \biggr]\\
    &  = \frac{1}{n+n_S} \biggl( \sum_{j=1}^n K_{i,j}^S + \sum_{j=1}^{n_S} L_{i,j}^S \biggr).
\end{align*}
Our first step in the direction of giving simple approximations to $\check{\theta}^\bS$ is to argue that the denominators $\hat{f}_{S_j}(X_{i_j})=(1+\epsilon_{i_j}^{S_j})f_{S_j,h}(X_{i_j})$ in $\hat{k}_{h}^{\bS}$ can be replaced by $(1-\epsilon_{i_j}^{S_j})^{-1}f_{S_j,h}(X_{i_j})$ up to a negligible error. To this end, for $\ell \in [m+1]$, write
\begin{align*}
    \check{\theta}^{\bS, (\ell)}&=  \frac{1}{(n)_{m+1}} \sum_{\bi \in \mathcal{I}_{m+1}} a(X_{i_1}) \biggl\{ \frac{K_h^{S_1}(X_{i_2}-X_{i_1})}{(1-\epsilon_{i_1}^{S_1})^{-1} f_h^{S_1}(X_{i_1})} -1 \biggr\} \ldots \biggl\{ \frac{K_h^{S_{\ell-1}}(X_{i_{\ell}}-X_{i_{\ell-1}})}{(1-\epsilon_{i_{\ell-1}}^{S_{\ell-1}})^{-1} f_h^{S_{\ell-1}}(X_{i_{\ell-1}})} -1 \biggr\} \\
    & \hspace{140pt} \times \biggl\{ \frac{K_h^{S_\ell}(X_{i_{\ell+1}}-X_{i_\ell})}{(1+\epsilon_{i_\ell}^{S_\ell})f_h^{S_\ell}(X_{i_{\ell}})} - 1 \biggr\} \ldots \biggl\{ \frac{K_h^{S_m}(X_{i_{m+1}}-X_{i_m})}{(1+\epsilon_{i_m}^{S_m})f_h^{S_m}(X_{i_{m}})} - 1 \biggr\} \\
    & = \frac{1}{(n)_{m+1}} \sum_{\bi \in \mathcal{I}_{m+1}} a(X_{i_1}) \bigl\{ K_{i_1,i_2}^{S_1} - (1+K_{i_1,i_2}^{S_1})\epsilon_{i_1}^{S_1} \bigr\} \ldots \bigl\{ K_{i_{\ell-1},i_\ell}^{S_{\ell-1}} - (1+K_{i_{\ell-1},i_\ell}^{S_{\ell-1}})\epsilon_{i_{\ell-1}}^{S_{\ell-1}} \bigr\} \\
    & \hspace{140pt} \times \biggl( \frac{K_{i_\ell,i_{\ell+1}}^{S_\ell} - \epsilon_{i_\ell}^{S_\ell} }{1+\epsilon_{i_\ell}^{S_\ell}} \biggr) \ldots \biggl( \frac{K_{i_m,i_{m+1}}^{S_m} - \epsilon_{i_m}^{S_m} }{1+\epsilon_{i_m}^{S_m}} \biggr) 
\end{align*}
so that $\check{\theta}^{\bS, (1)} = \check{\theta}^\bS$ and each successive $\check{\theta}^{\bS, (\ell+1)}$ is the result of starting with $\check{\theta}^{\bS, (\ell)}$ and making the replacement described above in the next factor in the summands. Now for $\ell \in [m]$ we use the fact that, since $K$ is supported on $[-1/2,1/2]^d$ and symmetric, we have $f_{S,h}(x) \geq c_0/2^d$ to see that
\begin{align}
\label{Eq:ltol+1}
    &\mathbb{E}\{(\check{\theta}^{\bS, (\ell+1)}-\check{\theta}^{\bS, (\ell)})^2\} \nonumber \\
    &=\mathbb{E}\biggl\{ \biggl[ \frac{1}{(n)_{m+1}} \sum_{\bi \in \mathcal{I}_{m+1}} a(X_{i_1}) \bigl\{ K_{i_1,i_2}^{S_1} - (1+K_{i_1,i_2}^{S_1})\epsilon_{i_1}^{S_1} \bigr\} \ldots \bigl\{ K_{i_{\ell-1},i_\ell}^{S_{\ell-1}} - (1+K_{i_{\ell-1},i_\ell}^{S_{\ell-1}})\epsilon_{i_{\ell-1}}^{S_{\ell-1}} \bigr\} \nonumber \\
    &\hspace{70pt} \times \frac{(\epsilon_{i_\ell}^{S_\ell})^2}{1+\epsilon_{i_\ell}^{S_\ell}} \bigl(K_{i_\ell,i_{\ell+1}}^{S_\ell} + 1 \bigr) \biggl( \frac{K_{i_{\ell+1},i_{\ell+2}}^{S_{\ell+1}} - \epsilon_{i_{\ell+1}}^{S_{\ell+1}} }{1+\epsilon_{i_{\ell+1}}^{S_{\ell+1}}} \biggr) \ldots \biggl( \frac{K_{i_m,i_{m+1}}^{S_m} - \epsilon_{i_m}^{S_m} }{1+\epsilon_{i_m}^{S_m}} \biggr)  \biggr]^2 \biggr\} \nonumber \\
    & \leq \mathbb{E}\biggl[ a(X_{1})^2 \bigl\{ K_{i_1,i_2}^{S_1} - (1+K_{i_1,i_2}^{S_1})\epsilon_{i_1}^{S_1} \bigr\}^2 \ldots \bigl\{ K_{i_{\ell-1},i_\ell}^{S_{\ell-1}} - (1+K_{i_{\ell-1},i_\ell}^{S_{\ell-1}})\epsilon_{i_{\ell-1}}^{S_{\ell-1}} \bigr\}^2 \nonumber \\
    &\hspace{70pt} \times \frac{(\epsilon_{\ell}^{S_\ell})^4}{(1+\epsilon_{\ell}^{S_\ell})^2} \bigl(K_{\ell,{\ell+1}}^{S_\ell} + 1 \bigr)^2 \biggl( \frac{K_{{\ell+1},{\ell+2}}^{S_{\ell+1}} - \epsilon_{{\ell+1}}^{S_{\ell+1}} }{1+\epsilon_{{\ell+1}}^{S_{\ell+1}}} \biggr)^2 \ldots \biggl( \frac{K_{m,{m+1}}^{S_m} - \epsilon_{m}^{S_m} }{1+\epsilon_{m}^{S_m}} \biggr)^2  \biggr] \nonumber \\
    & \leq \|a\|_\infty^2 \biggl( \frac{2^d \|K\|_\infty}{c_0 h^d} \biggr)^{2(\ell-1)} \bbE \biggl[ \biggl\{ \prod_{j=1}^{\ell-1} (1+|\epsilon_j^{S_j}|)^2 \biggr\} (\epsilon_{\ell}^{S_\ell})^4  \prod_{j=\ell}^m \max \biggl\{1, \frac{K_h^{S_j}(X_{j+1}-X_j)^2}{\hat{f}_{S_j}(X_j)^2} \biggr\}  \biggr] \nonumber \\
    & \leq \|a\|_\infty^2 \biggl( \frac{2^d \|K\|_\infty}{c_0 h^d} \biggr)^{2(\ell-1)} \max_{j \in [\ell-1]} \bbE^\frac{\ell-1}{m+1}  \bigl\{ (1+|\epsilon_j^{S_j}|)^{2(m+1)} \bigr\}  \times \bbE^{1/(m+1)} \bigl\{ (\epsilon_\ell^{S_\ell})^{4(m+1)} \bigr\} \nonumber \\
    & \hspace{130pt} \times \max_{j=\ell,\ldots,m} \bbE^{(m-\ell+1)/(m+1)} \biggl[ \max \biggl\{1, \frac{K_h^{S_j}(X_{j+1}-X_j)^{2(m+1)}}{\hat{f}_{S_j}(X_j)^{2(m+1)}} \biggr\} \biggr] \nonumber \\
    & \leq \|a\|_\infty^2 \biggl( \frac{2^{d+1} \|K\|_\infty}{c_0 h^d} \biggr)^{2(\ell-1)} \biggl[ 1+\max_{j \in [\ell-1]} \bbE^\frac{\ell-1}{m+1}  \bigl\{|\epsilon_j^{S_j}|^{2(m+1)} \bigr\}\biggr]  \times \bbE^{1/(m+1)} \bigl\{ (\epsilon_\ell^{S_\ell})^{4(m+1)} \bigr\} \nonumber \\
    & \hspace{130pt} \times \max_{j=\ell,\ldots,m} \bbE^{(m-\ell+1)/(m+1)} \biggl[ \max \biggl\{1, \frac{K_h^{S_j}(X_{j+1}-X_j)^{2(m+1)}}{\hat{f}_{S_j}(X_j)^{2(m+1)}} \biggr\} \biggr]
    %& \leq \frac{9^m \|a\|_\infty^2 \|K\|_\infty^{2 \ell}}{h^{2\sum_{j=1}^\ell |S_j|} \{\min_{S \in \bbS} \inf_{x \in [0,1]^d} f_{S,h}(x) \}^{2\ell}} \mathbb{E} \biggl[ (\epsilon_\ell^{S_\ell})^4 \biggl\{ \prod_{j=1}^{\ell-1} \max\{1,(\epsilon_j^{S_j})^2\} \biggr\} \max_{j \in [m]} \biggl\{ \frac{K_h^{S_j}(X_{j+1}-X_j)}{\hat{f}_{S_j}(X_j)}\biggr\}^{2(m-\ell)} \biggr] \\
    %& \leq \frac{9^m m \|a\|_\infty^2 \|K\|_\infty^{2\ell}}{h^{2\ell d} (c_0/2^d)^{2\ell}} \mathbb{E}^{1/3}\{(\epsilon_\ell^{S_\ell})^{12}\} \max_{j \in [m]} \mathbb{E}^{1/3}\biggl[ \biggl\{ \frac{K_h^{S_j}(X_{j+1}-X_j)}{\hat{f}_{S_j}(X_j)}\biggr\}^{6(m-\ell)} \biggr] \max_{j \in [m]} [1 + \mathbb{E}^{1/(3m)}\{(\epsilon_j^{S_j})^{6m}\}]^m.
\end{align}
We now bound each expectation appearing in the final line separately. By Rosenthal's inequality \citep[e.g.][Theorem~15.11]{boucheron2013concentration}, if we write $\kappa=e^{1/2}/(2(e^{1/2}-1))$, we have for any integer $q \geq 2$ that
\begin{align}
\label{Eq:Rosenthal}
    \mathbb{E}^{1/q}(|\epsilon_j^{S_j}|^q)  &\leq 2^{1+1/q}  \biggl\{ (6 \kappa q)^{1/2} \biggl( \frac{2^d \|K\|_\infty}{c_0(n+n_{S_j})h^{|S_j|}} \biggr)^{1/2} + 2q \kappa \biggl( \frac{2^d\|K\|_\infty}{c_0(n+n_{S_j})h^{|S_j|}} \biggr)^{1-1/q} \biggr\} \nonumber \\
    & \hspace{250pt} + \frac{2^{d+1} \|K\|_\infty}{c_0(n+n_{S_j})h^{|S_j|}}\nonumber \\
    & \leq 24q\biggl( \frac{2^d \|K\|_\infty}{c_0(n+n_{S_j})h^{|S_j|}} \biggr)^{1/2}.
\end{align}
Now since $(n+n_{S_j})\hat{f}_{S_j}(X_j) \geq K_h^{S_j}(X_{j+1}-X_j)$ we have by Hoeffding's inequality that
\begin{align}
\label{Eq:Hoeffding}
    \mathbb{E}& \biggl[1 \vee \biggl\{ \frac{K_h^{S_j}(X_{j+1}-X_j)}{\hat{f}_{S_j}(X_j)} \biggr\}^{2(m+1)}  \biggr] \nonumber \\
    & \leq (n + n_{S_j})^{2(m+1)} \mathbb{P}( \hat{f}_{S_j}(X_j) \leq c_0/2^{d+1}) + \Bigl( \frac{2^{d+1}\|K\|_\infty}{c_0 h^{|S_j|}} \Bigr)^{2(m+1)} \nonumber \\
    & \leq (n + n_{S_j})^{2(m+1)} \bbP \biggl( \frac{\hat{f}_{S_j}(X_j)}{f_{S_j,h}(X_j)} \leq \frac{1}{2} \biggr) + \Bigl( \frac{2^{d+1}\|K\|_\infty}{c_0 h^{|S_j|}} \Bigr)^{2(m+1)} \nonumber \\
    & \leq (n + n_{S_j})^{2(m+1)} \bbP \biggl( \sum_{i \neq j} K_{j,i}^{S_j} + \sum_{i=1}^{n_{S_j}} L_{j,i}^{S_j} \leq - \frac{n+n_{S_j}}{2}(1-2^{-5}) \biggr) + \Bigl( \frac{2^{d+1}\|K\|_\infty}{c_0 h^{|S_j|}} \Bigr)^{2(m+1)} \nonumber \\
    & \leq (n + n_{S_j})^{2(m+1)} \exp\biggl( - \frac{c_0^2 h^{2|S_j|} (n+n_{S_j})}{2^{2d+2}\|K\|_\infty^2 } \biggr) + \Bigl( \frac{2^{d+1}\|K\|_\infty}{c_0 h^{|S_j|}} \Bigr)^{2(m+1)} \nonumber \\
    & \leq 2\Bigl( \frac{2^{d+1}\|K\|_\infty}{c_0 h^{|S_j|}} \Bigr)^{2(m+1)},
\end{align}
where the third and final inequalities follow from our assumed lower bound~\eqref{Eq:nlowerbound} on $n$.
%\tb{provided we have the condition} $\frac{c_0^2(n+n_{S_j}-m)h^{2|S_j|}}{2^{2d+2} \|K\|_\infty^2} \geq m \log ( \frac{c_0(n+n_{S_j}-m)h^{|S_j|}}{2^{d+1} \|K\|_\infty} )$.
It now follows from~\eqref{Eq:ltol+1},~\eqref{Eq:Rosenthal} and~\eqref{Eq:Hoeffding} that
\begin{align}
\label{Eq:ThetalToThetal+1}
    \mathbb{E}\{(\check{\theta}^{\bS, (\ell+1)}&-\check{\theta}^{\bS, (\ell)})^2\} \nonumber \\
    & \leq \frac{\{96(m+1)\}^4}{n^2} \|a\|_\infty^2 \biggl( \frac{2^{d+1} \|K\|_\infty}{c_0 h^d} \biggr)^{2m+2} \biggl[ 1 + \biggl\{ 48(m+1) \biggl( \frac{2^{d} \|K\|_\infty}{c_0 n h^d} \biggr)^{1/2} \biggr\}^{2(\ell-1)} \biggr] \nonumber \\
    & \leq \frac{2\{96(m+1)\}^4}{n^2} \|a\|_\infty^2 \biggl( \frac{2^{d+1} \|K\|_\infty}{c_0 h^d} \biggr)^{2m+2}
    %& \leq \frac{9^m m \|a\|_\infty^2 \|K\|_\infty^{2\ell}}{h^{2\ell d} (c_0/2^d)^{2\ell}} 6^4 12^4 2^{7/3} \biggl( \frac{2^d \|K\|_\infty}{c_0nh^d} \biggr)^{2} 2^{1/3} \Bigl( \frac{2^{d+1}\|K\|_\infty}{c_0 h^d} \Bigr)^{2(m-\ell)} \nonumber \\
    %& \hspace{100pt} \times \biggl\{ 1 + 36^2m^2 2^{1/(3m)} \biggl( \frac{2^d \|K\|_\infty}{c_0nh^d} \biggr) \biggr\}^m \nonumber \\
    %& \leq \frac{2^{13}3^8 36^m m \|a\|_\infty^2}{n^2} \Bigl( \frac{2^d \|K\|_\infty}{c_0h^d} \Bigr)^{2m+2} \biggl\{ 1 + 2^5 3^4 m^2 \biggl( \frac{2^d \|K\|_\infty}{c_0nh^d} \biggr) \biggr\}^m \nonumber \\
    %& \leq \frac{2^{13}3^8 e 36^m m \|a\|_\infty^2}{n^2} \Bigl( \frac{2^d \|K\|_\infty}{c_0h^d} \Bigr)^{2m+2},
\end{align}
where the final line again follows from~\eqref{Eq:nlowerbound}. By writing $\check{\theta}^{\bS,(m+1)}- \check{\theta}^\bS = \sum_{\ell=1}^m (\check{\theta}^{\bS, (\ell+1)}-\check{\theta}^{\bS, (\ell)})$ and using the Cauchy--Schwarz inequality we can use~\eqref{Eq:ThetalToThetal+1} to show that we can approximate $\check{\theta}^\bS$ by $\check{\theta}^{\bS,(m+1)}$ with negligible mean-squared error. 

We may write
\[
    \check{\theta}^{\bS,(m+1)} = \frac{1}{(n)_{m+1}} \sum_{\bi \in \mathcal{I}_{m+1}} a(X_{i_1}) \bigl\{K_{i_1,i_2}^{S_1} -(1+K_{i_1,i_2}^{S_1})\epsilon_{i_1}^{S_1} \bigr\} \ldots \bigl\{ K_{i_m,i_{m+1}}^{S_m} - (1+K_{i_m,i_{m+1}}^{S_m})\epsilon_{i_m}^{S_m} \bigr\}.
\]
We now expand these brackets and argue that any term containing two or more $\epsilon_{i_j}^{S_j}$ factors can be neglected. To proceed, we define
\[
    \tilde{\theta}^{\bS} = \frac{1}{(n)_{m+1}} \sum_{\bi \in \mathcal{I}_{m+1}} k_h^{\bS}(X_{i_1},\ldots,X_{i_{m+1}}) = \frac{1}{(n)_{m+1}} \sum_{\bi \in \mathcal{I}_{m+1}} a(X_{i_1}) K_{i_1,i_2}^{S_1} \ldots K_{i_m,i_{m+1}}^{S_m}
\]
and for $\ell \in [m]$ we define
\begin{align*}
 \tilde{\theta}^{\bS,(\ell)} &=  \frac{1}{(n)_{m+1}} \sum_{\bi \in \mathcal{I}_{m+1}} a(X_{i_1}) K_{i_1,i_2}^{S_1} \ldots K_{i_{\ell-1},i_{\ell}}^{S_{\ell-1}} (1+K_{i_\ell,i_{\ell+1}}^{S_\ell})\epsilon_{i_\ell}^{S_\ell}K_{i_{\ell+1},i_{\ell+2}}^{S_{\ell+1}} \ldots K_{i_m,i_{m+1}}^{S_m}.
\end{align*}
Our aim now is to show that $\hat{\theta}^{\bS,(m+1)}$ can be approximated by $\tilde{\theta}^\bS - \sum_{\ell=1}^m \tilde{\theta}^{\bS,(\ell)}$, where there is at most one $\epsilon_{i_j}^{S_j}$ factor per term. Indeed, using~\eqref{Eq:Rosenthal} we have that
\begin{align}
\label{Eq:ThetahattoThetatilde}
    &\mathbb{E} \biggl\{ \biggl( \check{\theta}^{\bS,(m+1)} - \tilde{\theta}^\bS + \sum_{\ell=1}^m \tilde{\theta}^{\bS,(\ell)} \biggr)^2 \biggr\} \nonumber \\
    & \leq 2^{2m} \|a\|_\infty^2 \max_{\alpha \in \{0,1\}^m : \|\alpha\|_1 \geq 2} \mathbb{E} \biggl[ \prod_{\ell=1}^m (K_{\ell,{\ell+1}}^{S_\ell})^{2(1-\alpha_\ell)} \{(1+K_{\ell,{\ell+1}}^{S_\ell})\epsilon_{\ell}^{S_\ell}\}^{2\alpha_\ell} \biggr] \nonumber \\
    & \leq 2^{2m} \|a\|_\infty^2 \max_{\alpha \in \{0,1\}^m : \|\alpha\|_1 \geq 2} \mathbb{E}^{1/2} \biggl\{ \prod_{\ell=1}^m (K_{\ell,{\ell+1}}^{S_\ell})^{4(1-\alpha_\ell)} (1+K_{\ell,{\ell+1}}^{S_\ell})^{4\alpha_\ell} \biggr\} \mathbb{E}^{1/2} \biggl\{ \prod_{\ell=1}^m (\epsilon_{\ell}^{S_\ell})^{4 \alpha_\ell} \biggr\} \nonumber \\
    & \leq 2^{2m} \|a\|_\infty^2 \mathbb{E}^{1/2} \biggl[ \prod_{\ell=1}^m \biggl\{ 1 \vee \frac{K_h^{S_\ell}(X_{{\ell+1}}-X_{\ell})}{f_h^{S_\ell}(X_{{\ell}})} \biggr\}^{4} \biggr]\max_{\alpha \in \{0,1\}^m : \|\alpha\|_1 \geq 2} \prod_{\ell=1}^m \mathbb{E}^{1/(2\|\alpha\|_1)} \biggl[  (\epsilon_{\ell}^{S_\ell})^{4 \|\alpha\|_1 \alpha_\ell} \biggr] \nonumber \\
    & \leq 2^{2m}\|a\|_\infty^2  \biggl( \frac{2^d \|K\|_\infty}{c_0 h^d} \biggr)^{2m} \max_{A=2,\ldots,m} \biggl\{ 96 A \biggl( \frac{2^d \|K\|_\infty}{nc_0h^d} \biggr)^{1/2} \biggr\}^{2 A} \nonumber  \\
    & \leq \frac{2^{2m+20} 3^4 m^4 \|a\|_\infty^2 }{n^2} \biggl( \frac{2^d \|K\|_\infty}{c_0 h^d} \biggr)^{2m+2},
\end{align}
where we use~\eqref{Eq:nlowerbound} for the final inequality.

We now consider the different parts of $\epsilon_{i_\ell}^{S_\ell}$ separately in order to approximate each $\tilde{\theta}^{\bS,(\ell)}$ by one- and two-sample $U$-statistics that can be more easily understood. For a given $\bi \in \mathcal{I}_{m+1}$ we write
\begin{align*}
    \epsilon_{i_\ell}^{S_\ell} &= \frac{1}{n+n_{S_\ell}} \biggl[ \sum_{i \in [n] \setminus \{i_1,\ldots,i_{m+1} \}} K_{i_\ell,i}  +  \sum_{i=1}^{n_{S_\ell}} L_{i_\ell,i} + \sum_{j=1}^{m+1} K_{i_\ell,i_j} \biggr] \\
    & = \epsilon_{i_\ell}^{S_\ell,(1)} + \epsilon_{i_\ell}^{S_\ell,(2)} + \frac{1}{n+n_{S_\ell}} \sum_{j=1}^{m+1} K_{i_\ell,i_j}.
\end{align*}
Using the above definitions we also write
\begin{align*}
    &\tilde{\theta}^{\bS,(\ell),(j)} = \frac{1}{(n)_{m+1}} \sum_{\bi \in \mathcal{I}_{m+1}} a(X_{i_1}) K_{i_1,i_2}^{S_1} \ldots K_{i_{\ell-1},i_{\ell}}^{S_{\ell-1}} (1+K_{i_\ell,i_{\ell+1}}^{S_\ell})\epsilon_{i_\ell}^{S_\ell,(j)}K_{i_{\ell+1},i_{\ell+2}}^{S_{\ell+1}} \ldots K_{i_m,i_{m+1}}^{S_m}
\end{align*}
for $j=1,2$. Then
\begin{align}
\label{Eq:RefineEpsilon}
    &\mathbb{E}\{ (\tilde{\theta}^{\bS,(\ell)}-\tilde{\theta}^{\bS,(\ell),(1)}-\tilde{\theta}^{\bS,(\ell),(2)})^2 \}  \nonumber\\
    & \leq \frac{(m+1) \|a\|_\infty^2}{(n+n_{S_\ell})^2} \sum_{j=1}^{m+1} \bbE \bigl[ (K_{1,2}^{S_1})^2 \ldots (K_{\ell-1,\ell}^{S_{\ell-1}})^2 (1 + K_{\ell,\ell+1}^{S_\ell})^2 (K_{\ell+1,\ell+2}^{S_{\ell+1}})^2 \ldots (K_{m,m+1}^{S_{m}})^2 (K_{\ell,j}^{S_{\ell}})^2 \bigr] \nonumber \\ 
    & \leq \frac{(m+1)^2 \|a\|_\infty^2}{(n+n_{S_\ell})^2} \max_{S \in \bbS} \bbE^{\frac{m}{m+1}}\{(K_{1,2}^S)^{2(m+1)} \} \times \max_{S \in \bbS} \bbE^{\frac{1}{m+1}}\{(1+K_{1,2}^S)^{2(m+1)} \} \nonumber \\
    & \leq \frac{(m+1)^2 \|a\|_\infty^2}{(n+n_{S_\ell})^2} \max_{S \in \bbS} \bbE [\{1 \vee (1+(K_{1,2}^S)\}^{2(m+1)} ] \leq \frac{(m+1)^2 \|a\|_\infty^2}{n^2} \biggl( \frac{2^{d} \|K\|_\infty}{c_0 h^d} \biggr)^{2m+2}.
\end{align}
Now for $\ell \in [m-1]$ we notice that $\tilde{\theta}^{\bS,(\ell),(1)}$ is a degenerate $U$-statistic. Thus, letting $H \sim \mathrm{Hypergeometric}(n,m+2,m+2)$, we have that
\begin{align}
\label{Eq:NeglectThetal1}
    &\mathbb{E}\{ ( \tilde{\theta}^{\bS,(\ell),(1)})^2 \} \nonumber \\
    & \leq \bbP(H \geq 2) \mathbb{E} \bigl[ a(X_{1})^2 (K_{1,2}^{S_1})^2 \ldots (K_{\ell-1,\ell}^{S_{\ell-1}})^2 (1 + K_{\ell,\ell+1}^{S_\ell})^2  (K_{\ell+1,\ell+2}^{S_{\ell+1}})^2 \ldots (K_{m,m+1}^{S_{m}})^2 (K_{\ell,m+2}^{S_{\ell}})^2 \bigr] \nonumber \\
    & \leq \bbP(H \geq 2)\|a\|_\infty^2 \biggl( \frac{2^d \|K\|_\infty}{c_0 h^d} \biggr)^{m+1} = \bbP(H(H-1) \geq 2)\|a\|_\infty^2 \biggl( \frac{2^d \|K\|_\infty}{c_0 h^d} \biggr)^{m+1} \nonumber \\
    & \leq \frac{1}{2} \bbE\{H(H-1)\} \|a\|_\infty^2 \biggl( \frac{2^d \|K\|_\infty}{c_0 h^d} \biggr)^{m+1} = \frac{(m+2)^2(m+1)^2}{2n(n-1)}  \|a\|_\infty^2 \biggl( \frac{2^d \|K\|_\infty}{c_0 h^d} \biggr)^{m+1}.
    %& \leq (m+1) \frac{\binom{m+2}{2}\binom{n-m-2}{m}}{\binom{n}{m+2}} \|a\|_\infty^2 \biggl( \frac{2^d \|K\|_\infty}{c_0 h^d} \biggr)^{m+1} \leq \frac{(m+1)^3(m+2)^2}{n(n-1)} \|a\|_\infty^2 \biggl( \frac{2^d \|K\|_\infty}{c_0 h^d} \biggr)^{m+1}.
\end{align}
On the other hand, when $\ell=m$, we see that
\begin{align*}
    & \tilde{\theta}^{\bS} - \tilde{\theta}^{\bS,(m),(1)} =  \frac{1}{(n)_{m+2}} \sum_{\bi \in \mathcal{I}_{m+2}} k_h^{\bS}(X_{i_1},\ldots,X_{i_{m+1}}) \biggl\{ 1 - \frac{n-m-1}{n+n_{S_m}} (1+K_{i_m,i_{m+2}}^{S_m}) \biggr\} \\
    & = \frac{n_{S_m}+m+1}{n+n_{S_m}} \frac{1}{(n)_{m+1}} \sum_{\bi \in \mathcal{I}_{m+1}} k_h^{\bS}(X_{i_1},\ldots,X_{i_{m+1}}) \\
    & \hspace{50pt} - \frac{n-m-1}{n+n_{S_m}} \frac{1}{(n)_{m+2}} \sum_{\bi \in \mathcal{I}_{m+2}} k_h^{\bS}(X_{i_1},\ldots,X_{i_{m+1}}) K_{i_m,i_{m+2}}^{S_m},
\end{align*}
where the second term in the above line is a degenerate $U$-statistic. Replacing the first term by it linearisation and arguing as in~\eqref{Eq:NeglectThetal1} we therefore have that
\begin{align}
\label{Eq:UStatToIID1}
    &\mathbb{E} \biggl[ \biggl\{ \tilde{\theta}^{\bS} - \tilde{\theta}^{\bS,(m),(1)} - \frac{n_{S_m}+m+1}{n+n_{S_m}} \frac{1}{n} \sum_{i=1}^n \bar{k}_h^{\bS}(X_i) \biggr\}^2 \biggr] \nonumber \\
    & \leq \frac{m^2(m+1)^2}{2n(n-1)} \|a\|_\infty^2 \biggl( \frac{2^d \|K\|_\infty}{c_0 h^d} \biggr)^{m} + \frac{(m+1)^2(m+2)^2}{2n(n-1)} \|a\|_\infty^2 \biggl( \frac{2^d \|K\|_\infty}{c_0 h^d} \biggr)^{m+1}.
\end{align}
We have now seen that $\tilde{\theta}^{\bS} - \sum_{\ell=1}^m \tilde{\theta}^{\bS,(\ell),(1)}$ may be approximated by an average of independent and identically distributed random variables. We turn our attention to the $\tilde{\theta}^{\bS,(\ell),(2)}$ terms. These are now two-sample $U$-statistics rather than standard $U$-statistics, but similar arguments to those above still apply. For $\tilde{\theta}^{\bS,(m),(2)}$ we can use Lemma~A(iii) of \cite{serfling1980approximation} to see that
\begin{align}
\label{Eq:UStatToIID2}
    & \mathbb{E}\biggl[ \biggl\{\tilde{\theta}^{\bS,(m),(2)} - \frac{1}{n+n_{S_m}} \sum_{i=1}^{n_{S_m}} \bar{k}_h^{\bS}(X_{S_m,i}) \biggr\}^2 \biggr] \nonumber \\
    &= \frac{n_{S_m}}{(n+n_{S_m})^2} \mathbb{E}\biggl[\biggl\{\frac{1}{ (n)_{m+1}} \sum_{\bi \in \mathcal{I}_{m+1}} k_h^{\bS}(X_{i_1},\ldots,X_{i_m},X_{S_m,1}) (1+K_{i_m,i_{m+1}}^{S_m}) - \bar{k}_h^{\bS}(X_{S_m,1})  \biggr\}^2 \biggr]\nonumber \\
    & \leq \frac{m+1}{(n+n_{S_m})n}  \mathbb{E} \biggl[ k_h^{\bS}(X_{1},\ldots,X_{m},X_{S_m,1})^2 (1+K_{m,{m+1}}^{S_m})^2 \biggr]\nonumber \\
    & \leq \frac{m+1}{(n+n_{S_m})n} \|a\|_\infty^2 \biggl( \frac{2^d \|K\|_\infty}{c_0 h^d} \biggr)^{m+1}.
\end{align}
On the other hand, when $\ell \in [m-1]$ we can see that $\tilde{\theta}^{\bS,(\ell),(2)}$ is degenerate, so that
\begin{equation}
\label{Eq:NeglectThetal2}
    \mathbb{E} \{ (\tilde{\theta}^{\bS,(\ell),(2)})^2\} \leq \frac{m+1}{(n+n_{S_\ell})n} \|a\|_\infty^2 \biggl( \frac{2^d \|K\|_\infty}{c_0 h^d} \biggr)^{m+1}.
\end{equation}
From~\eqref{Eq:ThetalToThetal+1},~\eqref{Eq:ThetahattoThetatilde},~\eqref{Eq:RefineEpsilon},~\eqref{Eq:NeglectThetal1},~\eqref{Eq:UStatToIID1},~\eqref{Eq:UStatToIID2} and~\eqref{Eq:NeglectThetal2} we have now established that there exists a universal constant $C_1>0$ such that
\begin{align*}
    \mathbb{E} \biggl[ \biggl\{ \check{\theta}^{\bS} - \frac{n_{S_m}+m+1}{n+n_{S_m}} \frac{1}{n} &\sum_{i=1}^n \bar{k}_h^{\bS}(X_i) + \frac{n_{S_m}}{n+n_{S_m}} \frac{1}{n_{S_m}} \sum_{i=1}^{n_{S_m}} \bar{k}_h^{\bS}(X_{S_m,i}) \biggr\}^2 \biggr] \\
    &\leq \frac{\|a\|_\infty^2}{\min(n,\min_{S \in \mathbb{S}} n_S)^2} \biggl( \frac{C_1 2^d \|K\|_\infty}{c_0h^d} \biggr)^{2m+2}.
\end{align*}
It follows from~\eqref{Eq:PGF} that there exists a universal constant $C>0$ such that
\[
    \mathbb{E}\{ (\check{\theta}_{h}^{M} - \theta_{h}^{*,M})^2\} \leq \frac{\|a\|_\infty^2}{\min(n,\min_{S \in \mathbb{S}} n_S)^2} \biggl( \frac{C 2^d \|K\|_\infty}{c_0h^d} \biggr)^{2M+2},
\]
as required.
\end{proof}

\begin{proof}[Proof of Proposition~\ref{Prop:FiniteMUpperBound}]
We first deal with the minor technical point of approximating $n_S/n$ by $\lambda_S$ for $S \in \bbS$. For $\bS \in \bbS^{(m)}$ write $w_\bS=\prod_{j=1}^m \lambda_j/(1+\lambda_j)$. It is straightforward to see that
\[
    |w_\bS - v_\bS| \leq \sum_{j=1}^m \frac{|n\lambda_{S_j}-n_{S_j}|}{(1+\lambda_{S_j})(n+n_{S_j})} \leq \frac{1}{n} \sum_{j=1}^m \frac{n_{S_j}}{(1+\lambda_{S_j})(n+n_{S_j})} \leq \frac{m}{n}.
\]
Using the fact that both $\theta_{h}^{*,M}$ and $\theta^{*,M}$ are averages of independent and identically distributed random variables and using~\eqref{Eq:PGF} we can see therefore that
\begin{align}
\label{Eq:FiniteMStart}
    \mathbb{E}\{( \theta_{h}^{*,M} - \theta^{*,(M)})^2\} &\leq \frac{2^{M+2}}{\min_{S \in \bbS^+} n_S} \max_{m \in [M]} \max_{\bS \in \bbS^{(m)}} \mathbb{E}\bigl[ \{ v_\bS \bar{k}_h^{\bS}(X) - \bar{a}_{\bS}^{(m)}(X) \}^2 \bigr] \nonumber \\
    & = \frac{2^{M+2}}{\min_{S \in \bbS^+} n_S} \max_{m \in [M]} \max_{\bS \in \bbS^{(m)}} \Bigl\{  \mathbb{E}\bigl[ \{ w_\bS \bar{k}_h^{\bS}(X) - \bar{a}_{\bS}^{(m)}(X) \}^2 \bigr] + (v_\bS - w_\bS)^2 \Bigr\} \nonumber \\
    & \leq \frac{2^{M+2}}{\min_{S \in \bbS^+} n_S} \max_{m \in [M]} \max_{\bS \in \bbS^{(m)}} \Bigl\{  \mathbb{E}\bigl[ \{ w_\bS \bar{k}_h^{\bS}(X) - \bar{a}_{\bS}^{(m)}(X) \}^2 \bigr] + \frac{m^2}{n^2} \Bigr\},
\end{align}
and we turn our attention to the error in using $w_\bS \bar{k}_h^\bS(x)$ to approximate $\bar{a}_\bS^{(m)}(x)$. We will write $a_\bS^{(m)}=w_\bS^{-1} \bar{a}_\bS^{(m)}$ so that we have
\[
    a_{(\bS,S)}^{(m+1)}(x_S) = \bbE\{a_\bS^{(m)}(X) | X_S = x_S\}.
\]
First, we will control the smoothness of the functions $a_\bS^{(m)}$. Using the fact that $\|a_\bS^{(m)}\|_\infty \leq 2 \|a\|_\infty$, for any $m \in \mathbb{N}_0$, $\bS \in \bbS^{(m)}$, $S \in \bbS\setminus\{S_m\}$ and $x,x' \in [0,1]^d$ we have
\begin{align}
\label{Eq:SmoothnessToTV}
    &|a_{(\bS,S)}^{(m+1)}(x)-a_{(\bS,S)}^{(m)}(x')| = \bigl| \mathbb{E}\bigl\{ a_{\bS}^{(m)}(X) | X_S = x_S\bigr\} - \mathbb{E}\bigl\{ a_{\bS}^{(m)}(X') | X_S' = x_S'\bigr\} \bigr| \nonumber \\
    &= \biggl| \int \biggl[ \{ a_\bS^{(m)}(x_S,x_{S^c})-a_\bS^{(m)}(x_S',x_{S^c}) \} \frac{f(x_S,x_{S^c})}{f_S(x_S)} + a_\bS^{(m)}(x_S',x_{S^c}) \frac{f(x_S,x_{S^c})-f(x_S',x_{S^c})}{f_S(x_S)} \nonumber \\
    & \hspace{175pt} + a_\bS^{(m)}(x_S',x_{S^c})f(x_S',x_{S^c}) \biggl\{ \frac{1}{f_S(x_S)} - \frac{1}{f_S(x_S')} \biggr\} \biggr] \,dx_{S^c} \biggr| \nonumber \\
    & \leq \sup_{x_{S^c} \in [0,1]^{S^c}} |a_\bS^{(m)}(x_S,x_{S^c})-a_\bS^{(m)}(x_S',x_{S^c})| + 4(L_2/c_0) \|a\|_\infty \|x-x'\|_\infty^\beta.
\end{align}
It follows inductively, using~\eqref{Eq:aSmoothness} for the $m=0$ case, that
\begin{align}
\label{Eq:aSSmoothness}
    | a_\bS^{(m)}(x) - a_\bS^{(m)}(x') | &\leq L_1 \|x-x'\|_\infty^{\beta_1} +  4m (L_2/c_0) \|a\|_\infty \|x-x'\|_\infty^\beta.
\end{align}

For $\delta \in [0,1/2]$ write $\partial_\delta = [0,1]^d \setminus [\delta,1-\delta]^d$ for all those points within distance $\delta$ of the boundary of our sample space. In the following we make repeated use of the fact that for any $m \in \mathbb{N}$, $\bS \in \bbS^{(m)}$, $S \in \bbS\setminus\{S_m\}$ and $x \in [0,1]^d$ we have
\[
    \bar{k}_h^{(\bS,S)}(x) =  \mathbb{E} \biggl[ \bar{k}_h^\bS(X) \biggl\{ \frac{K_h^{S}(x-X)}{f_{S,h}(X)} - 1 \biggr\} \biggr] = \mathbb{E} \biggl[ \bar{k}_h^\bS(X) \frac{K_h^{S}(x-X)}{f_{S,h}(X)} \biggr\} \biggr] .
\]
First, we see inductively that for any $\bS \in \bbS^{(m)}$ we have $\|\bar{k}_h^\bS \|_\infty \leq (2^{d}C_0/c_0)^m \|a\|_\infty$. We will use this simple bound when $x \in \partial_{mh}$, where it is more difficult to control the difference between $\bar{k}_h^\bS$ and $a_\bS^{(m)}$. For $x \in [0,1]^d$ and $S \in \bbS$ define
\[
    \epsilon_S(x)= \mathbb{E} \biggl\{  \frac{K_h^{S}(x-X)}{f_{S,h}(X)} - 1 \biggr\} = \mathbb{E} \biggl\{K_h^{S}(x-X) \biggl( \frac{1}{f_{S,h}(X)} - \frac{1}{f_S(X)} \biggr) \biggr\}.
\]
Using the lower bound $f_{S,h}(x) \geq 2^{-d}c_0$ and the standard bound $|f_{S,h}(x)-f_S(x)| \leq L_2h^{\beta_2}$ when $x \in [h/2,1-h/2]^d$, we have
\begin{equation}
\label{Eq:epsilonSmoothness}
    |\epsilon_S(x)| \leq \frac{2^d}{c_0}(L_2h^{\beta_2} \mathbbm{1}_{\{x \not\in \partial_h \}} + C_0 \mathbbm{1}_{\{x \in \partial_h\}}).
\end{equation}
Our claim, which we prove inductively, is that whenever $m \in \mathbb{N}$, $\bS \in \bbS^{(m)}$ and $x \in [mh,1-mh]^d$ we have
\begin{equation}
\label{Eq:kSSmoothness}
    |\bar{k}_h^\bS(x)-a_\bS^{(m)}(x)| \leq \binom{m+2}{2} \biggl(\frac{2^d C_0}{c_0} \biggr)^{m} (C_0 L_1 h^{\beta_1} + \|a\|_\infty L_2 h^{\beta_2}).
\end{equation}
The $m=1$ case follows from~\eqref{Eq:aSSmoothness} and~\eqref{Eq:epsilonSmoothness} on writing
\begin{align*}
    &|\bar{k}_h^S(x)-a_S^{(1)}(x) | = \biggl| \mathbb{E} \biggl\{ a_S(X) \frac{K_h^S(x-X)}{f_{S,h}(X)} \biggr\} - a_S(x)] \biggr| \\
    & = \biggl| \mathbb{E} \biggl[ \bigl\{ a_S(X) - a_S(x) \bigr\} \frac{K_h^S(x-X)}{f_{S,h}(X)} \biggr] + \epsilon_S(x)a_S(x)] \biggr| \leq \frac{2^d}{c_0}( C_0 L_1h^{\beta_1} + \|a\|_\infty L_2 h^{\beta_2}).
\end{align*}
Now, for the induction step, for any $m \in \mathbb{N}$, $\bS \in \bbS^{(m)}$, $S \in \bbS\setminus\{S_m\}$ and $x \in [(m+1)h,1-(m+1)h]^d$ we have
\begin{align*}
    &\bigl| \bar{k}_h^{(\bS,S)}(x) - a_{(\bS,S)}^{(m+1)}(x) \bigr| = \biggl| \mathbb{E} \biggl[ \bar{k}_h^\bS(X) \biggl\{ \frac{K_h^{S}(x-X)}{f_{S,h}(X)} - 1 \biggr\} \biggr] - a_{(\bS,S)}^{(m+1)}(x)  \biggr| \\
    & = \biggl| \mathbb{E} \biggl[ \bigl\{\bar{k}_h^\bS(X) - a_\bS^{(m)}(X)  \bigr\} \frac{K_h^{S}(x-X)}{f_{S,h}(X)} + \bigl\{a_\bS^{(m)}(X) - a_{(\bS,S)}^{(m+1)}(x) \bigr\} \frac{K_h^{S}(x-X)}{f_{S,h}(X)} \\
    & \hspace{250pt} + a_{(\bS,S)}^{(m+1)}(x) \biggl\{ \frac{K_h^{S}(x-X)}{f_{S,h}(X)} - 1 \biggr\} \biggr] \\
    & = \biggl| \mathbb{E} \biggl[ \bigl\{\bar{k}_h^\bS(X) - a_\bS^{(m)}(X)  \bigr\} \frac{K_h^{S}(x-X)}{f_{S,h}(X)} + \bigl\{a_{(\bS,S)}^{(m+1)}(X) - a_{(\bS,S)}^{(m+1)}(x) \bigr\} \frac{K_h^{S}(x-X)}{f_{S,h}(X)} \biggr] \\
    & \hspace{250pt} + a_{(\bS,S)}^{(m+1)}(x) \epsilon_S(x) \biggr| \\
    & \leq \binom{m+2}{2} \biggl(\frac{2^d C_0}{c_0} \biggr)^{m} (C_0 L_1 h^{\beta_1} + \|a\|_\infty L_2 h^{\beta_2}) \mathbb{E} \biggl\{ \frac{K_h^S(x-X)}{f_{S,h}(X)} \biggr\} \\
    & \hspace{50pt} + \{ L_1 h^{\beta_1} + 4(m+1) \|a\|_\infty (L_2/c_0) h^{\beta_2}\} \mathbb{E} \biggl\{ \frac{K_h^S(x-X)}{f_{S,h}(X)} \biggr\} + 2 \|a\|_\infty \frac{2^d}{c_0} L_2 h^{\beta_2} \\
    & \leq \binom{m+2}{2} \biggl(\frac{2^d C_0}{c_0} \biggr)^{m+1} (C_0 L_1 h^{\beta_1} + \|a\|_\infty L_2 h^{\beta_2}) + \biggl(\frac{2^d C_0}{c_0} \biggr)^{2} \{ L_1 h^{\beta_1} + (m+2)  \|a\|_\infty L_2 h^{\beta_2}\} \\
    & \leq \binom{m+3}{2} \biggl(\frac{2^d C_0}{c_0} \biggr)^{m+1} (C_0 L_1 h^{\beta_1} + \|a\|_\infty L_2 h^{\beta_2}),
\end{align*}
proving the claim~\eqref{Eq:kSSmoothness}. It now follows from~\eqref{Eq:kSSmoothness} and our uniform bound on $\|\bar{k}_h^\bS\|_\infty$ that we have
\begin{align*}
    &\mathbb{E}\bigl[ \{\bar{k}_h^{\bS}(X) - a_{\bS}^{(m)}(X) \}^2 \bigr] \\
    &\leq  \biggl\{ \binom{m+2}{2} \biggl(\frac{2^d C_0}{c_0} \biggr)^{m} (C_0 L_1 h^{\beta_1} + \|a\|_\infty L_2 h^{\beta_2}) \biggr\}^2 + \mathbb{P}(X \in \delta_{mh})\biggl\{\biggl(\frac{2^{d}C_0}{c_0} \biggr)^{m} + 2\biggr\}^2 \|a\|_\infty^2 \\
    & \leq \binom{m+2}{2}^2 \biggl(\frac{2^d C_0}{c_0} \biggr)^{2m} (C_0 L_1 h^{\beta_1} + \|a\|_\infty L_2 h^{\beta_2})^2  + 2dmh C_0 \biggl\{\biggl(\frac{2^{d}C_0}{c_0} \biggr)^{m} + 2\biggr\}^2 \|a\|_\infty^2 \\
    & \leq 2 \binom{m+2}{2}^2 \biggl(\frac{2^{d}C_0}{c_0} \biggr)^{2m} \{C_0^2 L_1^2 h^{2{\beta_1}}+(L_2^2 h^{2\beta_2} + C_0 d h)\|a\|_\infty^2\}.
\end{align*}
The result now follows from~\eqref{Eq:FiniteMStart}.

\end{proof}

\begin{proof}[Proof of Corollary~\ref{Cor:MCAR}]
Using the notation of Section~\ref{Sec:UpperBound}, let $h,T,M$ be chosen such that
\begin{equation}
\label{Eq:Tuning}
    A^{2M}\{M T^d/(Nh^d) + h^{2\beta_\wedge} + \mathbb{P}(\|X\|_\infty \geq T)\} \rightarrow 0
\end{equation}
and $T,M \rightarrow \infty$. Writing $N_{[d]}=N-\sum_{S \in \bbS} N_S$, let $E=\{N_{[d]} \geq 2M, N_S \geq 1 \, \forall S \in \bbS\}$ be the event that our random sample sizes are sufficiently large that our fixed-sample-size estimator can be applied. To emphasise the dependence of this estimator on the sample sizes we write $\hat{\theta} \equiv \hat{\theta}(n_\bbS)$. We take our estimator in the current setting to be
\[
    \hat{\Theta} = \mathbbm{1}_E \hat{\theta}(N_\bbS).
\]
Arguing as in the proof of Theorem~\ref{Thm:UpperBoundShifted}, on the event $E$, we have almost surely that
\begin{align*}
    N \bbE &\bigl[ \bigl\{ \hat{\theta}(N_\bbS) - \theta \bigr\}^2 \bigm| N_\bbS \bigr] - \frac{N}{N_{[d]}} \mathcal{L}(\alpha_\bbS^*) \\
    &\leq \frac{N}{N_{[d]}} \biggl(1 +  \max_{S \in \bbS} \biggl| \frac{N_{[d]} \lambda_S}{N_S} -1 \biggr| \biggr) \biggl[ 4 \kappa \exp(-M/\kappa) \mathrm{Var}\, a(X)  \\
    & \hspace{50pt} + 200|\bbS|^2 (C/c)\biggl(1 + \frac{N}{\min_S N_S} \biggr) A^M B \biggl\{ \frac{MT^d}{N_{[d]}h^d} + h^{2 \beta_\wedge} + \bbP(\|X\|_\infty \geq T) \biggr\}^{1/2} \biggr] \\
    & \leq \max\biggl\{ \biggl(\frac{N}{N_{[d]}}\biggr)^{3/2},1\biggr\} \biggl(1 +  \max_{S \in \bbS} \biggl| \frac{N_{[d]} \lambda_S}{N_S} -1 \biggr| \biggr) \biggl(1 + \frac{N}{\min_S N_S} \biggr) \\
    & \hspace{8pt} \times \biggl[ 4 \kappa \exp(-M/\kappa) \mathrm{Var}\, a(X)+ 200|\bbS|^2 (C/c) A^M B \biggl\{ \frac{MT^d}{Nh^d} + h^{2 \beta_\wedge} + \bbP(\|X\|_\infty \geq T) \biggr\}^{1/2} \biggr],
\end{align*}
where we see that the final factor inside square brackets is deterministic and converges to zero as $N \rightarrow \infty$. Denoting this by $\delta_N$ we may now write
\begin{align*}
    N \bbE \bigl[ &\bigl\{ \hat{\theta}(N_\bbS) - \theta \bigr\}^2 \bigr] - \frac{1}{p_{[d]}} \mathcal{L}(\alpha_\bbS^*) = N \bbE \bigl[ \bigl\{ \hat{\theta}(N_\bbS) - \theta \bigr\}^2 \mathbbm{1}_E \bigr] - \frac{1}{p_{[d]}} \mathcal{L}(\alpha_\bbS^*) + N \theta^2 \bbP(E^c) \\
    & \leq \mathcal{L}(\alpha_\bbS^*) \bbE \biggl\{ \mathbbm{1}_E \biggl( \frac{N}{N_{[d]}} - \frac{1}{p_{[d]}} \biggr) \biggr\} + \|a\|_\infty^2 N \bbP(E^c) \\
    & \hspace{50pt} + \delta_N \bbE \biggl[ \mathbbm{1}_E \max\biggl\{ \biggl(\frac{N}{N_{[d]}}\biggr)^{3/2},1\biggr\} \biggl(1 +  \max_{S \in \bbS} \biggl| \frac{N_{[d]} \lambda_S}{N_S} -1 \biggr| \biggr) \biggl(1 + \frac{N}{\min_S N_S} \biggr) \biggr].
\end{align*}
It now suffices to show that each of these three terms converges to zero. For the first term, we see that $\mathbbm{1}_E N/N_{[d]}$ is uniformly integrable by writing
\[
    \bbE \biggl\{ \mathbbm{1}_E \frac{N}{N_{[d]}} \mathbbm{1}_{\{N \geq K N_{[d]}\}}  \biggr\} \leq N \bbP(N_{[d]}/N \leq 1/K),
\]
and using Hoeffding's inequality to see that this can be made arbitrarily small, uniformly in $N$ by choosing $K$ large enough. Thus, since $N/N_{[d}$ converges in probability to $1/p_{[d]}$ we may conclude that
\[
    \bbE \biggl\{ \mathbbm{1}_E \biggl( \frac{N}{N_{[d]}} - \frac{1}{p_{[d]}} \biggr) \biggr\} \rightarrow 0.
\]
Very similar arguments can be used to show that the third term also converges to zero. Finally, since~\eqref{Eq:Tuning} implies that $M \leq Np_{[d]}/4$ for $N$ sufficiently large, we have that
\begin{align*}
    N \bbP(E^c) &\leq N \sum_{S \in \bbS} \bbP(N_S =0) + N \bbP(N_{[d]} < Np_{[d]}/2) \\
    & \leq N|\bbS| \exp(- N \min_{S \in \bbS} p_S) + N \exp( - Np_{[d]}^2/2) \rightarrow 0,
\end{align*}
where we use Hoeffding's inequality for the final bound.
\end{proof}

\end{document}